\documentclass[10pt, reqno]{amsart}
\numberwithin{equation}{section}
\usepackage{amssymb,latexsym,amsfonts}
\usepackage{hyperref}
\usepackage{bm}

\DeclareFontFamily{OT1}{pzc}{}
\DeclareFontShape{OT1}{pzc}{m}{it}{<-> s*[1.200]pzcmi7t}{}
\DeclareMathAlphabet{\mathscr}{OT1}{pzc}{m}{it}

\newcommand{\R}{\mathbb{R}}

\newcommand{\wh}[1]{\widehat{#1}}

\newcommand{\eps}{\varepsilon}

\newcommand{\supp}{\text{supp}}

\newcommand{\D}{\mathcal{D}}

\newcommand{\EEE}{\mathbf{E}}

\newcommand{\Z}{\mathbb{Z}}
\newcommand{\qq}{\mathbf{q}}

\newcommand{\indic}{1\!\!1}
\newcommand{\xxx}{\mathbf{x}}
\newcommand{\vvv}{\mathbf{v}}
\renewcommand{\O}{\mathcal{O}}

\newtheorem{theorem}{Theorem}[section]

\newtheorem{prop}[theorem]{Proposition}
\newtheorem{lemma}[theorem]{Lemma}

\newtheorem{proposition}[theorem]{Proposition}

\theoremstyle{definition}
\newtheorem{definition}[theorem]{Definition}
\newtheorem{remark}[theorem]{Remark}

\theoremstyle{remark}

\newcommand{\T}{\mathbb{T}}

\begin{document}

\title[Endline bilinear restriction estimate]{An  endline bilinear restriction estimate for paraboloids}
\author[J. Yang]{Jianwei Urbain Yang${}^\dag$}
\email{jw-urbain.yang@bit.edu.cn}
\address{Jianwei Urbain Yang,
Department of Mathematics, Beijing Institute of Technology, Beijing 100081, P. R. China}

\thanks{$^\dag$Department of Mathematics, Beijing Institute of Technology. The author is  supported by NSFC grant No. 11901032 and Research fund program for young scholars of Beijing Institute of Technology.}

\date{}



\begin{abstract}
	We prove an $L^2\times L^2\to L^q_tL^r_x$
	bilinear adjoint Fourier restriction estimate
	for $n$-dimensional elliptic paraboloids, with $n\ge 2$ and $1\le q \le \infty$, $1\le r\le 2$ being on the endline
	$\frac{1}{q}=\frac{n+1}{2}\bigl(1-\frac{1}{r}\bigr)$
	except for the critical index.
	This includes the endpoint case when $q=r=\frac{n+3}{n+1}$,  a question left unsettled in Tao \cite{TaoGFA}.
	Apart from the critical index, it
	improves the sharp non-endline result of Lee-Vargas \cite{LeeVargas} to the full range, confirming a conjecture in the spirit of
	Foschi and Klainerman \cite{FoKl} on the  elliptic paraboloid.
	Our proof is accomplished by uniting the \emph{profound}
	induction-on-scale tactics based on the wave-table theory and the method of descent both stemming from
	 \cite{TaoMZ}.
\end{abstract}

\maketitle

\section{Introduction}\label{S:introduction}
Let $n\ge 2$ be an integer
and $\varSigma$ be the elliptic paraboloid
$$
\varSigma=\Bigl\{(\xi,\tau)\in\R^{n+1};\;\tau=-\frac{1}{2}|\xi|^2\Bigr\}
$$
with the surface measure $d\sigma$ on $\varSigma$.
For any test function $f$ on $\varSigma$, define the adjoint Fourier restriction operator $\wh{fd\sigma}(x,t)$
as
$$
\wh{fd\sigma}(x,t)=\int_{\varSigma}e^{2\pi i(x\cdot\xi+t\tau)}f(\xi,\tau)\,d\sigma(\xi,\tau).
$$

For any two smooth compact hypersurfaces $\varSigma_1$ and $\varSigma_2$ being transverse subsets of $\varSigma$, with induced surface measures $d\sigma_1$ and $d\sigma_2$ respectively, it is
proved in \cite{LeeVargas} that for $1<q,r\le\infty$
such that $\frac{1}{q}<\min\bigl\{\frac{n+1}{4}, \frac{n+1}{2}\bigl(1-\frac{1}{r}\bigr)\bigr\}$,
there exists a finite constant $C=C_{q,r,\varSigma_1,\varSigma_2}>0$ such that
\begin{equation}
\label{eq:B-parab}
\Bigl\|\prod_{j=1,2}\wh{f_jd\sigma_j}\Bigr\|_{L^q(\R_t; L^r(\R^n_x))}\le C \prod_{j=1,2}\|f_j\|_{L^2(\varSigma_j \,d\sigma_j)}
\end{equation}
holds for all test functions $f_1$ and $f_2$ supported on $\varSigma_1$ and $\varSigma_2$ respectively.
This is an extension of the previous result of Tao \cite{TaoGFA} in the case $q=r>\frac{n+3}{n+1}$ to the mixed-norms.
Moreover, it is pointed out in \cite[Section 2.1]{LeeVargas} that $\frac{1}{q}\le \frac{n+1}{2}\bigl(1-\frac{1}{r}\bigr)$
is necessary and a natural conjecture is that if
we let
\begin{equation}
\label{eq:FK-(q,r)-cond}
 \mathbf{\Gamma}=\Bigl\{(q,r);\,\frac{1}{q}=\frac{n+1}{2}\bigl(1-\frac{1}{r}\bigr),\,1\le q\le\infty,\,1\le r\le 2 \Bigr\},
\end{equation}
then for any $(q,r)\in\mathbf{\Gamma}$, there is a constant $C$ depending on $q,r, \varSigma_1$ and $\varSigma_2$ such that \eqref{eq:B-parab} holds. Notice that $\mathbf{\Gamma}$ is the borderline of the range for $(q,r)$ such that \eqref{eq:B-parab} could be valid when $1\le r\le 2$.
This conjecture, if true, is an endline version of \cite{LeeVargas} and it includes the endpoint estimate of \cite{TaoGFA} as  a special case. By using Bernstein's inequality, the endline result would imply  all the other cases.
\medskip

We call the left endpoint of the endline $\mathbf{\Gamma}$ in \eqref{eq:FK-(q,r)-cond} the \emph{critical index} for \eqref{eq:B-parab}
$$
(q_c,r_c):=\begin{cases}
(\frac{4}{3},2)\,& ,\;n=2\,,\\  \bigl(1,\frac{n+1}{n-1}\bigr)\;&,\; n\ge 3\,.
\end{cases}
$$
The endline estimates can be reduced to the strongest estimate corresponding to the critical index. Indeed, if \eqref{eq:B-parab} were true for the critical index $(q,r)=(q_c,r_c)$, then we would be able to obtain the
bilinear estimate with $(q,r)$ in the full range \eqref{eq:FK-(q,r)-cond} by interpolation with the energy estimates.\\

This conjecture on the bilinear restriction estimates is connected in a deep way to the null form estimates of wave equations, an important device in the study of nonlinear wave equations (c.f. \cite{FoKl,KlMa,KT,TaoMZ,Tataru,LeeRogersVargas,LeeVargas} ). It interacts dynamically with the \emph{linear} restriction  problems posed by Stein \cite{stein78}. The bilinear approach was initiated from Bourgain \cite{Bo95} improving the boundedness of  Mockenhaupt's  cone multiplier  \cite{M} and developed further in \cite{TVV,TV-1,TV-2}. Nowadays, it has become such a highly active research area that it is almost impossible to give a comprehensive summary for all the up-to-date works in a limited space and time.
On the other hand, there already exist so many excellent survey articles and monographs, we refer to \cite{Bo2000,Wolff99,Tao2004,guthpoly,DeBook}
and references therein for a panorama of this domain. Moreover, it turns out that the endpoint bilinear restriction estimates become more and more important in applications to PDEs. We only mention here, among other things, its connexion to the  uniqueness  in Calder\'{o}n's inverse conductivity problem \cite{Hab,HKL,FPV}, where the bilinear restriction estimate along with its extensions played an essential role and it is pointed out in \cite{HKL} that further improvements would be available provided one had  the endpoint results.  For more  applications of these bilinear estimates to nonlinear dispersive equations, we refer to \cite{Dod}.\\

In the case when $q=r>\frac{n+3}{n+1}$, the sharp (non-endpoint) bilinear estimate \eqref{eq:B-parab} was established by Tao \cite{TaoGFA} by adapting  the \emph{mild} induction-on-scale argument due to Wolff \cite{Wolff} for the sharp $L^2-$bilinear estimate on the cone.
The endpoint case $q=r=\frac{n+3}{n+1}$ was left open in \cite{TaoGFA} and is recently investigated by J. Lee \cite{Lee21}. We note that a new approach towards the Fourier restriction problems is proposed by Muscalu and Oliveira \cite{MO} relating the restriction theory with the multilinear harmonic analysis, where the authors obtained sharp
linear and multilinear restriction theorems provided certain tensor product conditions on the input functions are satisfied.
Although it was observed in \cite[Section 9]{MO} that one may extend the admissible range of the exponents for \eqref{eq:B-parab} under extra hypothesis on the amount of transversality, the endpoint case remains unsettled even if the tensor-product condition is fulfilled.\\

The purpose of this paper is to prove the bilinear estimate \eqref{eq:B-parab} for all $(q,r)$ on the endline $\mathbf{\Gamma}$ except for the critical index $(q_c,r_c)$ for all $n\ge 2$ and our main result reads
\begin{theorem}
	\label{thm:main}
	Let $n\ge 2$ and $\varSigma_1$, $\varSigma_2$ be two disjoint compact subsets of $\varSigma$. Then, for any
	$(q,r)\in\mathbf{\Gamma}\setminus\{(q_c,r_c)\}$, there is a finite constant $C=C_{q,r,\varSigma_1,\varSigma_2}>0$ depending only on $q,r,\varSigma_1,\varSigma_2$
	such that \eqref{eq:B-parab} holds for all test functions
	$f_1$ and $f_2$ defined on $\varSigma_1$ and $\varSigma_2$ respectively.
\end{theorem}

\begin{remark}
	In case of $q=r$, the conjecture is usually refered as the Machedon-Klainerman conjecture \cite{FoKl,TaoGFA,TaoMZ,Wolff}, especially in  three dimensions. In general dimensions, Foschi and Klainerman provided a tentative description on possible bilinear estimates of this  form \cite{FoKl}. The mixed-norm extension is due to Lee and Vargas \cite{LeeVargas}. To highlight the origin of the question, we cautiously  refer to it as the Foschi-Klainerman conjecture.
\end{remark}

We briefly describe our proof for Theorem \ref{thm:main}.
In \cite{Wolff}, Wolff proved the sharp $L^2\times L^2\to L^q_{t,x}$ bilinear estimate on the cone for all $q>\frac{n+3}{n+1}$, by means of his celebrated \emph{induction on scale} argument, to which we would refer as a \underline{\emph{mild}} version.
To resolve the endpoint case $q=\frac{n+3}{n+1}$, Tao introduced  in \cite{TaoMZ} a \underline{\emph{profound}} version of the induction argument which enhanced the method in \cite{Wolff}. By building up an effective wave-table theory and exploring  the possible ways that local energy could concentrate, the endpoint bilinear estimate on the cone was proved in \cite{TaoMZ} in the symmetric norms and then extended to the mixed-norms by Temur \cite{Temur}\,( see also \cite[Section 2.1]{LeeVargas} for the non-endline case), as well as to the variable coefficient setting by J. Lee \cite{JLeeTAMS} (in the symmetric norms). Moreover, the endpoint result in \cite{TaoMZ} is also generalized to the case when one of the waves has  large frequency in order to develop sharp null form estimates, which is nearly optimal  due to its connexion with the  (back then) unsettled endpoint bilinear restriction estimates on the paraboloid \cite[Section 17]{TaoMZ}. See also \cite[Section 2.1]{LeeVargas} for the version of mixed-norms. For extensions of Tao's results to some second order hyperbolic equations with rough coefficients, we refer to Tataru \cite{Tataru}.\smallskip

A crucial geometric fact utilized in the proof of the endpoint bilinear estimates is that at any point of the cone, the normal vector is always in the lightray directions, due to the single vanishing principle curvature on the cone along the radiative null direction. This may be regarded as a lightcone version  of the \emph{Kakeya compression} phenomenon, compared to those observed by Bourgain \cite{Bo91} and  Bourgain-Guth \cite{BoGu}, where distorted tubes contained in a neighborhood of subvarieties are considered. Combining this fact with the energy estimates on  lightcones of opposite colour    \cite[ Section 13]{TaoMZ}, one is able to dispose of the energy-concentrated case in order to close the enhanced  induction for the endpoint problem.
This geometric property fails on the paraboloid $\varSigma$ since the Gaussian curvature is nowhere vanishing.
See \cite{Lee21} for a study on the endpoint case, by using a new energy concentration argument, where \eqref{eq:B-parab} is still considered in the symmetric case, \emph{i.e.} $q=r=\frac{n+3}{n+1}$.\smallskip

The idea of this paper is different from  \cite{Lee21}.
We retain the original induction scheme of Tao \cite{TaoMZ} using the same energy concentration, and prove the bilinear estimate in the mixed-norms on  the whole endline apart from the critical index.  A novel ingredient that we take in is the use of  the \emph{method of descent}  proposed in the same paper \cite{TaoMZ}.
To illustrate the idea of this method, let us start with   the three dimensional spacetime.
As a  well-known fact, a 2-plane parallel to a generatrix of a (2-dim) cone in $\R^3$, and not passing through its vertex, intersects the cone in a (1-dim) parabola \cite{AI}. This elementary fact is  readily generalized to higher dimensions. Indeed, for  $n\ge2$, consider the $(n+1)$-dimensional backward cone in $\R^{n+2}$
$$
\mathscr{V}:=\bigl\{(\xxx,t)\in \R^{n+1}\times\R;\, \,t=-|(x,x_{n+1})|\,\bigr\},
$$
with $\xxx=(x,x_{n+1})$, $x_{n+1}$ being the auxiliary variable, $|(x,x_{n+1})|=\sqrt{|x|^2+x_{n+1}^2}$. Denote $\mathbf{e}_\pm=\frac{\mathbf{e}_{n+1}\pm \mathbf{e}_{n+2}}{\sqrt{2}},$
where
$
\mathbf{e}_j=(0,\ldots,0,\underbrace{1}_{j-th},0,\ldots,0),\quad \forall \;j\in\{1,\ldots,n+2\}.
$
Let $\lambda>1$ and $\varPi_{\lambda}$ be the $(n+1)-$dimensional hyperplane passing through the point $-\lambda \mathbf{e}_+$ and being normal to $\mathbf{e}_+$. Then $\mathscr{ P}^\lambda:=\varPi_{\lambda}\cap \mathscr{V}$
is an $n$-dimensional elliptic paraboloid in $\varPi_{\lambda}$,  symmetric around the  axis passing through $-\lambda \mathbf{e}_+$  in direction of  $\mathbf{e}_-$, and  parametrized by the circular variables  $x\in \R^n_x=
\mathrm{span}(\mathbf{e}_1,\mathbf{e}_2,\ldots,\mathbf{e}_n)$. The drawback of this fact is that the focal point of $\mathscr{ P}^\lambda$ depends on the varying parameter $\lambda$. To overcome this obstacle, one may  stretch the integration  along the $\mathbf{e}_-$ direction  by $\lambda$  so that $\varSigma$ can be treated as a limiting surface after scaling back along $\mathbf{e}_-$ direction  and   letting $\lambda\to +\infty$. To match this change of variable, a negative power of $\lambda$ will be involved, which is related to the null form estimates from dimensional analysis.
\smallskip

This method was introduced as an \emph{intermediate} step to demonstrate that the conjectured null form estimate (see (85) of Problem 17.1 \cite{TaoMZ}) implies the Machedon-Klainerman conjecture for paraboloids in the symmetric norm, \emph{i.e.} $q=r$ in \eqref{eq:B-parab},  integrating out  the one dimensional  auxiliary variable $x_{n+1}$ after taking limit.
Our strategy towards Theorem \ref{thm:main} is to show that the (unlabelled) bilinear estimate on P. 260 of \cite{TaoMZ}
$$\|\phi \psi\|_{p}\lesssim R^{1/p}\|f\|_2\|g\|_2,$$
as a consequence of the \emph{unsettled} stronger estimate (85)  in  \cite[Section 17]{TaoMZ}, and employed in the \emph{intermediate} step  to get the endpoint bilinear estimate on paraboloids, can be indeed proved directly by suitably modifying the profound induction on scale argument of \cite{TaoMZ}, not only in the symmetric norms, but also in the mixed $L^q_tL^r_x$  norms for all $q,r$ on the endline $\mathbf{\Gamma}\setminus\{(q_c,r_c)\}$, without first resolving the much more difficult question on the null form conjecture. Since we will work essentially with the $\O(1)-$neighbourhood of the cross section $\mathscr{ P}^\lambda$ on the cone, the above mentioned Kakeya compression property remains valid, replacing the the spatial unit sphere in which circular components used to be resident \cite{TaoMZ}, with a subset of the paraboloid $\varSigma$. The distribution of the directions of the tubes associated to the careful wave-packet decomposition become congregated by a ratio  $\O(\lambda^{-1})$, which will be compensated by the $\lambda^{-1/q}-$factor arising from an average along the $\lambda-$stretched direction  $\mathbf{e}_-$.
\smallskip

The only missing answer for \eqref{eq:B-parab} in Theorem \ref{thm:main} is the critical index $(q_c,r_c)$, which is out of reach by the current  method. In fact, this problem shares a  level of the same difficulty  concerning the endpoint multilinear restriction theorem of Bonnett-Carbery-Tao \cite{BCT}, a very difficult open question. Even as a weaker result, the endpoint multilinear Kakeya inequality can only be established through the intricate algebraic topological method by Guth \cite{guth}.
\\

The paper is organized as follows.
In Section \ref{sec:pre+top}, we first introduce a $\lambda-$dependent  operator $S^\lambda$ in a
similar fashion to that of \cite{TaoMZ} and study the basic properties associated to the corresponding dispersive equation, emphasizing the energy estimates for the waves on conic sets of opposite colour. We then prove the careful wave packet decomposition and  construct the wave tables on stretched spacetime cubes for the red and blue waves, to be specified in the context below.
The crucial property that  the red and blue waves on
a stretched cube can be  effectively approximated via $C_0-$\emph{quilts} of the wave tables on a quantitative \emph{interior} of a proper enlargement of  the cube in spirit of \cite{TaoMZ} will be proved. The proof is reduced to a tamed bilinear $L^2-$Kakeya type estimate, which eradicates the logarithmic loss as reminded in the last section of \cite{TaoGFA}.  In Section \ref{sect:huygens}, we introduce the spatial localization operator $P_D$ in the $S^\lambda$-operator version and  use these operators to capture the energy concentration of waves. In Section \ref{sec:wt}, we introduce the core quantity $A^{\lambda}(R)$ to bootstrap with respect to the scales $R\le \lambda$. To this end,   an auxiliary quantity $\mathscr{A}^{\lambda}(R,r,r')$ will get involved, which is defined based on the notion of \emph{energy concentration}. Here, the two parameters $r,r'$ are  scales for  measuring the level of the concentration of energy. This is crucial for the endpoint estimate as in \cite{TaoMZ}  in order to wrest in a universal constant strictly less than one. To close the induction, $\mathscr{A}^\lambda$ needs to be controlled  by  $A^\lambda$ up to a constant very close to one, which is easy in the non-concentrated case. The difficult part is the case when the energy is highly concentrated, for which we  make use of the Kakeya compression and a non-optimal control on the exterior energy in terms of $A^\lambda$ by inductive hypothesis. Finally,  in Section \ref{sect:pf-thm}, we close the induction on $A^\lambda(R)$ and complete the proof. \smallskip

To end up this section, we remark that  it seems that  the same argument should  work  also for general hyperbolic paraboloids for the endpoint problems of bilinear restriction estimate left open in \cite{LeeTAMS,Vargas,LeeVargas}. One might also be able to get the endline result of the Wave-Schr\"{o}dinger bilinear restriction estimates, where the off-endline case  is established  by Candy \cite{Candy} and applied to  the wave-Schr\"{o}dinger interactions in the Zakharov system \cite{CHN}. Moreover, it is probably more interesting to tackle the endpoint case of the bilinear estimates related to  Klein-Gordon equations, which have been investigated by Bruce et al \cite{BOS} and Candy \cite{Candy} in the non-endpoint case.
 Finally, it is plausible that the $\eps-$loss of the bilinear oscillatory integral estimates in \cite{LeeJFA} can also be removed.
\subsection*{Notations}
For any fixed $(q,r)\in\mathbf{\Gamma}\setminus\{(q_c,r_c)\}$, 
let $N\ge1$ be a sufficiently  large integer depending only on $n,q,r$ and let $C_0=C_0(\eps,N)=2^{\lfloor\frac{N}{\eps}\rfloor^{10}}$, where $\eps>0$ will  be taken small when necessary in the process of the proof,  but it will never tend to  zero.
We use  $A\lesssim B$, $A=O(B)$ or $A=\mathcal{O}(B)$ to denote  $A\le CB$ for some $C>0$, which may change from line to line and  depends only on $n,\eps,q,r$, but not explicitly on $C_0$.
We use $A\ll B$ to denote $A\le C^{-1} B$ for some sufficiently large constant $C$.

\subsection*{Acknowledgements}
The author is supported by NSFC grant No. 11901032 and Research fund program for young scholars of Beijing Institute of Technology.
The author is also grateful to LAGA in Universit\'e de Sorbone Paris Nord, where part of this work was done.

\section{Preliminaries }\label{sec:pre+top}
\subsection{The $S^\lambda-$propagator and its basic properties}
For each $j=1,2$, let $V_j$ be the projection from $\varSigma_j$ to the $\xi-$variables. Let  $e_1=(1,0,\ldots,0)\in\R^n$.
By compactness and a finite partition of $\varSigma_1, \varSigma_2$, we may assume 
$$
V_1=\Bigl\{\xi\in \R^n:\bigl|\xi-e_1\bigr|\le\frac{1}{200n} \Bigr\},\quad
V_2=\Bigl\{\xi\in\R^n:\bigl|\xi\bigr|\le\frac{1}{200n}\Bigr\}\,,
$$ 
after using a suitable rotation, scaling and the Galilean transformation. 
Due to technical reasons, we also need the slightly enlarged version of $V_j$, namely $$\widetilde{V}_j=\Bigl\{\xi\in\R^n:\,\mathsf{dist}(\xi,V_j)\le \frac{1}{100n}\Bigr\},\;\quad j=1,2.$$ 
We denote 
$\mathcal{B}:=\{(\xi,s):|\xi|\le 2, |s|\le 2 \}$ for short.
By Plancherel's theorem, it is more convenient to work in the language of dispersive equations. \smallskip

For any $\lambda\ge 2^{C_0}$, we introduce the $S^\lambda(t)$ operator\,:
\begin{definition}
	For any $f_j\in \mathcal{S}(\R^{n+1})$ with $j\in\{1,2\}$ such that $\wh{f}_j\in C_0^\infty(\widetilde{V}_j\times I)$ where $I=[-2,2]$, let
	$$
	S_j^\lambda(t)f_j(\xxx)=\iint
	e^{2\pi i \bigl(x\cdot\xi+x_{n+1}s-\frac{t}{2}\frac{|\xi|^2}{\lambda+s}\bigr)}
	a_j(\xi,s)
	\wh{f}_j(\xi,s)\,d\xi ds,
	$$
	where we denote $\mathbf{x}=(x,x_{n+1})$ for brevity and $f_j\mapsto \wh{f}_j$ is the  Fourier transform on $\R^{n+1}$ and
	$a_j\in C_c^\infty(\R^{n+1})$ such that $a_j$ equals to one  on $ \widetilde{V}_j\times I$ and that $a_j$
	vanishes outside $\{(\xi,s);\,\mathsf{dist}((\xi,s),\widetilde{V}_j\times I)\le (100n)^{-1}\}$.
\end{definition}

\begin{prop}
	\label{pp:KKK}
	For each $j=1,2$,
	let $\Xi^\lambda_j=\bigl\{\frac{\xi}{s+\lambda};(\xi,s)\in \supp\;a_j\bigr\}$. Define
	$$
	\mathcal{K}_j^\lambda(\xxx,t)=\iint
	e^{2\pi i \bigl(x\cdot \xi+x_{n+1}s-\frac{t}{2}\frac{|\xi|^2}{\lambda+s}\bigr)}
	a_j(\xi,s)\;d\xi ds.
	$$
	Then,
	\begin{equation}
	\label{eq:K-*}
		S_j^\lambda(t)f(\xxx)= \mathcal{K}_j^\lambda(\cdot,t)*f(\xxx),
	\end{equation}
	with
	\begin{equation}
	\label{eq:KKK}
		\bigl|\mathcal{K}_j^\lambda(\xxx,t)\bigr|\lesssim_{M}
	\bigl(1+\mathsf{dist}\bigl((\xxx,t),\,{\mathbf{ \Lambda}_j^\lambda}\;\bigr)\bigr)^{-M}
	\end{equation}
	for all $(\xxx,t)\in\R^{n+2}$  and all integers $M\ge 1$,
	where
	$$
	\mathbf{ \Lambda}_j^\lambda:=\bigcup_{\ell\in\R}\Bigl\{\ell\Bigl(v,-\frac{|v|^2}{2},1\Bigr); \, v\in\,2\;
	\Xi^\lambda_j\Bigr\},
	$$
	with $2\,\Xi_j^\lambda$ being the set with the same center of $\Xi_j^\lambda$ but with the double diameter. Here, $\bm{\Lambda}_j^\lambda$
	is a conic hypersurface in $\R^{n+2}$ with $\mathrm{dim}(\mathbf{ \Lambda}_j^\lambda)=n+1$.
\end{prop}
\begin{proof}
	By definition, \eqref{eq:K-*} is clear. Moreover, we have \eqref{eq:KKK} by using the trivial estimate if $(\xxx,t)$ is in a $C-$neighbourhood of $\mathbf{ \Lambda}^\lambda_j$. Next, assume  that $(\xxx,t)$ is $C$ away from $\mathbf{ \Lambda}^\lambda_j$ for $C\gg1$.
	Letting
	$$
	L=\frac{1+(2\pi i)^{-1}(x-t(s+\lambda)^{-1}\xi)\cdot \partial_\xi+(2\pi i)^{-1}(x_{n+1}+\frac{t}{2}(s+\lambda)^{-2}|\xi|^2)\partial_s}{1+|x-t(s+\lambda)^{-1}\xi|^2+\bigl(x_{n+1}+\frac{t}{2}(s+\lambda)^{-2}|\xi|^2\bigr)^2},
	$$
	and integrating by parts using
	$$
	L^M e^{2\pi i \bigl(x\cdot \xi+x_{n+1}s-\frac{t}{2}\frac{|\xi|^2}{\lambda+s}\bigr)}
	=e^{2\pi i \bigl(x\cdot \xi+x_{n+1}s-\frac{t}{2}\frac{|\xi|^2}{\lambda+s}\bigr)},\quad \forall \; M\ge 1,
	$$
	we have (c.f. \cite[Chapter 1]{Sogge})
	\begin{multline*}
	\mathcal{K}_j^\lambda(\xxx,t)=\sum_{m=0}^{M}\sum_{|\gamma|=m}\iint_{\R^{n+1}}
	e^{2\pi i\bigl(x\cdot\xi+x_{n+1}s-\frac{t}{2}\frac{|\xi|^2}{\lambda+s}\bigr)} c_{\gamma,M}(\xxx,t;\xi,s+\lambda)\,\partial^\gamma a_j(\xi,s)d\xi ds,
	\end{multline*}
	where $\{c_{\gamma,M}\}_{\gamma}$ are smooth functions, 
	satisfying that  for all $0\le m\le M$ and all multi-indices $\gamma$ with $|\gamma|=m$ 
	$$
	|c_{\gamma,M}(\xxx,t;\xi,s+\lambda)|\lesssim_{M}
	\Bigl(1+\bigl|\frac{t\,\xi}{\lambda}\bigr|\Bigr)^{M-m}	\biggl(1+\Bigl|x-\frac{t\,\xi}{s+\lambda}\Bigr|+\Bigl|x_{n+1}+\frac{t\,|\xi|^2}{{2}(s+\lambda)^2}\Bigr|\biggr)^{-2M+m} .$$
	To see this, denote $\langle s\rangle=(1+|s|^2)^{\frac{1}{2}}$
	and let $\mathscr{Z}=\langle\xxx-t\mathbf{v}\rangle$
	with $\vvv=\bigl(\frac{\xi}{s+\lambda},-\frac{|\xi|^2}{2(s+\lambda)^2}\bigr)$. The adjoint operator $L^*$ of $L$ can be written into the form $L^*=\bm{\alpha}\cdot \partial_{\xi,s}+\beta$
	where $\bm{\alpha}=\bm{\alpha}(\xxx,t;\xi,s+\lambda)\in\mathbb{C}^{n+1}$ and $\beta=\beta(\xxx,t;\xi,s+\lambda)\in\mathbb{C}$
	are smooth functions such that on $\supp\;a_j$, we have $|\bm{\alpha}|\lesssim\mathscr{Z}^{-1}$ and $|\beta|\lesssim\langle t\xi/\lambda\rangle \mathscr{Z}^{-2}$ and for all $\gamma$ with $|\gamma|\ge1$, we have $|\partial^\gamma \bm{\alpha}|\sim |\partial^{\tilde{\gamma}}\beta|$ for some $\tilde{\gamma}$ with $|\tilde{\gamma}|+1=|\gamma|$. Moreover, we have $|\partial^\gamma\beta|\lesssim\bigl(\langle t\xi/\lambda\rangle \mathscr{Z}^{-2-|\gamma|}\bigr)$.
	Here $\partial=\partial_{\xi,s}$ refers to taking derivatives only in the frequency variables $(\xi,s)$.
	For any $0\le m\le M$ and $\gamma$ with $|\gamma|=m$, one easily finds that $c_{\gamma,M}$ is homogeneous of  order $M$,  where the exponent of the $\beta$-factor  is   at most $(M-m)$.

	Since $|\xxx-\frac{t}{\lambda}\vvv^\lambda|\gtrsim \langle t/\lambda\rangle$ with $\vvv^\lambda=\bigl(\frac{\lambda\xi}{s+\lambda},-\frac{\lambda|\xi|^2}{2(s+\lambda)^2}\bigr)$ for all $(\xi,s)\in\supp\, a_j$,  by combining this with the bound on $c_{\gamma,M}$, we find that $\mathcal{K}^\lambda_j(\xxx,t)$  can be bounded with
	\begin{multline*}
		 \sup_{(\xi,s)\in \supp\, a_j}\Bigl(1+\bigl|x-t(s+\lambda)^{-1}\xi\bigr|+\bigl|x_{n+1}+\frac{t}{2}(s+\lambda)^{-2}|\xi|^2\bigr|\Bigr)^{-M}\\
	\lesssim \bigl(1+\mathsf{dist}\bigl((\xxx,t),{\mathbf{ \Lambda}_j^\lambda}\;\bigr)\bigr)^{-M}\,.\quad\qquad\qquad
	\end{multline*}
The proof is complete.
\end{proof}
It is convenient to call  $F_j^\lambda(\xxx,t):=S_j^\lambda(t) f_j(\xxx)$ the \emph{red} and \emph{blue} waves respectively for $j=1,2$. For each $j$, the \emph{energy} of $F_j^\lambda(\xxx,t)$ is defined as
$$
\EEE(F_j^\lambda):=\|F_j^\lambda(\cdot,0)\|_{L^2(\R_\xxx^{n+1})}^2\,\,.
$$
When there is no need to distinguish the color, we shall simply call $F^\lambda$ a \emph{wave}.
The following energy estimates on conic sets of opposite colour is crucial.
\begin{lemma}
	\label{lem:opposite}
	For $j=1,2$, let $\mathbf{ \Lambda}_j^\lambda(z_0,r)$ be an $\O(r)-$neighbourhood of $\mathbf{ \Lambda}_j^\lambda+z_0$ with $z_0=(\xxx_0,t_0)\in \R^{n+2}$
	and $r\ge 1$. Then,  we have
	\begin{equation}
	\label{eq:opp}\| F^\lambda_j\|_{L^2(\mathbf{ \Lambda}^\lambda_k(z_0,r))}\lesssim (\lambda r)^{1/2}
	\EEE(F_j^\lambda)^{1/2},\;\quad \forall\; j,k\in \{1,2\},\; j\ne k\,,
	\end{equation}
	for all $z_0$ and $r\ge 1$.
\end{lemma}
\begin{proof}
	The argument is similar to \cite{TaoMZ}.
	By translation invariance which is clear from modulation of the input function depending on $z_0$ in the frequency space, we may take $z_0=(\mathbf{0},0)\in\R^{n+1}\times\R$.
	By symmetry, we only consider $(j,k)=(1,2)$. Let
	$$
	\mathfrak{D}_2^{\lambda,t}=\bigl\{\mathbf{x}\in\R^{n+1};\,\mathsf{dist}((\xxx,t),\mathbf{ \Lambda}_2^\lambda)\lesssim r\bigr\}.
	$$
	Let $S^{\lambda,*}_1$ be the adjoint of $S_1^\lambda$.
	By the $TT^*$ principle, it suffices to show
	$$
	\Bigl\| \int S_1^{\lambda,*}(t)\bigl[ \indic_{\mathfrak{D}^{\lambda,t}_2}\, H(\cdot,t)\bigr]dt \Bigr\|_{L^2(\R^{n+1})}
	\lesssim (\lambda r)^{1/2}\|H\|_{L^2(\R^{n+2})}
	$$
	for all $H\in L^2(\R^{n+2})$.
	Taking squares and multiplying out, we have
	\begin{multline}
		\Bigl\| \int S_1^{\lambda,*}(t)\bigl[ \indic_{\mathfrak{D}^{\lambda,t}_2}\, H(\cdot,t)\bigr]dt \Bigr\|_{L^2(\R^{n+1})}^2\\
	\lesssim
	\iint
	\indic_{\mathfrak{D}^{\lambda,\tilde{t}}_2}(\tilde{\xxx})\,
	\mathcal{K}_1^\lambda(\tilde{\xxx}-\xxx,\tilde{t}-t)\;
	\indic_{\mathfrak{D}^{\lambda,t}_2}(\xxx)\, H(\xxx,t)\;\overline{H(\tilde{\xxx},\tilde{t})}\;d\xxx d\tilde{\xxx}dtd\tilde{t},\label{eq:HHH}
	\end{multline}
	where $\mathcal{K}_1^\lambda$ is given by Proposition \ref{pp:KKK} with $a_1$ replaced by $a_1^2$.
	\smallskip
	
	By  Cauchy-Schwarz and the $L^2-$boundedness  $$ \sup_{t,\tilde{t}}\;\bigl\|\indic_{\mathfrak{D}^{\lambda,\tilde{t}}_2}S^\lambda_{ 1}(\tilde{t})\circ S^{\lambda,*}_{ 1}(t)\indic_{\mathfrak{D}^{\lambda,t}_2}\bigr\|_{L^2(\R^{n+1})\to L^2(\R^{n+1})}=\mathcal{O}(1),$$  the  $|t-\tilde{t}|\lesssim \lambda r$ part of the integral \eqref{eq:HHH} is  bounded by $\lambda r\|H\|_2^2$.
	\smallskip

	Next, we estimate the $|t-\tilde{t}|\gg \lambda r$ part of the integral.
   For any $u\in\Xi_1^\lambda$ and $\tilde{v}\in\Xi_2^\lambda$, let $\mathfrak{ L}(\tilde{v},u)\subset\R^{n+1}$ be the straight line passing through the origin along the direction $\bigl(\tilde{v}-u,-\frac{1}{2}|\tilde{v}|^2+\frac{1}{2}|u|^2\bigr)$. Define $\mathfrak{ L}(\tilde{v},v)$ for $v\in \Xi_2^\lambda$ in the same way.  If we let $\mathfrak{ L}^{\Delta}(\tilde{v},u)$ be the two ends on $\mathfrak{ L}(\tilde{v},u)$ outside the ball
   in the spacetime $\R^{n+2}$ of radius $\Delta$ with $\Delta\gg  r$ and centered at the origin, then we have $\mathsf{dist}\bigl(\mathfrak{ L}^{\Delta}(\tilde{v},u),\,\mathfrak{ L}(\tilde{v},v)\bigr)\gtrsim \Delta$ for all $u\in\Xi_1^\lambda$ and $v,\tilde{v}\in\Xi_2^\lambda$, thanks to the non-vanishing Gaussian curvature of $\varSigma$.
   Indeed, this is clear if the directions of $(u-\tilde{v})$ and $(v-\tilde{v})$ are separated by a fix small constant $0<\theta\ll1$. Otherwise, there is $\tilde{u}$ having the property that $|\tilde{u}-u|\lesssim \lambda^{-1}$
   and $(\tilde{u},-\frac{1}{2}|\tilde{u}|^2)$ belongs to the two-plane $\Pi=\Pi(v,\tilde{v})$ passing through the origin such that $\mathfrak{ L}(\tilde{v},v)\subset \Pi$ and $(\underbrace{0,\ldots,0}_{n\text{ times}},1)\in \Pi$ (the co-planar case). Using the condition that
   $\mathsf{diam}(\widetilde{V}_1), \mathsf{diam}(\widetilde{V}_2)\ll \mathsf{dist}(\widetilde{V}_1,\widetilde{V}_2)$
   and the strict convexity of the parabola (curvature property), it is easy to deduce that ( by using  Taylor's expansion say) the two vectors $
   (\tilde{u}-\tilde{v},-\frac{1}{2}|\tilde{u}|^2+\frac{1}{2}|\tilde{v}|^2)$
   and $(v-\tilde{v},-\frac{1}{2}|v|^2+\frac{1}{2}|\tilde{v}|^2)
   $ are separated by an angle $\varphi\gtrsim 1$ which depends only on $\widetilde{V}_1, \widetilde{V}_2$. Simple solid geometric comparison inequalities yield the result.

  Using this fact,  $\mathsf{dist}(\Xi_1^\lambda,\Xi_2^\lambda)\gtrsim \lambda^{-1}$ and the  $\mathfrak{D}_2^{\lambda,t}$, $\mathfrak{D}_2^{\lambda,\tilde{t}}$ constraints for $(\xxx,t),(\tilde{\xxx},\tilde{t})$:
  $$
  \xxx=\Bigl(tv,-\frac{t}{2}|v|^2\Bigr)+\O(r),\quad
  \tilde{\xxx}=\Bigl(\tilde{t} \tilde{v},-\frac{\tilde{t}}{2}|\tilde{v}|^2\Bigr)+\O(r)
  $$
  for some $v,\tilde{v}\in\Xi_2^\lambda$, one easily deduces that by using triangle inequality
  \begin{equation}
  \label{eq:OPJKK}
  \bigl|\nabla_{\xi,s}\bigl(\langle \tilde{\xxx}-\xxx,(\xi,s)\rangle-(\tilde{t}-t)(2(\lambda+s))^{-1}|\xi|^2\bigr)\bigr|\gtrsim \lambda^{-1} |t-\tilde{t}|,
  \end{equation}
  for all $ (\xi,s)\in\supp\; a_1.$
Using \eqref{eq:KKK} and a  non-stationary phase (integration by parts) argument yield
	$$
	\eqref{eq:HHH}
	\lesssim_N
	\iint (1+|t-\tilde{t}|/\lambda r)^{-N} \|H(\cdot,t)\|_2\,|H(\cdot,\tilde{t})\|_2\;dt d\tilde{t},
	$$
	concluding the proof by using Schur's test.
	We remark that one may need normalize  \eqref{eq:OPJKK} by dividing $r$ on its both sides when defining the invariant differential operator akin to  $L$ as in the proof of Proposition  \ref{pp:KKK}.
\end{proof}

\subsection{The $(\lambda,\varpi,\varrho)-$wavepacket decomposition}
\begin{lemma}
	\label{lem:wp-d}
	Let $\lambda \ge 2^{10 C_0}$,
	$0<\varpi\le 2^{-C_0}$ and $\varrho\in[2^{C_0/2},\lambda]$ with $C_0$ large. Define
	$\mathcal{ L}=\varpi^{-2}\varrho\,\Z^{n+1}$ and  $\varGamma=\varrho^{-1}\Z^n$. For any function $f\in\mathcal{S}(\R^{n+1})$ such that
	$ \;\wh{f} $ is supported in $\mathcal{B}$, the following statement holds:

	For each $(\vvv,\mu)\in\mathcal{ L}\times\varGamma$, there
	is a wave $F^\lambda_{\vvv,\mu}$ such that we have
	\begin{equation}
	\label{eq:wpd}
	S^\lambda(t) f(\xxx)=\sum_{(\vvv,\mu)\in\mathcal{ L}\times\varGamma}F^\lambda_{\vvv,\mu}(\xxx,t),\quad\forall \,(\xxx,t)\in\R^{n+1}\times\R.
	\end{equation}
	Moreover, for any $B\gg 1$, there are $c_{\vvv,\mu,k}>0$ and $\phi_{\vvv,\mu,k}\in C^\infty(\R^{n+2})$ with $k\in\Z$  such that we may decompose further
    $$	F^\lambda_{\vvv,\mu}(\xxx,t)=\sum_{k\in\Z} c_{\vvv,\mu,k}\,\, \phi_{\vvv,\mu,k}(\xxx,t),$$
	for all $\xxx\in\R^{n+1}$ and $t\in\R$, 	
	and that there is a constant $C_n>0$, only depending on $n$, for which we have
	\begin{equation}
	\label{eq:L-2-sum}
	\varrho^n
	\sum_{(\vvv,\mu)\in\mathcal{ L}\times\varGamma}\sum_{k\in\Z} c_{\vvv,\mu,k}^2\lesssim_B \varpi^{-C_n}
	\EEE(S^\lambda f).
	\end{equation}
	For any $t_0\in\R$ and $\vvv=(\nu,\nu_{n+1})\in(\R^n\times\R)\cap\mathcal{ L}$, we have for any integer $M\ge 1$
	\begin{multline}
	\label{eq:decay-wp-0}
	|\phi_{\vvv,\mu,k}(\xxx,t)|\lesssim_{B,M} \varpi^{-O(M)}\Bigl(1+\varpi^{2}\varrho^{-1}\bigl|k-\nu_{n+1}\bigr|\Bigr)^{-B}\\
	\times\Bigl(1+\varrho^{-1}\bigl|x-\nu-t\frac{\mu}{\lambda}\bigr|+\bigl|x_{n+1}-k+t\frac{|\mu|^2}{2\lambda^2}\bigr|\Bigr)^{-M},
	\end{multline}
	for all $\xxx=(x,x_{n+1})\in\R^{n+1}$ and $|t-t_0|\lesssim \lambda\varrho^2$.
	
	Finally, the Bessel type inequality holds
	\begin{equation}
	\label{eq:bessel}
	\Biggl(\sum_{\varDelta}\sup_t\,\Bigl\|\sum_{(\vvv,\mu)\in\mathcal{ L}\times\varGamma} m_{\vvv,\mu}^\varDelta F^\lambda_{\vvv,\mu}(\cdot,t)\Bigr\|^2_{L^2(\R^{n+1})}\Biggr)^{\frac{1}{2}}
	\le (1+C_n\varpi)\|f\|_{L^2},
	\end{equation}
	for all $m^\varDelta_{\vvv,\mu}\ge 0$ such that $\displaystyle\sup_{(\vvv,\mu)\in\mathcal{ L}\times\varGamma}\sum_{\varDelta} m^\varDelta_{\vvv,\mu}\le 1$ where  $\sum_\varDelta$ is summing over a finite number of $\varDelta$'s.
\end{lemma}
\begin{proof}
	By translation in the physical spacetime and the modulation in the frequency space, we may take $t_0=0$ without loss of generality.
	
	Let  $\mathbf{ \Upsilon}_0\in\mathcal{S}(\R^{n+1})$
	be a non-negative Schwartz function
	such that $\wh{\mathbf{ \Upsilon}}_0$ is supported in $U:=\{(\xi,s)\in\R^{n+1};\,|(\xi,s)|\le 1/10\}$ and that $\wh{\mathbf{ \Upsilon}}_0$ equals to one on $\frac{1}{2}U$.  Put
	$$
	\mathbf{\Upsilon}_\vvv(\xxx)=\mathbf{ \Upsilon}_0(\varpi^2\varrho^{-1}(\xxx-\vvv)),\quad \vvv\in\mathcal{ L}.
	$$
	By the Poisson summation, we have
	$\sum_{\vvv\in\mathcal{ L}}\mathbf{\Upsilon}_\vvv(\xxx)=1$ for all $\xxx$.\smallskip
	
	Let $\Box=\underbrace{[-1/2,1/2)\times\cdots\times[-1/2,1/2)}_{n\,\text{times}}$
	and $\indic_\Box$ be the characteristic function of the unit box $\Box$.	
 For each
	$(\vvv,\mu)\in\mathcal{ L}\times\varGamma$, let
	$$
	a_{\vvv,\mu}(\xxx,\xi)=\mathbf{\Upsilon}_\vvv(\xxx)
	\bigl(\indic_\Box*\indic_\Box\bigr)(\varrho(\xi-\mu)).
	$$
	For any $f\in\mathcal{S}(\R^{n+1})$ such that
	$\supp  \;\wh{f}(\xi,s) \subset\mathcal{B}$,
	define
	$$
	f_{\vvv,\mu}(\xxx)=\int_{\mathcal{B}}
	e^{2\pi i \langle\xxx,(\xi,s)\rangle}
	a_{\vvv,\mu}(\xxx,\xi) \,\wh{f}(\xi,s)\;d\xi ds.
	$$
	Then, by Fubini's theorem, we have $f(\xxx)=\sum_{(\vvv,\mu)\in\mathcal{ L}\times\varGamma} f_{\vvv,\mu}(\xxx)$ for all $\xxx\in\R^{n+1}$.\smallskip
	
	By linearity of $S^\lambda(t)$, we have
	\begin{equation}
	\label{eq:wp-decomp}
	S^\lambda(t)f(\xxx)=\sum_{(\vvv,\mu)\in\mathcal{ L}\times\varGamma}
	S^\lambda(t)f_{\vvv,\mu}(\xxx)\,.
	\end{equation}
	Denoting  $F^\lambda(t)=S^\lambda(t) f$ and $F^\lambda_{\vvv,\mu}(t)=S^\lambda(t) f_{\vvv,\mu}$, we get  \eqref{eq:wpd}.\smallskip

	To obtain the further decomposition, let $\alpha\in C_c^\infty(\R^{n})$ be such that $\alpha$
	equals to one on $\{\xi\in\R^{n}; |\xi|\le 50 n\}$ and vanishes outside an $\O(1)-$neighborhood of this set. Let $\beta\in C_c^\infty(\R)$ be a similar function such that $\beta$ equals to one on $[-50,50]$. Put
	 $p(\xi,s)=\alpha(\xi)\beta(s)$ and
	define
	$$
	K_{t,\mu}^{\lambda,\varrho}(\xxx)=\int_{\R^{n+1}}
	e^{2\pi i\bigl(x\cdot \xi+x_{n+1}s-\frac{t}{2}\frac{|\xi|^2}{\lambda+s}\bigr)}
	p\bigl(\varrho (\xi-\mu),s\bigr)
	\,d\xi ds.
	$$
	We have $F_{\vvv,\mu}^\lambda(t)=K^{\lambda,\varrho}_{t,\mu}*f_{\vvv,\mu}$. Changing variables, we have
	$$
	K_{t,\mu}^{\lambda,\varrho}(\xxx)=\frac{e^{2\pi ix\cdot \mu}}{\varrho^n}
	\int_{\R^{n+1}}
	e^{2\pi i\bigl(\varrho^{-1}x\cdot \xi+x_{n+1}s-\frac{t}{2}\frac{|\mu+\varrho^{-1}\xi|^2}{\lambda+s}\bigr)}
	p(\xi,s)\,d\xi ds.
	$$
	Expanding $\frac{|\mu+\varrho^{-1}\xi|^2}{2(\lambda+s)}$ with respect to $(\xi,s)$ using
	$
	(1+\theta)^{-1}=1-\theta+\theta^2\mathcal{E}(\theta)
	$
	with $\mathcal{E}(\theta)=(1+\theta)^{-1}$, we have
	\begin{multline}
	\label{eq:exp-K}
		K_{t,\mu}^{\lambda,\varrho}(\xxx)=\varrho^{-n}e^{2\pi ix\cdot \mu}e^{-\pi it\lambda^{-1}|\mu|^2}\\
\times	\int_{\R^{n+1}}
	e^{2\pi i\bigl(\varrho^{-1}(x-t\frac{\mu}{\lambda})\cdot \xi+(x_{n+1}+t\frac{|\mu|^2}{2\lambda^2})s+t\,\mathscr{E}_\mu^{\lambda,\varrho}(\xi,s)\bigr)}
	p(\xi,s)\,d\xi ds,\qquad
	\end{multline}
where $\mathscr{E}_\mu^{\lambda,\varrho}$ is a smooth function  on the support of $ p$ and bounded along with all its derivatives by $\O(\lambda^{-1}\varrho^{-2})$ for all $\mu$, thanks to $\varrho\le \lambda$.
In fact, write
$$
\frac{|\mu+\varrho^{-1}\xi|^2}{2(\lambda+s)}=
\frac{\varrho^{-2}|\xi|^2+2\varrho^{-1}\langle\xi,\mu\rangle+|\mu|^2}{2\lambda}\Bigl(1+\frac{s}{\lambda}\Bigr)^{-1},
$$
with $|s|/\lambda\lesssim 2^{-10C_0}\ll
C_0^{-1000}$ for large $C_0$. Elementary algebraic manipulations lead to \eqref{eq:exp-K} with
\begin{align*}
\mathscr{E}_\mu^{\lambda,\varrho}(\xi,s)=-\frac{|\xi|^2}{2\lambda\varrho^2}+\frac{|\xi|^2s}{2\lambda^2\varrho^2}+\frac{\langle\xi,\mu\rangle s}{\lambda^2\varrho}
-\frac{|\xi|^2s^2}{2\lambda^3\varrho^2}\mathcal{E}\Bigl(\frac{s}{\lambda}\Bigr)&\\
-\frac{\langle\xi,\mu\rangle s^2}{\varrho\lambda^3}\mathcal{E}\Bigl(\frac{s}{\lambda}\Bigr)-\frac{|\mu|^2s^2}{2\lambda^3}\mathcal{E}\Bigl(\frac{s}{\lambda}\Bigr)&.
\end{align*}
To treat $\mathscr{E}_\mu^{\lambda,\varrho}$ as an error term, we need the stability condition $\varrho\le \lambda$.\smallskip

Letting
$$
\mathscr{L}=\frac{1+(2\pi i)^{-1}\bigl(\varrho^{-1}(x-t\frac{\mu}{\lambda})+t\partial_\xi\mathscr{E}_\mu^{\lambda,\varrho}\bigr)\cdot\partial_\xi+(2\pi i)^{-1}(x_{n+1}+t\frac{|\mu|^2}{2\lambda^2}+t\partial_s\mathscr{E}_\mu^{\lambda,\varrho})\partial_s}{1+
	\bigl |\varrho^{-1}(x-t\,\frac{\mu}{\lambda})+t\,\partial_\xi\mathscr{E}_\mu^{\lambda,\varrho}\;\bigr|^2+\bigl|x_{n+1}+t\,\frac{|\mu|^2}{2\lambda^2}+t\,\partial_s\mathscr{E}_\mu^{\lambda,\varrho}\bigr|^2},
$$
such that  for any integer $M\ge 1$, we have $$\mathscr{L}^Me^{2\pi i\bigl(\varrho^{-1}(x-t\frac{\mu}{\lambda})\cdot \xi+(x_{n+1}+t\frac{|\mu|^2}{2\lambda^2})s+t\,\mathscr{E}_\mu^{\lambda,\varrho}\bigr)}=e^{2\pi i\bigl(\varrho^{-1}(x-t\frac{\mu}{\lambda})\cdot \xi+(x_{n+1}+t\frac{|\mu|^2}{2\lambda^2})s+t\,\mathscr{E}_\mu^{\lambda,\varrho}\bigr)}.$$
Noting that
\begin{align*}
1+
\bigl |\varrho^{-1}(x-t\frac{\mu}{\lambda})+t\,\partial_\xi\mathscr{E}_\mu^{\lambda,\varrho}(\xi,s)\;\bigr|+&\bigl|x_{n+1}+t\,\frac{|\mu|^2}{2\lambda^2}+t\,\partial_s\mathscr{E}_\mu^{\lambda,\varrho}(\xi,s)\bigr|\\
&\gtrsim
\bigl |\varrho^{-1}(x-t\,\frac{\mu}{\lambda})\bigr|+\bigl|x_{n+1}+t\,\frac{|\mu|^2}{2\lambda^2}\bigr|
\end{align*}
holds for all $(\xi,s)\in\supp \,p$
and all $|t|\lesssim \lambda \varrho^2$,
we have by the non-stationary phase argument ($M$-fold integration by parts, see also the formula for $\mathcal{K}_j^\lambda$ in the proof of Proposition \ref{pp:KKK} )
\begin{equation}
\label{eq:K-v-mu-pp}
\bigl| K^{\lambda,\varrho}_{t,\mu}(\xxx)\bigr|
\lesssim_M \;\varrho^{-n}\Bigl(1+\varrho^{-1}\bigl|x-t\frac{\mu}{\lambda}\bigr|+\bigl|x_{n+1}+t\,\frac{|\mu|^2}{2\lambda^2}\bigr|\Bigr)^{-M}.
\end{equation}
In other words, for $t$ being contained in an interval of length $ \O(\lambda\varrho^2)$, the kernel function  $K^{\lambda,\varrho}_{t,\mu}(\xxx)$ is concentrated on a $\underbrace{\varrho\times\cdots\times\varrho}_{n\text{ times}}\times1\times\varrho^2\lambda$
plate, denoted as $\mathfrak{V}^{\lambda,\varrho}_\mu$, which is oriented along the direction $(\frac{\mu}{\lambda},-\frac{|\mu|^2}{2\lambda^2},1)$
with thickness being approximately one in the $x_{n+1}$ direction and of width $\varrho$ in the circular directions. We call $\mathfrak{V}^{\lambda,\varrho}_\mu$ the  \emph{ concentration plate} for $K^{\lambda,\varrho}_{t,\mu}$ and it is clear that $\mathfrak{V}^{\lambda,\varrho}_\mu$ is contained in an $\O(\varrho)-$neighbourhood of $\mathbf{ \Lambda}^{\lambda}_j$ if we have the condition that $\mu\in\varGamma\cap \widetilde{V}_j$.\smallskip

Let $\eta_0\in\mathcal{ S}(\R)$ have the same property as $\mathbf{ \Upsilon}_0$ such that if we  put $\eta_k(s)=\eta_0(s-k)$, we have the partition of unity $\sum_{k\in\Z}\eta_k(s)=1$ for all $s\in\R$. Writing $f_{\vvv,\mu}(\xxx)=\sum_{k\in\Z}f_{\vvv,\mu,k}$ with $f_{\vvv,\mu,k}(\xxx):=\eta_k(x_{n+1})f_{\vvv,\mu}(\xxx)$, we have
$F^\lambda_{\vvv,\mu}(t)=\sum_{k\in\Z} F^\lambda_{\vvv,\mu,k}(t)$ with $F^\lambda_{\vvv,\mu,k}(t):=K^{\lambda,\varrho}_{t,\mu}*f_{\vvv,\mu,k}$\,.\smallskip

Now, we need the  plate maximal function
$$
f\mapsto \mathcal{M}^{\varpi,\varrho}f(\xxx)=\sup_{r>0} \frac{1}{|\mathcal{R}^{\varpi,\varrho}_r|}\int _{\mathcal{R}_r^{\varpi,\varrho}}
|f(\xxx-\xxx')|\,d\xxx'
$$
where
$$\mathcal{R}^{\varpi,\varrho}_r=\bigl\{(x_1,\ldots,x_{n+1})\in\R^{n+1};|x_1|,\ldots,|x_n|\le r \varpi^{-2}\varrho ,\,|x_{n+1}|\le r\bigr\}.$$
Let $f_\mu=\sum_{\vvv} f_{\vvv,\mu}$.
For each $\vvv=(\nu,\nu_{n+1})$ and $k\in\Z$, we define 
$$
c_{\vvv,\mu,k}:=
\Bigl(1+\varpi^2\varrho^{-1}|k-\nu_{n+1}|\Bigr)^{-B}
\mathcal{M}^{\varpi,\varrho}f_\mu(\nu,k),\;\,
\phi_{\vvv,\mu,k}(t):=\frac{1}{c_{\vvv,\mu,k}} F^\lambda_{\vvv,\mu,k}(t)\,.
$$
Then, we have for any $M\ge 1$
\begin{multline}
\label{eq:decay-wp}
|\phi_{\vvv,\mu,k}(\xxx,t)|\lesssim_{B,M} \varpi^{-O(M)}\Bigl(1+\varpi^2\varrho^{-1}|k-\nu_{n+1}|\Bigr)^{-B}\\
\times\Bigl(1+\varrho^{-1}\bigl|x-\nu-t\frac{\mu}{\lambda}\bigr|+\bigl|x_{n+1}-k+t\frac{|\mu|^2}{2\lambda^2}\bigr|\Bigr)^{-M}
\end{multline}
for all $\vvv,\mu,k$ and $\xxx$ and all $t$ which is contained in an interval of length $\lambda\varrho^2$.

The argument for \eqref{eq:decay-wp} is standard by using dyadic decomposition. We sketch it briefly.
Let $$(a^\lambda_\mu,b^\lambda_\mu)=\Bigl(x-\nu-t\lambda^{-1}\mu,\;x_{n+1}-k+\frac{t}{2}\lambda^{-2}|\mu|^2\Bigr).$$ 
Consider first $|k-\nu_{n+1}|\lesssim \varpi^{-2}\varrho$.
If $|a^\lambda_\mu|\lesssim \varpi^{-2}\varrho$ and $|b^\lambda_\mu|\lesssim \varpi^{-2}$, we use \eqref{eq:K-v-mu-pp} with a (different) sufficiently large $M$, and $F^\lambda_{\vvv,\mu,k}(t)=\int K^{\lambda,\varrho}_{t,\mu}(\xxx-\xxx')f_{\vvv,\mu,k}(\xxx')d\xxx'$ incorporated with the concentration property of $\mathbf{ \Upsilon}_{\vvv}$ that it is concentrated on a ball of radius $\varpi^{-2}\varrho$ centered at $\vvv$, and with $\eta_k$ on
an interval of length being roughly one centered at $k$. Here, these variables are referred to be the $\xxx'$ in the convolution $F^\lambda_{\vvv,\mu,k}$ and should not be confused with the fixed $\xxx=(x,x_{n+1})$ in $(a^\lambda_\mu,b^\lambda_\mu)$. Thus, in this case, \eqref{eq:decay-wp} follows from the trivial averaging argument by using the telescoping decomposition $$\R^{n+1}=\mathcal{R}^{\varpi,\varrho}_1\cup \bigcup_{h=1}^\infty\Bigl(\mathcal{R}^{\varpi,\varrho}_{2^h}\setminus\mathcal{R}^{\varpi,\varrho}_{2^{h-1}}\Bigr).$$ 
Next, consider $|a^\lambda_\mu|\gg \varpi^{-2}\varrho$ or $|b^\lambda_\mu|\gg \varpi^{-2}$. We only take the case $|a^\lambda_\mu|\gg \varpi^{-2}\varrho$ and $|b^\lambda_\mu|\lesssim \varpi^{-2}$ to illustrate the idea and the other cases are tackled in the same way. By Fubini theorem, we integrate first w.r.t. the $x-$component. Split the integration over $\R^n_{x}$ into the union of the ball $\{|x|\le  \varpi^{-2}\varrho\}$ and dyadic annuli $\{x\,:\,2^k\varpi^{-2}\varrho\le|x|\le 2^{k+1}\varpi^{-2}\varrho\}$ for $k\ge 1$.
Let $C\gg 1$ be a universal constant and consider $\mathscr{K}^{a^{\lambda}_\mu}:=\{k\ge 5C; \, 2^{k-C}\le \varpi^{2}\varrho^{-1} |a^\lambda_\mu|\le 2^{k+C}\}$ where  clearly we have $\mathrm{card}\,\mathscr{K}^{a^\lambda_\mu}\lesssim C$. For all $k\in\mathscr{K}^{a^\lambda_\mu}$, we use the fast decay of $\mathbf{ \Upsilon}_0$ fixed at the beginning of the proof and $2^k\approx_C |a^\lambda_\mu| \varpi^2 \varrho^{-1}$ to conclude the proof. For $k\not\in\mathscr{K}^{a^\lambda_\mu}$, consider if $2^k\le 2^{-C} |a^\lambda_\mu|\varrho^{-1}\varpi^2$, we use $|a^\lambda_{\mu}-x|\gtrsim |a^\lambda_\mu|$
and the rapid decay of $\mathbf{ \Upsilon}_0$ to conclude the proof; if $2^k\ge 2^{C}|a^\lambda_\mu|\varrho^{-1}\varpi^2$, then we use $|a^\lambda_\mu-x|\gtrsim |x|\gtrsim |a^\lambda_\mu|$ and the same argument as above to conclude the result. The same dyadic decomposition argument implies the desired result for the other two cases. For more details, one may consult \cite{TaoGFA,LeeTAMS,MA}. Consider next when $|k-\nu_{n+1}|\gg \varpi^{-2}\varrho$. Using the rapid decay of $\eta_0$ and $\mathbf{ \Upsilon}_0$ so that for any $B>0$, one can bring in a factor $\lesssim_B \Bigl(1+\varpi^2\varrho^{-1}|k-\nu_{n+1}|\Bigr)^{-10B}$ and the rest part of the proof is the same.
\medskip

To show  \eqref{eq:L-2-sum}, using $\supp\,\wh{f}_\mu(\xi,s)\subset\{(\xi,s); |\xi-\mu|\ll \varrho^{-1},\;s\in[-10,10]\}$ for all $\mu$, we claim that for any  fixed $C\ge 1$, one has
\begin{equation}
\label{eq:max-plate}
\mathcal{M}^{\varpi,\varrho}f_\mu(\nu,k) \le \varpi^{-O(1)} \mathcal{M}^{\varpi,\varrho}f_\mu(\xxx),
\end{equation}
for all $\xxx\in(\nu,k)+C \mathcal{R}_1^{\varpi,\varrho}$ and all $k\in\Z,\,(\vvv,\mu)\in\mathcal{ L}\times\varGamma$, where the implicit constant in $O(1)$ depends only on $C$ and $n$.

Squaring both sides of \eqref{eq:max-plate} and integrating on 
$(\nu,k)+ C\mathcal{R}_1^{\varpi,\varrho}$ then summing over $\vvv,k$,
we obtain \eqref{eq:L-2-sum} by the $L^2-$boundedness of mutli-parameter maximal functions over all rectangles with sides parallel to axes (c.f. Chapter 2 of Stein \cite{SteinHA}) and then summing over $\mu$, by Plancherel and almost orthogonality in the frequency space. When changing orders in summing over $\vvv,\mu,k$, one needs to take advantagne of the fact that when $k$ is at a distance $\approx 2^{\gamma}\varpi^{-2}\varrho$ away from $\nu_{n+1}$, for some $\gamma\ge 1$, there is a factor $2^{-B\gamma }$ with $B\gg 1$ that ensures the convergence of the geometric series. Thus, on each dyadic level $2^\gamma$,  one may classify  $\nu_{n+1}$  into arithmetic progressions of length $\approx2^{\gamma}$ so that the essential finite overlappedness occurs on each class. The $\O(2^\gamma)-$loss is eaten by $2^{-B\gamma}$ with $B\gg 1$.

It remains to show \eqref{eq:max-plate}, which is deduced  by the same argument of \cite{TaoGFA} based on the uncertainty principle. We leave the proof to Appendix \ref{sec:max}.\medskip

We next prove the Bessel type inequality \eqref{eq:bessel}. For any $\xi'\in\R^n$, define
$$
\mathscr{P}_{\mu,\varrho}^{\,\xi'} f(\xxx)=\iint
e^{2\pi i(x\cdot\xi+x_{n+1}s)}\indic_\Box(\varrho(\xi-\mu-\xi'))
\wh{f}(\xi,s)\;d\xi ds.$$
Then, $f_\mu(\xxx)$ is  the average of $\mathscr{P}_{\mu,\varrho}^{\,\xi'} f(\xxx)$
over $\varrho^{-1}\Box$ with respect to $\xi'$.
By Plancherel's theorem and Minkowski's inequality, we have
\begin{multline}
\label{eq:pf-bes}
\text{the left side of } \eqref{eq:bessel} \\
\le \varrho^n\int_{\varrho^{-1}\Box}\Biggl( \sum_{\varDelta}
\Bigl\|\sum_{(\vvv,\mu)\in\mathcal{ L}\times\varGamma} m^\varDelta_{\vvv,\mu} \mathbf{\Upsilon}_{\vvv}(\cdot)\, \mathscr{P}_{\mu,\varrho}^{\,\xi'}f(\cdot)\Bigr\|_{2}^2
\Biggr)^{\frac{1}{2}}d\xi'.
\end{multline}
For each  $\mu \in\varGamma$, define
$
\mathbb{B}_{\varrho,\mu}=\mu+\frac{1}{\varrho}\Box
$
and let
$$
\mathscr{O}=\bigcup_{\mu\in \varGamma}
\Bigl\{\xi\in\mathbb{B}_{\varrho,\mu};\; \mathsf{dist}(\xi, \R^n\setminus \mathbb{B}_{\varrho,\mu})\ge \varpi^2\varrho^{-1}\Bigr\}.
$$
For any $\xi'\in \varrho^{-1}\Box$, define
$$
\varPi_{\mathscr{O}+\xi'}:\;f(\xxx)\mapsto
\iint e^{2\pi i(x\cdot\xi+x_{n+1}s)}
\indic_{\{(\xi,s)\,;\;\xi\in\mathscr{O}+\xi'\}}(\xi,s)
\wh{f}(\xi,s)\;d\xi ds.
$$
Splitting $f=\bigl(\varPi_{\mathscr{O}+\xi'} f\bigr)+\bigl(\mathrm{id}-\varPi_{\mathscr{O}+\xi'} \bigr)f$ and using the triangle inequality, we have
$$\text{the right side of } \eqref{eq:pf-bes}  \le \mathbf{I}+\mathbf{II},$$where
\begin{align}
\label{eq:main}
\mathbf{I}=\,&\varrho^n\int_{\varrho^{-1}\Box}\Biggl( \sum_{\varDelta}
\Bigl\|\sum_{(\vvv,\mu)\in\mathcal{ L}\times\varGamma} m^\varDelta_{\vvv,\mu} \mathbf{\Upsilon}_{\vvv}(\cdot)\, \mathscr{P}_{\mu,\varrho}^{\,\xi'}\circ \varPi_{\mathscr{O}+\xi'} f(\cdot)\Bigr\|_{2}^2
\Biggr)^{\frac{1}{2}}d\xi',\\
\label{eq:error}
\mathbf{II}=\,&\varrho^n\int_{\varrho^{-1}\Box}\Biggl( \sum_{\varDelta}
\Bigl\|\sum_{(\vvv,\mu)\in\mathcal{ L}\times\varGamma} m^\varDelta_{\vvv,\mu} \mathbf{\Upsilon}_{\vvv}(\cdot)\, \mathscr{P}_{\mu,\varrho}^{\,\xi'}\circ\bigl( \mathrm{id}- \varPi_{\mathscr{O}+\xi'}\bigr) f(\cdot)\Bigr\|_{2}^2
\Biggr)^{\frac{1}{2}}d\xi'.
\end{align}

To deal with $\mathbf{I}$, we use the  Plancherel theorem and the strict orthogonality  from the pairwise $\varpi^{2}\varrho^{-1}$-separateness between the simply connected components of $\mathscr{O}$, which allows a petite amplification in the frequency space caused by  convolution with  $\wh{\mathbf{\Upsilon}}_\vvv$.
Note that the enlargement on the support of the $s-$variable in the frequency space does not affect the disjointness of the cylindrical sets $\{(\xi,s); \xi\in\mathbb{B}_{\varrho,\mu} ,\; \mathsf{dist}(\xi, \R^n\setminus \mathbb{B}_{\varrho,\mu})\ge \varpi^2\varrho^{-1}/100\}$ as $\mu$ ranges in $\varGamma$. We have
\begin{multline*}
\mathbf{I}\le\varrho^n\int_{\varrho^{-1}\Box}\Biggl( \sum_{\mu}
\Bigl\|\sum_{\varDelta,\vvv} m^\varDelta_{\vvv,\mu} \mathbf{\Upsilon}_{\vvv}(\cdot)\, \mathscr{P}_{\mu,\varrho}^{\,\xi'}\circ \varPi_{\mathscr{O}+\xi'} f(\cdot)\Bigr\|_{2}^2
\Biggr)^{\frac{1}{2}}d\xi'\\
\le \varrho^n\int_{\varrho^{-1}\Box}
\Biggl( \sum_{\mu}
\Bigl\|\, \mathscr{P}_{\mu,\varrho}^{\,\xi'}\circ \varPi_{\mathscr{O}+\xi'} f(\cdot)\Bigr\|_{2}^2
\Biggr)^{\frac{1}{2}}d\xi'\\
\le \varrho^n\int_{\varrho^{-1}\Box}\Biggl(\sum_{\mu}
\int\;\Bigl\|\indic_\Box\bigl(\varrho(\cdot-\mu-\xi')\bigr)
\wh{\varPi_{\mathscr{O}+\xi'}f}(\cdot,s)\Bigr\|^2_{L^2(\R^{n}_{\xi})}ds\Biggr)^{\frac{1}{2}}d\xi'
\le \|f\|_{L^2},
\end{multline*}
where we have used $\ell^1\subset\ell^2$ to get the first inequality and $$\sup_\mu\sum_{\varDelta,\vvv} m_{\vvv,\mu}^\varDelta \mathbf{\Upsilon}_\vvv\le\sum_{\vvv }\mathbf{\Upsilon}_\vvv\le 1,$$  for the second estimate. We used the strict orthogonality again in the last step.\smallskip

For $\mathbf{II}$, we have no strict orthogonality in the Fourier side anymore since the frequencies are located at the $\varpi^2\varrho^{-1}-$neighbourhood of the boundary inside $\mathbb{B}_{\varrho,\mu}$. In particular, no disjointness of the frequency variables can be used. Instead, by using the Plancherel theorem and the almost orthogonality followed with Cauchy-Schwarz as in dealing with $\mathbf{I}$ above, we have
\begin{align*}
\mathbf{II}\lesssim_n \varrho^n\int_{\varrho^{-1}\Box}\Biggl( \sum_{\mu}
\Bigl\|\sum_{\varDelta,\vvv} m^\varDelta_{\vvv,\mu} \mathbf{\Upsilon}_{\vvv}(\cdot)\, \mathscr{P}_{\mu,\varrho}^{\,\xi'}\circ \bigl(\mathrm{id}-\varPi_{\mathscr{O}+\xi'}\bigr) f(\cdot)\Bigr\|_{2}^2
\Biggr)^{\frac{1}{2}}d\xi' &\\
\lesssim_n \Biggl(\varrho^n\int_{\varrho^{-1}\Box}
\Bigl\|\bigl(\mathrm{id}-\varPi_{\mathscr{O}+\xi'}\bigr)f\Bigr\|^2_{L^2}d\xi'\Biggr)^{\frac{1}{2}}
\lesssim_n \varpi \|f\|_{L^2},&
\end{align*}
where in the last estimate we have used
the Fubini theorem and that
$$
\sup_{\xi\in\R^n}\,
\varrho^n\int_{\varrho^{-1}\Box}\Bigl(1-\indic_{\mathscr{O}+\xi'}(\xi)\Bigr)\;d\xi'\le\,C_n\,\varpi^2\,,
$$
which is
an obvious fact by noting that  $\xi'$  in the above integrand is restricted inside the intersection of a cube of size $\varrho^{-1}$ and an $\O(\varpi^2\varrho^{-1})-$neighbordhood of  $\cup_\mu\partial\,\mathbb{B}_{\varrho,\mu}$, union of the boundaries of the $\mathbb{B}_{\varrho,\mu}$'s. The proof is complete.
\end{proof}

\subsection{Construction of  the wave tables}

We set off the construction for the $S^\lambda-$version of  the wave table theory akin to  \cite{TaoMZ}.
The advantage of using wave tables is to eradicate the logarithmic loss arising from the repeatedly used dyadic pigeonhole principle \cite{TaoGFA,Wolff}, which blocked the approach to the endpoint results.

 From this section on, we start  adopting new notations $F^\lambda$ and $G^\lambda$  to denote respectively the red and blue waves. We apply  Lemma \ref{lem:wp-d} with $\varrho=R^{1/2}$
$$
F^\lambda(\xxx,t)=\sum_{(\vvv,\mu)\in\mathcal{ L}\times\varGamma_1} F_{\vvv,\mu}^\lambda(\xxx,t),\quad
G^\lambda(\xxx,t)=\sum_{(\vvv,\mu)\in\mathcal{ L}\times\varGamma_2} G_{\vvv,\mu}^\lambda(\xxx,t),
$$
where  for each $j\in\{1,2\}$, $\varGamma_j:=\varGamma\cap \widetilde{V}_j$
and  $t $ is always assumed to be contained in an interval of length $\O(\lambda R)$. Moreover, for each $(\vvv,\mu)\in \mathcal{ L}\times \varGamma_1$, the wave packet $F_{\vvv,\mu}^\lambda=\sum_k F^\lambda_{\vvv,\mu,k}$ (and similarly for $G_{\vvv,\mu}^\lambda$) is tightly concentrated  on a tube $T^\lambda_{\vvv,\mu}$, which is the union of the plates $\mathfrak{ V}^\lambda_{\nu,\mu,k}$, on which $F^\lambda_{\vvv,\mu,k}$ is concentrated in the sense of \eqref{eq:decay-wp-0}, for $k$ satisfying $|k-\nu_{n+1}|\lesssim \varpi^{-2}\varrho$. 
For each of those $k$'s with the property that $|k-\nu_{n+1}|\gg \varpi^{-2}\varrho$, $F^\lambda_{\vvv,\mu,k}$ decreases very fast in the wave envelope $F^\lambda_{\vvv,\mu}=\sum_k F^\lambda_{\vvv,\mu,k}$ in view of the $\ell^2-$summability \eqref{eq:L-2-sum} and \eqref{eq:decay-wp-0}.

Note that each $\mathfrak{ V}^\lambda_{\vvv,\mu,k}$ is oriented 
in the direction of $(\frac{\mu}{\lambda},-\frac{|\mu|^2}{2\lambda^2},1)$, parallel to $T^\lambda_{\vvv,\mu}$  for the fixed $\mu$. We shall say $T^\lambda_{\vvv,\mu}$ is \emph{parametrized by} $(\vvv,\mu)$. Moreover, each tube $T^\lambda_{\vvv,\mu}$ is of dimensions $\underbrace{\sqrt{R}\times\cdots\times\sqrt{R}}_{(n+1)\,\text{times}}\times \lambda R$.

Denoting $F^\lambda_T=F^\lambda_{\vvv,\mu}$ with $T=T^\lambda_{\vvv,\mu}$, we  rewrite the  decomposition for $F^\lambda$ into the form  $F^\lambda=\sum_{T_1\in\mathbf{T}_1}F^\lambda_{T_1}$,
where  $\mathbf{T}_1$ is the    the collection of $T_1$ tubes   associated to the red waves in the above sense.
For each $T_1$,  we write $$F^\lambda_{T_1}=\sum_{\mathfrak{ V}^\lambda_{\vvv,\mu,k}\subset T_1}F^\lambda_{\vvv,\mu,k}+\sum_{\mathfrak{ V}^\lambda_{\vvv,\mu,k}\not\subset T_1}F^\lambda_{\vvv,\mu,k}:=F^{\lambda, g}_{T_1}+F^{\lambda,b}_{T_1},$$ 
where $F^{\lambda,b}_{T_1}$ is the (global) part corresponding to the Schwartz tails. Similarly, we have the decomposition for the blue wave $G^\lambda=\sum_{T_2\in\mathbf{T}_2}G^\lambda_{T_2}$
with the local/global decomposition $G^\lambda_{T_2}=G_{T_2}^{\lambda,g}+G_{T_2}^{\lambda,b}$ for each $T_2\in\mathbf{T}_2$. See also Section 3 of \cite{Wolff} for the same decomposition. \medskip

For each $T_j\in\mathbf{T}_j$ with $j\in\{1,2\}$, we  use $\psi_{T_j}$ to denote the bump function
$$\psi_{T_j}(\xxx,t)=\min\bigl\{1, \mathsf{dist}((\xxx,t),T_j)^{-N}\bigr\}$$ adapted to $T_j$.\smallskip

When we say a spacetime cube, we mean a  cube with sides parallel to the axes. A $\lambda-$\emph{stretched}
cube of size $R$ is a cube $Q^\lambda_R\subset\R^{n+2}_{\xxx,t}$ such that the length  in the vertical direction equals to $\lambda R$, \emph{i.e.} along the $t-$axis, and having sides $R$ in the horizontal components $\R^{n+1}_{\xxx}$.
We denote  $\mathcal{Q}_{C_0}(Q_R^\lambda)$ to be the cubes obtained by bisecting each  side of $Q_R^\lambda$
consecutively such that  every $\varDelta\in \mathcal{Q}_{C_0}(Q_R^\lambda)$ is a $\lambda-$stretched
cube of size $2^{-C_0}R$.
\medskip

Fix $\varDelta \in \mathcal{Q}_{C_0}(Q_R^\lambda)$, let
$\mathsf{K}_{Q^\lambda_R}(\varDelta)=\bigl\{\qq\subset \varDelta;\,\qq\in\mathcal{Q}_{J}(Q^\lambda_R)\bigr\}$ with $J\approx \log R$. Each $\qq$ is a $\lambda-$stretched cube of size $\approx\sqrt{R}$. Let $\chi\in \mathcal{S}(\R^{n+2})$ be such that $\wh{\chi}$ is compactly supported in a small neighbourhood of the origin and $\chi\ge 1$ on double of the unit ball. Let $\mathcal{A}_\qq$ be the affine transform sending the John ellipsoid inside  $\qq$ to the unit ball such that  if we let  $\chi_\qq=\chi\circ \mathcal{A}_\qq$, then we have $\chi_\qq\ge \indic_\qq$, where $\indic_\qq$ is the characteristic function of $\qq$\,.
\begin{definition}
	Let $Q=Q^\lambda_R$. For each $\varDelta\in\mathcal{Q}_{C_0}(Q)$ and $T_1\in \mathbf{T}_1$, define
	$$
	m_{T_1}^{G^\lambda,\,\varDelta}=\sum_{\qq\in\mathsf{K}_{Q}(\varDelta)}\sum_{T_2\in\mathbf{T}_2}\bigl\|\chi_\qq\, \psi_{T_1}\,\psi_{T_2}^{-50}\,G_{T_2}^\lambda\bigr\|^2_{L^2(\R^{n+2})},
	$$
	and set $m_{T_1}^{G^\lambda}=\sum_{\varDelta\in\mathcal{Q}_{C_0}(Q)}m_{T_1}^{G^\lambda,\,\varDelta}$. The $(\lambda,\varpi,R^{1/2})-$\emph{wave table} $\mathcal{F}^\lambda$ for $F^\lambda$
	with respect to $G^\lambda$	over $Q$ is defined as the vector-valued function
	$$
	\mathcal{F}^\lambda=\mathcal{F}^\lambda_{\varpi,R^{1/2}}
	=\bigl(\mathcal{F}^{\lambda,\,\varDelta}_\varpi\bigr)_{\varDelta\in\mathcal{Q}_{C_0}(Q)},
	$$
	with
	$$
	\mathcal{F}_\varpi^{\lambda,\varDelta}(\xxx,t):=\sum_{T_1\in\mathbf{T}_1}
	\frac{m_{T_1}^{G^\lambda,\,\varDelta}}{m_{T_1}^{G^\lambda}}F_{T_1}^\lambda(\xxx,t).
	$$
	Similarly, we define the $(\lambda,\varpi,R^{1/2})-$\emph{wave table} $ \mathcal{G}^\lambda$ for $G^\lambda$
	with respect to $F^\lambda$	over $Q$ in the symmetric way:
	$$
	\mathcal{G}^\lambda=\mathcal{G}^\lambda_{\varpi,R^{1/2}}
	=\bigl(\mathcal{G}^{\lambda,\,\varDelta}_\varpi\bigr)_{\varDelta\in\mathcal{Q}_{C_0}(Q)},
	$$
	with
	$$
	\mathcal{G}_\varpi^{\lambda,\varDelta}(\xxx,t):=\sum_{T_2\in\mathbf{T}_2}
	\frac{m_{T_2}^{F^\lambda,\,\varDelta}}{m_{T_2}^{F^\lambda}}G_{T_2}^\lambda(\xxx,t),
	$$
	where
	$$
	m_{T_2}^{F^\lambda,\,\varDelta}=\sum_{\qq\in\mathsf{K}_{Q}(\varDelta)}\sum_{T_1\in\mathbf{T}_1}\bigl\|\chi_\qq\, \psi_{T_2}\,\psi_{T_1}^{-50}\,F_{T_1}^\lambda\bigr\|^2_{L^2(\R^{n+2})},
	$$
	and  $m_{T_2}^{F^\lambda}=\sum_{\varDelta\in\mathcal{Q}_{C_0}(Q)}m_{T_2}^{F^\lambda,\,\varDelta}$.
\end{definition}

Clearly, we have
$$
F^\lambda=\sum_{\varDelta\in\mathcal{Q}_{C_0}(Q)} \mathcal{F}^{\lambda,\,\varDelta}_\varpi\,,\quad
G^\lambda=\sum_{\varDelta\in\mathcal{Q}_{C_0}(Q)} \mathcal{G}^{\lambda,\,\varDelta}_\varpi\,.
$$
\begin{remark}
	Note that the weights $m^{G^\lambda,\varDelta}_{T_1}$ and $m^{F^\lambda,\varDelta}_{T_2}$  appear  different from that of \cite{TaoMZ} but  closer to the form of  those in \cite{Lee21}. However, these two forms are essentially equivalent and the definition we adopt here is more convenient when dealing with the paraboloid. We have also refined this definition by inserting $\psi_{T_j}^{-50}$ to the $L^2-$integral for a technical reason.
\end{remark}
\begin{remark}
	To lighten notations, we will suppress the subscript $\varpi$ and omit the $G^\lambda, F^\lambda$ on the shoulders of $m^{G^\lambda,\varDelta}_{T_1}$ and  $m^{F^\lambda,\varDelta}_{T_2}$ respectively. The dependence on various of these parameters will be clear from the context.
\end{remark}

Note that by the linearity of the operator $S^\lambda(t)$, for any $\varDelta$, we find that $\mathcal{F}^{\lambda,\varDelta}$ and $\mathcal{G}^{\lambda,\varDelta}$ are red and blue waves respectively, and one may define the energy  $\EEE(\mathcal{F}^{\lambda,\varDelta})$ and $\EEE(\mathcal{G}^{\lambda,\varDelta})$ as in the beginning of  Section \ref{sec:pre+top}.
By using the Bessel type inequality \eqref{eq:bessel}, we have
\begin{lemma}
	\label{lem:bessel-w-t}
	There is a constant $C_n$ depending only on $n$ such that we have
	\begin{align}
  \EEE(\mathcal{F}^\lambda)^{1/2}:=	\Bigl(\sum_{\varDelta\in\mathcal{Q}_{C_0}(Q)} \EEE\bigl(\mathcal{F}^{\lambda,\varDelta}\bigr)\Bigr)^\frac{1}{2}\le& \;(1+C_n\,\varpi) \EEE(F^\lambda)^{1/2}\;,\\
	\EEE(\mathcal{G}^\lambda)^{1/2}:=	\Bigl(\sum_{\varDelta\in\mathcal{Q}_{C_0}(Q)} \EEE\bigl(\mathcal{G}^{\lambda,\varDelta}\bigr)\Bigr)^\frac{1}{2}\le& \;(1+C_n\,\varpi) \EEE(G^\lambda)^{1/2}\;,
	\end{align}
	for any $(\lambda,\varpi,R^{1/2})$ wave tables $\mathcal{F}^\lambda,\mathcal{G}^\lambda$ over a spacetime cube $Q$.
\end{lemma}
Next, we define the $C_0-$\emph{quilts} of $\mathcal{F}^\lambda$ and $\mathcal{G}^\lambda$ on $Q=Q^\lambda_R$ as
$$
\bigl[\mathcal{F}^\lambda\bigr]_{C_0}=\sum_{\varDelta\in\mathcal{Q}_{C_0}(Q)}\indic_\varDelta\, \mathcal{F}^{\lambda,\,\varDelta}\,,\quad
\bigl[\mathcal{G}^\lambda\bigr]_{C_0}=\sum_{\varDelta\in\mathcal{Q}_{C_0}(Q)}\indic_\varDelta\, \mathcal{G}^{\lambda,\,\varDelta}\,.
$$
The \emph{$(\varpi,C_0)-$interior} of $Q$ is defined as
$$
\mathfrak{I}^{\varpi,\,C_0}(Q)=\bigcup_{\varDelta\in\mathcal{Q}_{C_0}(Q)}(1-\varpi)\varDelta.
$$
Here $(1-\varpi)\varDelta$ is the stretched cube of the same center with $\varDelta$, but with side length multiplied by the constant $(1-\varpi)$ with $0<\varpi\ll 1$.
\smallskip

These cubes play a crucial role to obtain the effective approximation to the product of red and blue waves via the $C_0-$quilts.\smallskip

Let $z_0=(\xxx_0,t_0)\in \R^{n+2}$ and define
the  conic set
$$\mathcal{C}^\lambda(z_0,r)=\mathbf{ \Lambda}_1^\lambda(z_0,r)\cup \mathbf{ \Lambda}_2^\lambda(z_0,r),$$
with $\mathbf{ \Lambda}_j^\lambda(z_0,r)$  given in Lemma \ref{lem:opposite}  for $j=1,2$.
Let $X_{z_0}^{\varpi,r}(Q)=\mathfrak{I}^{\varpi,C_0}(Q)\cap \,\mathcal{C}^\lambda(z_0,r)$.
\smallskip

For any $u\in L^\infty_{loc}(\R^{n+1}_{\xxx}\times\R_t)$ and any measurable subset $\Omega\subset\R^{n+2}$, such that $\Omega=\bigcup_{t\in \mathcal{I}}\bigl(\varPi_{t}\times \{t\}\bigr)$ for some $\mathcal{I}\subset\R$, we denote
$$
\|u\|_{Z(\Omega)}:=\Bigl(\int_{\mathcal{I}}\Bigl(\int_{\varPi_{t}}|u(\xxx,t)|^{s}d\xxx\Bigr)^\frac{q}{s}dt\Bigr)^{\frac{1}{q}},
$$
with  $(q,s)=\bigl(q_c^+,r_c^-\bigr)\in\mathbf{\Gamma}$,
where for any $\gamma\in\R$, we denote $\gamma^+$ ( resp. $\gamma^-$) as a real number greater (resp. smaller) than  but  sufficiently close to $\gamma$. Thus, the sense of $Z\bigl(\mathfrak{I}^{\varpi,C_0}(Q)\bigr)$
and $Z\bigl(X_{z_0}^{\varpi,r}(Q)\bigr)$ is clearly understood.
\medskip

The effective approximation of $\|F^\lambda G^\lambda\|_{Z(Q)}$ via $C_0-$quilts below plays a fundamental role in the endpoint theory of bilinear estimates \cite{TaoMZ}.

\begin{prop}
	\label{pp:C-0 quilt}
	For any $R\in [2^{10 C_0},\lambda]$ , $\varpi\in (0, 2^{-C_0}]$ and $Q=Q^\lambda_R$, there exists a  constant $C$, depending only on $n$ and independent of $C_0$,
	such that if $F^\lambda$ and $G^\lambda$ are red and blue waves with $\EEE(F^\lambda)=\EEE(G^\lambda)=1$, and $\mathcal{F}^\lambda,\mathcal{G}^\lambda$ are the $(\lambda,\varpi,R^{1/2})-$wave tables for  $F^\lambda$ and $G^\lambda$  over $Q^*:=CQ$ respectively,   we have
	\begin{equation}
	\label{eq:app-1}
	\|F^\lambda G^\lambda\|_{Z(Q)}\le (1+C\varpi)\bigl\|\bigl[\mathcal{F}^\lambda\bigr]_{C_0}\bigl[\mathcal{G}^\lambda\bigr]_{C_0}\bigr\|_{Z(\mathfrak{I}^{\varpi,C_0}(Q^*))}+\lambda^{\frac{1}{q}}\varpi^{-O(1)}
	\end{equation}
	and there exists $\kappa=\kappa_Z>0$ such that
		\begin{multline}
	\label{eq:app-2}
	\|F^\lambda G^\lambda\|_{Z(Q\cap\, \mathcal{C}^\lambda(z_0,r))}\le\, (1+C\varpi)\,\bigl\|\bigl[\mathcal{F}^\lambda\bigr]_{C_0}\bigl[\mathcal{G}^\lambda\bigr]_{C_0}\bigr\|_{Z(X_{z_0}^{\varpi,r}(Q^*))}\\
	+\lambda^{\frac{1}{q}}\;\varpi^{-O(1)}\;\Bigl(1+\frac{R}{r}\Bigr)^{-\kappa},
	\end{multline}
	holds for all $z_0\in\R^{n+2}$.
\end{prop}
 \begin{remark}
 	As pointed out in \cite{TaoMZ}, this result is a pigeonhole-free version of the arguments in Wolff \cite{Wolff}. For this reason, we refer to the method of \cite{TaoMZ} as a profound version of the induction on scale argument.
 \end{remark}
\begin{proof}
	\noindent\emph{Step 1.  Reduction to a tamed bilinear $L^2-$Kakeya type estimate.}
	We omit $z_0$ in $X^{\varpi,r}_{z_0}$ for brevity.
	By the averaging argument in \cite{TaoMZ,LeeVargas,Temur}, there is a  universal constant $C$  such that we have
	$$
		\|F^\lambda G^\lambda\|_{Z(Q)}\le (1+C\varpi)\bigl\|F^\lambda G^\lambda\bigr\|_{Z(\mathfrak{I}^{\varpi,C_0}(Q^*))}
	$$
	$$
	\|F^\lambda G^\lambda\|_{Z(Q\cap\mathcal{C}^\lambda(z_0,r))}\le (1+C\varpi)\bigl\|F^\lambda G^\lambda\bigr\|_{Z(X^{\varpi,r}(Q^*))}.
	$$
	To get these two estimates, we note that the same method in \cite{TaoMZ} is directly applied to the stretched cubes in symmetric norms, and the mixed norms follows from the duality argument \cite{LeeVargas}.\smallskip

	Writing
	\begin{multline*}
	F^\lambda G^\lambda=\bigl[\mathcal{F}^\lambda\bigr]_{C_0}\bigl[\mathcal{G}^\lambda\bigr]_{C_0}
	+\bigl[\mathcal{F}^\lambda\bigr]_{C_0}\bigl(G^\lambda-\bigl[\mathcal{G}^\lambda\bigr]_{C_0}\bigr)+\bigl(F^\lambda-\bigl[\mathcal{F}^\lambda\bigr]_{C_0}\bigr)\,{G}^\lambda\;,
	\end{multline*}
	we are reduced to 
	\begin{equation}
	\label{eq:red-1}
	\bigl\| \bigl[\mathcal{F}^\lambda\bigr]_{C_0}\bigl(G^\lambda-\bigl[\mathcal{G}^\lambda\bigr]_{C_0}\bigr)\bigr\|_{Z(\mathfrak{I}^{\varpi,C_0}(Q^*))}\le\,\varpi^{-O(1)}\lambda^{\frac{1}{q}}\,,
	\end{equation}
	\begin{equation}
	\label{eq:red-2}
	\bigl\| \bigl(F^\lambda-\bigl[\mathcal{F}^\lambda\bigr]_{C_0}\bigr)\,{G}^\lambda\bigr\|_{Z(\mathfrak{I}^{\varpi,C_0}(Q^*))}\le\,\varpi^{-O(1)}\,\lambda^{\frac{1}{q}},
\end{equation}
	\begin{equation}
\label{eq:red-3}
\bigl\| \bigl[\mathcal{F}^\lambda\bigr]_{C_0}\bigl(G^\lambda-\bigl[\mathcal{G}^\lambda\bigr]_{C_0}\bigr)\bigr\|_{Z(X^{\varpi,r}(Q^*))}\le\,\varpi^{-O(1)}\Bigl(1+\frac{R}{r}\Bigr)^{-\kappa}\lambda^{\frac{1}{q}}\,,
\end{equation}
\begin{equation}
\label{eq:red-4}
\bigl\| \bigl(F^\lambda-\bigl[\mathcal{F}^\lambda\bigr]_{C_0}\bigr)\,{G}^\lambda\bigr\|_{Z(X^{\varpi,r}(Q^*))}\le\,\varpi^{-O(1)}\,\Bigl(1+\frac{R}{r}\Bigr)^{-\kappa}\lambda^{\frac{1}{q}}\,.
\end{equation}
By symmetry, we only show \eqref{eq:red-1}
and \eqref{eq:red-3}. By Minkowski inequality, we are reduced to showing that
	\begin{equation}
\label{eq:red-1'}
\max_{\varDelta,\varDelta'\in\mathcal{Q}_{C_0}(Q^*)}\bigl\| \mathcal{F}^{\lambda,\varDelta}\, \mathcal{G}^{\lambda,\varDelta'}\bigr\|_{Z(\mathfrak{I}^{\varpi,C_0}(Q^*)\setminus \{\varDelta'\})}\le\,\varpi^{-O(1)}\lambda^{\frac{1}{q}}\,,
\end{equation}
and
	\begin{equation}
\label{eq:red-3'}
\max_{\varDelta,\varDelta'\in\mathcal{Q}_{C_0}(Q^*)}\
\bigl\| \mathcal{F}^{\lambda,\varDelta}\, \mathcal{G}^{\lambda,\varDelta'}\bigr\|_{Z(X^{\varpi,r}(Q^*)\setminus\{\varDelta'\})}\le\,\varpi^{-O(1)}\Bigl(1+\frac{R}{r}\Bigr)^{-\kappa}\lambda^{\frac{1}{q}}\,.
\end{equation}
\subsection*{Claim} If
\begin{equation}
\label{eq:bil-kakeya}
\max_{\varDelta,\varDelta'\in\mathcal{Q}_{C_0}(Q^*)}
\bigl\|\mathcal{F}^{\lambda,\varDelta}\, \mathcal{G}^{\lambda,\varDelta'}\bigr\|_{L^2_{t,\xxx}(\mathfrak{I}^{\varpi,C_0}(Q^*)\setminus\{\varDelta'\})}\le\,\varpi^{-O(1)} \lambda^{1/2}R^{-\frac{n-1}{4}},
\end{equation}
then
\eqref{eq:red-1'} and \eqref{eq:red-3'} hold.\\

To show this claim, using Cauchy-Schwarz inequalities and Lemma \ref{lem:bessel-w-t}, we get
	\begin{equation}
\label{eq:red-1'-energy}
\max_{\varDelta,\varDelta'\in\mathcal{Q}_{C_0}(Q^*)}\bigl\| \mathcal{F}^{\lambda,\varDelta}\, \mathcal{G}^{\lambda,\varDelta'}\bigr\|_{L^q_tL^1_x\bigl(\mathfrak{I}^{\varpi,C_0}(Q^*)\bigr)}\lesssim\,(\lambda R)^{1/q}\,.
\end{equation}
Similarly, applying Lemma \ref{lem:opposite} to $\mathcal{F}^{\lambda,\varDelta}$, $\mathcal{G}^{\lambda,\varDelta'}$ in place of $F^\lambda_1, F^\lambda_2$ there and Minkowski inequality, we have by interpolation with the energy estimates in Lemma \ref{lem:bessel-w-t}
	\begin{equation}
\label{eq:red-3'-energy}
\max_{\varDelta,\varDelta'\in\mathcal{Q}_{C_0}(Q^*)}\bigl\| \mathcal{F}^{\lambda,\varDelta}\, \mathcal{G}^{\lambda,\varDelta'}\bigr\|_{L^q_tL^1_x\bigl(X^{\varpi,r}(Q^*)\bigr)}\lesssim\,(\lambda R)^{1/q}\Bigl(1+\frac{R}{r}\Bigr)^{-\frac{1}{2q}}\,.
\end{equation}
From \eqref{eq:bil-kakeya} and H\"older inequality,  we have
\begin{equation}
\label{eq:good-bi}
\max_{\varDelta,\varDelta'\in\mathcal{Q}_{C_0}(Q^*)}\bigl\| \mathcal{F}^{\lambda,\varDelta} \mathcal{G}^{\lambda,\varDelta'}\bigr\|_{L^q_tL^2_x\bigl(\mathfrak{I}^{\varpi,C_0}(Q^*)\setminus\{\varDelta'\}\bigr)}\le\,\varpi^{-O(1)}\lambda^{1/q} R^{\frac{1}{q}-\frac{n+1}{4}}\,.
\end{equation}
Interpolating \eqref{eq:good-bi} with \eqref{eq:red-1'-energy} and \eqref{eq:red-3'-energy} respectively, we obtain \eqref{eq:red-1'} and \eqref{eq:red-3'}, where $\kappa=\frac{1}{q}\bigl(\frac{1}{s}-\frac{1}{2}\bigr)>0$ with $1<s<2$. Thus the claim is verified.\medskip

It remains to prove \eqref{eq:bil-kakeya}, to which we refer as a tamed bilinear $L^2$-Kakeya type estimate
	for except $\varpi^{-O(1)}$, no $\log R$ -loss involved compared to \cite{TaoGFA}. See \cite{TaoMZ} for the cone case.
\medskip
	
	\noindent\emph{Step 2. Bilinear $L^2-$reduction for \eqref{eq:bil-kakeya}.} For any $\varDelta,\varDelta'$,  by definition of $\mathfrak{I}^{\varpi,C_0}(Q^*)$
	\begin{multline*}
	\bigl\|\mathcal{F}^{\lambda,\varDelta} \mathcal{G}^{\lambda,\varDelta'}\bigr\|_{L^2_{t,\xxx}(\mathfrak{I}^{\varpi,C_0}(Q^*)\setminus\{\varDelta'\})}^2
	\lesssim
	\sum_{\varDelta''\in\mathcal{Q}_{C_0}(Q^*)\setminus\{\varDelta'\}}\sum_{\qq\in\mathsf{K}_{Q^*}(\varDelta'')}
	\bigl\|\mathcal{F}^{\lambda,\varDelta} \mathcal{G}^{\lambda,\varDelta'}\bigr\|_{L^2(\qq)}^2\\
	\lesssim 2^{O(C_0)} \max_{\varDelta''\in\mathcal{Q}_{C_0}(Q^*)\setminus\{\varDelta'\}}\sum_{\qq\in\mathsf{K}_{Q^*}(\varDelta'')}
	\bigl\|\mathcal{F}^{\lambda,\varDelta} \mathcal{G}^{\lambda,\varDelta'}\bigr\|_{L^2(\qq)}^2\;.
	\end{multline*}
	 Recall
	 $$
	 \mathcal{F}^{\lambda,\varDelta}(\xxx,t)=\sum_{T_1\in\mathbf{T}_1}
	 \frac{m_{T_1}^{\varDelta}}{m_{T_1}}F_{T_1}^\lambda(\xxx,t),\;\;
	  \mathcal{G}^{\lambda,\varDelta'}(\xxx,t)=\sum_{T_2\in\mathbf{T}_2}
	 \frac{m_{T_2}^{\varDelta'}}{m_{T_2}}G_{T_2}^\lambda(\xxx,t).	
	 $$
	 For any $\qq\in\mathsf{K}_{Q^*}(\varDelta'')$ with $\varDelta''\in\mathcal{Q}_{C_0}(Q^*)\setminus\{\varDelta'\}$,	 we have
	 \begin{align*}
	  \bigl\|\mathcal{F}^{\lambda,\varDelta} \mathcal{G}^{\lambda,\varDelta'}\bigr\|_{L^2(\qq)}^2\lesssim&\;
	  \Biggl\|\sum_{\substack{T_1\in\mathbf{T}_1,T_2\in\mathbf{T}_2\\ T_1\cap 100CQ^\lambda_R\neq\emptyset, T_2\cap 100CQ^\lambda_R\neq\emptyset}}\frac{m_{T_1}^{\varDelta}}{m_{T_1}}F_{T_1}^\lambda\frac{m_{T_2}^{\varDelta'}}{m_{T_2}}G_{T_2}^\lambda\Biggr\|_{L^2(\qq)}^2\\
	  &+\Biggl\|\sum_{\substack{T_1\in\mathbf{T}_1,T_2\in\mathbf{T}_2\\ T_1\cap 100CQ^\lambda_R=\emptyset, \text{ or }T_2\cap 100CQ^\lambda_R=\emptyset}}\frac{m_{T_1}^{\varDelta}}{m_{T_1}}F_{T_1}^\lambda\frac{m_{T_2}^{\varDelta'}}{m_{T_2}}G_{T_2}^\lambda\Biggr\|_{L^2(\qq)}^2,
	 \end{align*}
	where the second term is bounded by $\O(R^{-N})$
	using the rapid decay of the wave packets $F_{T_1}^\lambda=F^{\lambda,g}_{T_1}+F^{\lambda,b}_{T_1}, G^{\lambda}_{T_2}=G^{\lambda,g}_{T_2}+ G^{\lambda,b}_{T_2}$ away from $T_1$ and $T_2$ along with the $\ell^2-$summation \eqref{eq:L-2-sum}. Indeed, if $T_1\cap 100 CQ^\lambda_R=\emptyset$, then on any $\qq\subset CQ^\lambda_R$, we have $|F^{\lambda,g}_{T_1}|\lesssim_N R^{-N}$ since each $\mathfrak{ V}^\lambda\subset T_1$ is at a distance $\gtrsim R$ away from $\qq$. The $F^{\lambda,b}_{T_1}$ term is small as well by using the additional decaying factor in \eqref{eq:decay-wp-0} and the above distance condition. The same argument applies to $G^\lambda_{T_2}$.\smallskip
	
	For the first term, squaring it out, we are led to estimating
	\begin{multline*}
	\sum_{\substack{\qq\in\mathsf{K}_{Q^*}(\varDelta'')\\ \varDelta''\in \mathcal{Q}_{C_0}(Q^*)\setminus\{\varDelta'\} }}\;
	\sum_{\substack{(T_1,T_2),(\bar{T}_1,\bar{T}_2)\in\mathbf{T}_1\times\mathbf{T}_2 \\
T_j\cap 100CQ^\lambda_R\neq\emptyset,\,\bar{T}_j\cap 100CQ^\lambda_R\neq\emptyset,\,j=1,2}}
\frac{m_{T_1}^{\varDelta}\;m_{T_2}^{\varDelta'}\;m_{\bar{T}_1}^{\varDelta}\;m_{\bar{T}_2}^{\varDelta'}}{m_{T_1}\;m_{T_2}\;m_{\bar{T}_1}\;m_{\bar{T}_2}}\\
\underbrace{\iint \chi_\qq^8(\xxx,t)\; F_{T_1}^\lambda(\xxx,t)\;
G^\lambda_{T_2}(\xxx,t)\,\;\overline{F^\lambda_{\bar{T}_1}(\xxx,t)}\;\overline{G^\lambda_{\bar{T}_2}(\xxx,t)}\;d\xxx dt}_{:=\mathcal{I}^{\lambda}(\qq,T_1,T_2,\bar{T}_1,\bar{T}_2)}.
	\end{multline*}
	 For each $\qq$ and $T_1,T_2,\bar{T}_1,\bar{T}_2$ in the above summand,  $\mathcal{I}^\lambda(\qq,T_1,T_2,\bar{T}_1,\bar{T}_2)$ equals to
	\begin{equation}
	\label{eq:qud}
	\iint_{\R^{n+2}}\wh{\chi_\qq^2 F^\lambda_{T_1}}*\wh{\chi_\qq^2 G^\lambda_{T_2}}(\xi,s,\tau)\;
	\overline{\wh{\chi_\qq^2 F^\lambda_{\bar{T}_1}}*\wh{\chi_\qq^2 G_{\bar{T}_2}^\lambda}(\xi,s,\tau)}\; d\xi dsd\tau
	\end{equation}
	by using Parseval's identity.

  For any $T$,  belonging either to $\mathbf{T}_1$ or $\mathbf{T}_2$, we let $\mu_T$ be such that
	 the direction of $T$ is given by  $(\frac{\mu_{T}}{\lambda},-\frac{|\mu_T|^2}{2\lambda^2},1)$ and define
	$$
	\mathfrak{J}_{\mu_T,\lambda}=\Bigl\{(\xi,\tau)\,;\;\; \xi=\mu_T+\O(R^{-1/2}),\; \tau=-\frac{|\mu_T|^2}{2\lambda}+\O\bigl({\lambda}^{-1}R^{-1/2}\bigr)\Bigr\}.
	$$
	Then, 	using $\supp(u*v)\subset\supp(u)+\supp(v)$ for any distributions $u$ and $v$, it is easy to see  that by adjusting the implicit constant in the perturbation terms $\O(\cdot\cdot\cdot)$ if necessary, $\wh{\chi_\qq^2 F^\lambda_{T}}(\xi,s,\tau)$
	 vanishes if $(\xi,\tau)\not\in\mathfrak{J}_{\mu_T,\lambda}$ for $T$ being $T_1,\bar{T}_1$. Likewise for $\wh{\chi_\qq^2 G^\lambda_{T}}(\xi,s,\tau)$
	 with $T$ being $T_2,\bar{T}_2$. In order to see this, we note that the support of $\wh{\chi}_\qq$ is contained in a plate of dimensions $R^{-1/2}\times\cdots \times R^{-1/2}\times \lambda^{-1}R^{-1/2}$ and use a similar  expansion of the phase function as in the proof of the wavepacket deomposition once more with $R\le\lambda^2$.\smallskip

	The condition $\mathcal{I}^\lambda(\qq,T_1,T_2,\bar{T}_1,\bar{T}_2)\ne 0$ entails
$$ \bigl(\mathfrak{J}_{\mu_{T_1},\lambda}+\mathfrak{J}_{\mu_{T_2},\lambda}\bigr)\cap\bigl(\mathfrak{J}_{\mu_{\bar{T}_1},\lambda}+\mathfrak{J}_{\mu_{\bar{T}_2},\lambda}\bigr)\neq\emptyset\,.
$$
Hence $\mu_{T_1},\mu_{T_2},\mu_{\bar{T}_1},\mu_{\bar{T}_2}$ must belong to
\begin{multline*}
	\mathfrak{S}_{R^{-1/2}}:=\Bigl\{ (\mu_1,\mu_2,\bar{\mu}_1,\bar{\mu}_2)\in \varGamma_1\times\varGamma_2\times\varGamma_1\times\varGamma_2;\\
	\mu_1+\mu_2=\bar{\mu}_1+\bar{\mu}_2+\O(R^{-1/2}),\;
	|\mu_1|^2+|\mu_2|^2=|\bar{\mu}_1|^2+|\bar{\mu}_2|^2+\O(R^{-1/2})\Bigr\}
\end{multline*}
	Given  $\mu_1$ and $\bar{\mu}_2$, if we let
$\Pi_{\mu_1,\bar{\mu}_2}=\bigl\{\xi\in\R^n;\,\langle \xi-\bar{\mu}_2,\mu_1-\bar{\mu}_2\rangle=0\bigr\}$
and let
$\Pi^{R^{-1/2}}_{\mu_1,\bar{\mu}_2}$ be the $\O(R^{-1/2})-$neighborhood of $\Pi_{\mu_1,\bar{\mu}_2}$, then $\mu_2\in \Pi^{R^{-1/2}}_{\mu_1,\bar{\mu}_2}$ in order that $(\mu_1,\mu_2,\bar{\mu}_1,\bar{\mu}_2)\in \mathfrak{S}_{R^{-1/2}}$ for some $\bar{\mu}_1\in \varGamma_1$, and this is the observation made in \cite{TaoGFA}. Consequently, for fixed $\mu_1,\mu_2,\bar{\mu}_2$, if we denote $\bar{\mu}_1\prec (\mu_1,\mu_2,\bar{\mu}_2)$ as the $\bar{\mu}_1$'s determined by the $\mathfrak{S}_{R^{-1/2}}$ relation, then
$\mathrm{card}\{\bar{\mu}_1\in\varGamma_1:\,\bar{\mu}_1\prec (\mu_1,\mu_2,\bar{\mu}_2)\}=\O(1)$ for all $(\mu_1,\mu_2,\bar{\mu}_2)$.\smallskip
	
Using 
$$ \frac{m_{T_1}^{G^\lambda,\,\varDelta}m_{T_2}^{F^\lambda,\,\varDelta'}m_{\bar{T}_1}^{G^\lambda,\,\varDelta}m_{\bar{T}_2}^{F^\lambda,\,\varDelta'}}{m_{T_1}^{G^\lambda}\,\;m_{T_2}^{F^\lambda}\,\;m_{\bar{T}_1}^{G^\lambda}\,\;m_{\bar{T}_2}^{F^\lambda}} \le \Biggl( \frac{m_{T_2}^{F^\lambda,\,\varDelta'}}{m_{\bar{T}_2}^{F^\lambda}}\Biggr)^{1/2} \Biggl( \frac{m_{\bar{T}_2}^{F^\lambda,\,\varDelta'}}{m_{T_2}^{F^\lambda}}\Biggr)^{1/2}$$
and Cauchy-Schwarz, we have as in  \cite{Lee21} that
$$
 \max_{\varDelta''\in\mathcal{Q}_{C_0}(Q^*)\setminus\{\varDelta'\}}\sum_{\qq\in\mathsf{K}_{Q^*}(\varDelta'')}
 \Biggl\|\sum_{\substack{T_1\in\mathbf{T}_1,T_2\in\mathbf{T}_2\\ T_1\cap CQ^\lambda_R\neq\emptyset, T_2\cap CQ^\lambda_R\neq\emptyset}}\frac{m_{T_1}^{G^\lambda,\,\varDelta}}{m_{T_1}^{G^\lambda}}F_{T_1}^\lambda\frac{m_{T_2}^{F^\lambda,\,\varDelta'}}{m_{T_2}^{F^\lambda}}G_{T_2}^\lambda\Biggr\|_{L^2(\qq)}^2\\
$$	
is bounded by the product of
\begin{align*}
\Bigg(& \max_{\varDelta''\in\mathcal{Q}_{C_0}(Q^*)\setminus\{\varDelta'\}}
\sum_{\qq\in\mathsf{K}_{Q^*}(\varDelta'')}\quad
\sum_{\substack{T_1,\bar{T}_1\in\mathbf{T}_1,T_2,\bar{T}_2\in\mathbf{T}_2,\; (\mu_{T_1},\;\mu_{T_2},\;\mu_{\bar{T}_1},\;\mu_{\bar{T}_2})\in\mathfrak{S}_{R^{-1/2}}	\\T_1,\bar{T}_1\cap CQ^\lambda_R\neq\emptyset, \;T_2,\bar{T}_2\cap CQ^\lambda_R\neq\emptyset}}\\
&\quad\iint\;\; \Biggl|\chi_\qq^4(\xxx,t)\psi_{T_2}(\xxx_\qq,t_\qq)\psi_{\bar{T}_1}(\xxx,t)\frac{(\psi_{T_1}^{-1}F_{T_1}^\lambda)(\xxx,t)G^\lambda_{\bar{T}_2}(\xxx,t)}{\psi_{\bar{T}_2}(\xxx_\qq,t_\qq)}\Biggr|^2 \frac{m_{T_2}^{F^\lambda,\,\varDelta'}}{m_{\bar{T}_2}^{F^\lambda}}\,d\xxx dt\Bigg)^{1/2}
\end{align*}
and
\begin{align*}
\Bigg(& \max_{\varDelta''\in\mathcal{Q}_{C_0}(Q^*)\setminus\{\varDelta'\}}\sum_{\qq\in\mathsf{K}_{Q^*}(\varDelta'')}\quad\sum_{\substack{T_1,\bar{T}_1\in\mathbf{T}_1,T_2,\bar{T}_2\in\mathbf{T}_2,\;
		(\mu_{T_1},\;\mu_{T_2},\;\mu_{\bar{T}_1},\;\mu_{\bar{T}_2})\in\mathfrak{S}_{R^{-1/2}}	\\T_1,\bar{T}_1\cap CQ^\lambda_R\neq\emptyset, \;T_2,\bar{T}_2\cap CQ^\lambda_R\neq\emptyset}}\\
&\quad\iint\;\; \Biggl|\chi_\qq^4(\xxx,t)\psi_{\bar{T}_2}(\xxx_\qq,t_\qq)\psi_{T_1}(\xxx,t)\frac{(\psi_{\bar{T}_1}^{-1}F_{\bar{T}_1}^\lambda)(\xxx,t)G^\lambda_{T_2}(\xxx,t)}{\psi_{T_2}(\xxx_\qq,t_\qq)}\Biggr|^2 \frac{m_{\bar{T}_2}^{F^\lambda,\,\varDelta'}}{m_{T_2}^{F^\lambda}}\,d\xxx dt\Bigg)^{1/2}
\end{align*}
where $(\xxx_\qq,t_\qq)$ is the center of $\qq$ and we have abused notation replacing $100C$ with $C$ which does not affect the proof.

By symmetry, we only estimate the first one, and the second term is handled by using the same method.\medskip

\noindent\emph{Step 3. End of the proof.}
Squaring  the first term at the end of Step 2 and using $\ell^1\subset\ell^2$, we need to estimate
\begin{multline*} \max_{\varDelta''\in\mathcal{Q}_{C_0}(Q^*)\setminus\{\varDelta'\}}\sum_{\qq\in\mathsf{K}_{Q^*}(\varDelta'')}\quad\sum_{T_1\in\mathbf{T}_1,\bar{T}_2\in\mathbf{T}_2,\;T_1\cap CQ^\lambda_R\neq\emptyset, \,\bar{T}_2\cap CQ^\lambda_R\neq\emptyset}\\
	\frac{1}{m_{\bar{T}_2}^{F^\lambda}}\iint_{\R^{n+2}}\, \chi_\qq^8(\xxx,t)\;\Bigl|\frac{(\psi_{T_1}^{-1}F_{T_1}^\lambda)(\xxx,t)\;G^\lambda_{\bar{T}_2}(\xxx,t)}{\psi_{\bar{T}_2}(\xxx_\qq,t_\qq)}\Bigr|^2\\
	\sum_{\substack{T_2\in \mathbf{T}_2,\; \mu_{T_2}\in\Pi_{\mu_{T_1},\mu_{\bar{T}_2}}^{R^{-1/2}}\\ T_2\cap CQ^\lambda_R\neq\emptyset }}m_{T_2}^{F^\lambda,\varDelta'}\psi_{T_2}^2(\xxx_\qq,t_\qq)\;\;\Bigl|\sum_{\substack{\bar{T}_1\in\mathbf{T}_1,\;\mu_{\bar{T}_1}\prec (\mu_{T_1},\mu_{T_2},\mu_{\bar{T}_2})\\
	\bar{T}_1\cap CQ^\lambda_R\neq \emptyset}}\psi_{\bar{T}_1}(\xxx,t)\Bigr|^2d\xxx dt.
\end{multline*}
By using the uniform estimate $\mathrm{card}\{\bar{\mu}_1:\bar{\mu}_1\prec (\mu_1,\mu_2,\bar{\mu}_2)\}=\O(1)$ for all $(\mu_1,\mu_2,\bar{\mu}_2)$ and the concentration property of bump function $\psi_{\bar{T}_1}$, we have the uniform multiplicity estimate
$$
\max_{T_1,T_2,\bar{T}_2}	\Bigl\|\sum_{\substack{\bar{T}_1\in\mathbf{T}_1,\;\mu_{\bar{T}_1}\prec (\mu_{T_1},\mu_{T_2},\mu_{\bar{T}_2})\\
			\bar{T}_1\cap CQ^\lambda_R\neq \emptyset}}\psi_{\bar{T}_1}\Bigr\|_{L^\infty}\lesssim \varpi^{-O(1)}.
$$
Dropping this term in the integral,
we are  reduced  to estimating the product of
\begin{equation}
\label{eq:max} \max_{\substack{\varDelta''\in\mathcal{Q}_{C_0}(Q^*)\setminus\{\varDelta'\}\\ \qq\in\mathsf{K}_{Q^*}(\varDelta'')}}\max_{T_1\in\mathbf{T}_1,\bar{T}_2\in\mathbf{T}_2,}	\sum_{\substack{T_2\in \mathbf{T}_2,\; \mu_{T_2}\in\Pi_{\mu_{T_1},\mu_{\bar{T}_2}}^{R^{-1/2}}\\ T_2\cap CQ^\lambda_R\neq\emptyset }}m_{T_2}^{F^\lambda,\varDelta'}\psi_{T_2}^2(\xxx_\qq,t_\qq)
\end{equation}
	and
\begin{equation}
\label{eq:trans}	\sum_{\substack{\varDelta''\in\mathcal{Q}_{C_0}(Q^*)\setminus\{\varDelta'\}\\ \qq\in\mathsf{K}_{Q^*}(\varDelta'')}}\sum_{\substack{T_1\in\mathbf{T}_1,\bar{T}_2\in\mathbf{T}_2\\ T_1\cap CQ^\lambda_R\neq\emptyset,\,\bar{T}_2\cap CQ^\lambda_R\neq\emptyset}}
			\frac{1}{m_{\bar{T}_2}^{F^\lambda}}\iint \chi_\qq^8\frac{\Bigl|(\psi_{T_1}^{-1}F_{T_1}^\lambda)G^\lambda_{\bar{T}_2}\Bigr|^2}{\psi_{\bar{T}_2}(\xxx_\qq,t_\qq)^2}d\xxx dt.
\end{equation}

To estimate \eqref{eq:max}, plug
$$ m_{T_2}^{F^\lambda,\,\varDelta'}=\sum_{\qq'\in\mathsf{K}_{Q^*}(\varDelta')}\sum_{T_1'\in\mathbf{T}_1}\bigl\|\chi_{\qq'}\, \psi_{T_2}\,\psi_{T_1'}^{-50}\,F_{T_1'}^\lambda\bigr\|^2_{L^2(\R^{n+2})}
$$
into \eqref{eq:max}. Let $\chi_{\varDelta'}:=\sum_{\qq'\in\mathsf{K}_{Q^*}(\varDelta')} \chi_{\qq'}$
and
$$	\mathcal{W}_{\qq,\varDelta',T_1,\bar{T}_2}^{\lambda,R^{1/2}}(\xxx,t):=\chi_{\varDelta'}(\xxx,t)\sum_{\substack{T_2\in\mathbf{T}_2,\;
	\mu_{T_2}\in\Pi^{R^{-1/2}}_{\mu_{T_1},\mu_{\bar{T}_2}}\\ T_2\cap CQ^\lambda_R\neq\emptyset}}
\bigl(\psi_{T_2}(\xxx_\qq,t_\qq)\psi_{T_2}(\xxx,t)\bigr)^2.
$$
Rearranging the order of summation,  we find that  \eqref{eq:max} can be  bounded with
\begin{multline}
\label{eq:good} \max_{\substack{\varDelta''\in\mathcal{Q}_{C_0}(Q^*)\setminus\{\varDelta'\}\\ \qq\in\mathsf{K}_{Q^*}(\varDelta'')}}\max_{T_1\in\mathbf{T}_1,\bar{T}_2\in\mathbf{T}_2,}
		\sum_{T_1'\in\mathbf{T}_1}\\
\int\chi_{\varDelta'}(\xxx,t) \,\Bigl|\psi_{T_1'}(\xxx,t)^{-50} F_{T'_1}^\lambda(\xxx,t)\Bigr|^2\mathcal{W}_{\qq,\varDelta',T_1,\bar{T}_2}^{\lambda,R^{1/2}}(\xxx,t)\;d\xxx dt\,.
\end{multline}

	To estimate \eqref{eq:good}, we note that for any $\varDelta''\in\mathcal{Q}_{C_0}(Q^*)\setminus\{\varDelta'\}$  , $\varDelta''$ is either separated from $\varDelta'$  along the vertical time-direction by $\O(\varpi 2^{-C_0}\lambda R)$ or  horizontally in the $\xxx-$direction separated away from $\varDelta'$ by a distance  $\varpi 2^{-C_0}R$, where in the second case, the projection of $\varDelta''$ to the temporal component coincides with that of $\varDelta'$.
	If we think of $\psi_{T_2}$ as the indicator function of $T_2$ and likewise for $\chi_{\varDelta'}$, neglecting the Schwartz tails for the moment, then the above separation property implies that $\mathcal{W}_{\qq,\varDelta',T_1,\bar{T}_2}^{\lambda,R^{1/2}}$ is bounded by the characteristic function of the intersection of $\varDelta'$ and
\begin{equation}
\label{eq:UT}
\bigcup_{\substack{T_2\in\mathbf{T}_2,\;
\mu_{T_2}\in\Pi^{R^{-1/2}}_{\mu_{T_1},\mu_{\bar{T}_2}}\\  z_\qq\in T_2\cap CQ^\lambda_R\neq\emptyset}}T_2
	\end{equation}
which is  contained in an $\O(R^{1/2})-$neighbourhood of $\mathbf{ \Lambda}_2^\lambda+z_\qq$ where $z_{\qq}=(\xxx_\qq,t_\qq)$.
This is ensured by the above separateness condition between $\varDelta''$ and $\varDelta'$.
	In fact,  in the first case, since the directions of the $T_2$-tubes are $\O(R^{-1/2}\lambda^{-1})$ separated, any point in $\varDelta'$ belongs to at most $\O(\varpi^{-O(1)})$  many $T_2$ tubes passing through $z_\qq$ (a bush at $z_\qq$). In the second case, since $\qq\subset \varDelta''$, if one starts from $\xxx_\qq$ and travels along the direction $(\frac{\mu_{T_2}}{\lambda},-\frac{|\mu_{T_2}|^2}{2\lambda^2})$,  then one needs spend for a period of time  at least of length $\gtrsim R\lambda$ to arrive at the image of the projection from $\varDelta'$ to $\R_\xxx^{n+1}$. Consequently, these tubes barely meets $\varDelta'$. By the $R^{-1/2}-$separateness
	of the $\mu_{T_2}$'s, we obtain a uniform $\O(\varpi^{-O(1)})$ bound on the multiplicity of overlappings. Thus, by affording  a constant $\varpi^{-O(1)}$, the characteristic function of the set \eqref{eq:UT} is bounded  by that of  $\mathbf{ \Lambda}_2^\lambda(z_\qq, R^{1/2})$. \medskip

	To clarify the idea of the proof based on this observation, let us think of $\psi_{T_1'}^{-50}F^\lambda_{T_1'}$ as the characteristic function of the tube $T_1'$ up to some constants, which is reasonable since $F^\lambda_{ T_1'}$ has sufficent decay to eat the growth of $\psi_{T_1'}^{-50}$ in view of \eqref{eq:decay-wp-0}. Using the decomposition $F^\lambda_{T'_1}=\sum_k F^\lambda_{T'_1,k}$, where we  neglect the $F^{\lambda,b}_{T'_1}$ part by confining $k$ in the summation such that $|k-\nu_{n+1}|\lesssim \varpi^{-2}\varrho$, we may
	partition $T_1'=\cup_{\mathfrak{ V}\subset T_1'}\mathfrak V$ with $\mathfrak{ V}$ being parallel plates and of dimensions $1\times \underbrace{R^{1/2}\times \cdots \times R^{1/2}}_{n\text{ times}}\times R\lambda$ such that they are in direction of $T_1'$. Then the transversality condition implies $\mathfrak{ V}\cap T_2$ is contained in a rectangle of dimension $1\times \underbrace{R^{1/2}\times\cdots\times R^{1/2}}_{n\text{ times}}\times \lambda R^{1/2}$ for  any $T_2$ in \eqref{eq:UT} and any $\mathfrak{ V}\subset T_1'$ in the above sense. Next, noting that by the more demanding
	 transversality condition that $\forall\, T_2$ in \eqref{eq:UT}, $\mu_{T_2}$ must satisfy $\mu_{T_2}\in \Pi_{\mu_{T_1},\mu_{\bar{T}_2}}^{R^{-1/2}}$ for the given $T_1,\bar{T}_2$, we claim that there are at most $\O(1)$ many $T_2$'s in \eqref{eq:UT}, which intersect  in a common sector of length $\lambda R^{1/2}$ with  any fixed $\mathfrak{V}$, taken from the partition of any tube $T_1'$ as above. This is because, $\Pi_{\mu_{T_1},\mu_{\bar{T}_2}}\subset\R^n_\xi$ is an $(n-1)$ dimensinal  hyperplane passing through $\mu_{\bar{T}_2}$ and orthogonal to the vector $(\mu_{T_1}-\mu_{\bar{T}_2})$, from which one is convinced with the claim by noting the nowhere vanishing curvature of paraboloids.
	
	 Indeed, if we let $\ell_T$ be the axis of a given tube $T$, then in view of the $\O(R^{1/2})$ perturbation, namely $T\subset\ell_T^{R^{1/2}}$
 where $\ell^{R^{1/2}}$ is the $\O(R^{1/2})$ neighborhood of $\ell$. For any $T_1'$ and $\mathfrak{ V}\subset T_1'$, partition $\mathfrak{ V}$ into the union of sectors $\{S_j\}_{j=1,\ldots, O(\sqrt{R})}$ of length $\lambda R^{1/2}$, i.e. $\mathfrak{ V}=\cup_j S_j$. For any $T_2,T_2'$ taken  from the bush \eqref{eq:UT}, such that for some fixed $j$, we have $T_2\cap S_j\ne\emptyset$ and $T_2'\cap S_j\ne\emptyset$, then we have 
 $$
 \ell^{CR^{1/2}}_{T_1'}\cap \ell^{CR^{1/2}}_{T_2}\cap \ell^{CR^{1/2}}_{T_2'}\neq\emptyset,\;
 $$
 for some fixed large $C$,
 which entails that 
  $\mu_{T_2},\mu_{T_2'}$ and $\mu_{T_1'}$ must be almost colinear. Let  $\ell_*=\ell(\mu_{T_2},\mu_{T_2'},\mu_{T_1'})\subset\R^n$
be the  line such that $\mu_{T_2},\mu_{T_2'},\mu_{T_1'}\in \ell_*^{CR^{-1/2}}$. Then $\ell_*$ intersects with $\Pi_{\mu_{T_1},\mu_{\bar{T}_2}}$ transversely and the claim follows immediately. We have thus concluded the result, provided that it is legitimate to  neglect the effects from Schwartz tails.

For the general case, one may incorporate the Schwartz tails by the standard dyadic decomposition  exploring the rapid decay of $F^{\lambda,b}_{T_1'}$, $\psi_{T_2}$ and $\chi_{\varDelta'}$ etc. away from the geometric objects where concentration occurs. We refer to \cite{Wolff} and omit the details.

Therefore,  by Lemma \ref{lem:wp-d} and write
$$F^\lambda_{T_1'}(\xxx,t)=\underbrace{\sum_{k:\mathfrak{ V}^\lambda_{\nu,\mu_{T_1'},k}\subset T_1'} c_{\vvv,\mu_{T_1'},k}\,\phi_{\vvv,\mu_{T_1'},k}(\xxx,t)}_{F^{\lambda,g}_{T'_1}}+F^{\lambda,b}_{T'_1}.
$$
Using Cauchy-Schwarz, we bound the contribution of $F^{\lambda,g}_{T_1'}$ to \eqref{eq:good} with
\begin{multline}
\sum_{T_1'}\sum_{k:\;\mathfrak{ V}^\lambda_{\nu,\mu_{T_1'},k}\subset T_1'} c_{\vvv,\mu_{T_1'},k}^2\underbrace{\int \chi_{\varDelta'} \psi_{T_1'}(\xxx,t)^{-100}\phi_{\vvv,\mu_{T_1'},k}(\xxx,t)
	\mathcal{W}_{\qq,\varDelta',T_1,\bar{T}_2}^{\lambda,R^{1/2}}(\xxx,t)d\xxx dt}_{\text{bounded by }O(\lambda R^{\frac{n+1}{2}})}\\
 \lesssim\;\lambda  R^{1/2}\Bigl( R^{\frac{n}{2}}\sum_{\vvv,\mu,k} c_{\vvv,\mu,k}^2\Bigr)
 \lesssim \;\lambda R^{1/2}\varpi^{-O(1)}
	  \label{eq:Good-1}
 \end{multline}
 where we have used \eqref{eq:L-2-sum} in the last estimate. For the $F^{\lambda,b}_{T_1'}$, one use the decay in \eqref{eq:decay-wp-0} and dyadically decompose for each $T_1'$, the summation over $k$ into segments $|k-\nu_{n+1}|\sim 2^\gamma \varpi^{-2}\varrho $ for $\gamma\ge 1$ and apply the same argument as for $F_{T_1'}^{\lambda,g}$ and summing over the dyadic pieces to complete the proof.\medskip

 To estimate \eqref{eq:trans}, we use Wolff's Bernstein type inequality \cite[Lemma 3.2]{Wolff}
 $$
 \bigl\|\chi_\qq^3 \psi_{T_1}^{-1} F_{T_1}^\lambda\bigr\|_{L^\infty}\lesssim R^{-\frac{n+1}{4}}\lambda^{-1/2}\|\psi_{T_1}^{-50}F^\lambda_{T_1}\chi_\qq\|_{L^2}.
 $$
 In fact, the support $\wh{F^\lambda_{T_1}}*\wh{\chi_\qq}$
 is contained in a  $1\times\underbrace{ R^{-1/2}\times\cdots\times R^{-1/2}}_{n\text{  times}}\times \lambda^{-1}R^{-1/2}$ rectangle. Reproducing
 $\chi_\qq F_{T_1}^\lambda=\bigl(\chi_\qq F_{T_1}^\lambda\bigr)*\rho_{R^{1/2}}^\lambda$ for some $L^1-$normalized Schwartz function  $\rho_{R^{1/2}}^\lambda$ essentially supported in a box of volume $\lambda R^{\frac{n+1}{2}}$, it is easy to see the $L^2\to L^\infty$ norm of the operator associated to the kernel
 $$
 \mathscr{H}(\xxx,t;\xxx',t'):=
 \chi_\qq^2(\xxx,t)\psi_{T_1}^{-1}(\xxx,t)\rho_{R^{1/2}}^\lambda(\xxx-\xxx',t-t')\psi_{T_1}^{50}(\xxx',t')
 $$
 is bounded by $\O(\lambda^{-\frac{1}{2}}R^{-\frac{n+1}{4}})$.

 Thus \eqref{eq:trans} is bounded by
 \begin{multline*}
 R^{-\frac{n+1}{2}}\lambda^{-1}	\sum_{\substack{\varDelta''\in\mathcal{Q}_{C_0}(Q^*)\setminus\{\varDelta'\}\\ \qq\in\mathsf{K}_{Q^*}(\varDelta'')}}\sum_{\substack{T_1\in\mathbf{T}_1,\bar{T}_2\in\mathbf{T}_2\\ T_1\cap CQ^\lambda_R\neq\emptyset,\,\bar{T}_2\cap CQ^\lambda_R\neq\emptyset}}\\
 \Bigl(\psi_{\bar{T}_2}(\xxx_\qq,t_\qq)^2\|\psi_{T_1}^{-50}F^\lambda_{T_1}\chi_\qq\|_{L^2}^2/m_{\bar{T}_2}^{F^\lambda}\Bigr)\int \chi_\qq^2\bigl|G^\lambda_{\bar{T}_2}\bigr|^2\psi_{\bar{T}_2}(\xxx_\qq,t_\qq)^{-10}d\xxx dt,
 \end{multline*}
 which is bounded by $\O(R^{-\frac{n}{2}})$, once we evoke the definition of $m_{\bar{T}_2}^{F^\lambda}$ along with the concentration property of $G^\lambda_{\bar{T}_2}$, leading to the estimate
 $$
 \sum_{\bar{T}_2}\sup_{\qq}\| \chi_\qq G^\lambda_{\bar{T}_2}\|_2^2\psi_{\bar{T}_2}(\xxx_\qq,t_\qq)^{-10}\le\varpi^{-O(1)} \lambda R^{1/2}.
 $$
 This estimate is readily deduced by using the wave packet decomposition in Lemma \ref{lem:wp-d} and the same argument in Appendix I of  \cite{TaoMZ}. We only remind the reader that the lifespan of $\qq$ here is multiplied by $\lambda$ which is the sole difference between the proof of the above estimate and (62) of \cite{TaoMZ}.
 \medskip

 Collecting the estimates on \eqref{eq:max} and \eqref{eq:trans}, we obtain \eqref{eq:bil-kakeya} and the proof is complete.

\end{proof}

\section{The Huygens principle: spatial localization}\label{sect:huygens}

In this section, we summarize some properties of the spatial localization
operators introduced  in \cite{TaoMZ} to capture the energy concentration of red and blue waves.

We start with some notations.
A \emph{disk} is a subset $D$ resident in $\R^{n+2}_{\xxx,t}$ of the form $$D=D(\xxx_D,t_D;r_D)=\{(\xxx,t_D):|\xxx-\xxx_D|\le r_D\},$$
for some $(\xxx_D,t_D)\in \R^{n+2}$ and $r_D> 0$.
We call $t_D$ the \emph{time coordinate} of  $D$ and $r_D$ the \emph{radius} of $D$.
The indicator function of $D$ is defined as $$\indic_D(\xxx)=\begin{cases}
1\, ,\;(\xxx,t_D)\in D\,,\\  0\;,\; (\xxx,t_D)\not\in D\,.
\end{cases}$$
Let
$
D_{\pm}:=D\bigl(\xxx_D,t_D;r_D(1\pm r_D^{-\frac{1}{2N}})\bigr).
$
For any $c>0$, let  $c D:=D(\xxx_D,\,t_D; c\, r_D)$.
Define the  \emph{disk exterior} of $D$ as
\begin{equation}
\label{eq:ext}
D^{\mathsf{ext}}=D^{\mathsf{ext}}(\xxx_D,t_D;r_D)=\{(\xxx,t_D):\, |\xxx-\xxx_D|>r_D\}.
\end{equation}
For any $u\in L^\infty_{loc}(\R^{n+2})$ and disk $D$,  we write
\begin{align*}
&\|u\|_{L^2(D)}:=\;\Bigl(\int_{|\xxx-\xxx_D|\le r_D}\bigl|u(\xxx,t_D)\bigr|^2d\xxx\Bigr)^\frac{1}{2},\\\;\;
&\|u\|_{L^2(D^{\mathsf{ext}})}:=\;\Bigl(\int_{|\xxx-\xxx_D|> r_D}\bigl|u(\xxx,t_D)\bigr|^2d\xxx\Bigr)^\frac{1}{2}.
\end{align*}
\subsection{The localization operator $P_D$}
In order to localize the red and blue waves in the physical space, we introduce the localization operator $P_D$ as in \cite{TaoMZ}.
Let $\mathbf{\Upsilon}_0(\xxx)\ge 0$ be the Schwartz function in the proof of Lemma \ref{lem:wp-d}.
For every $r>0$, we set $\mathbf{ \Upsilon}_r(\xxx)=r^{-(n+1)}\mathbf{\Upsilon}_0(r^{-1}\xxx)$.

\begin{definition}
	Let $F^\lambda(t)=S^\lambda(t)f$ with $f\in\mathcal{S}(\R^{n+1})$ such that $ \supp\,\wh{f}\subset\mathcal{B}$.
	For any disk $D=D(\xxx_D,t_D;r_D)$, we define $P_D F^\lambda$ at time $t_D$ as
$$
P_D F^\lambda(t_D)=\Bigl(\indic_D*\mathbf{ \Upsilon}_{r_D^{1-\frac{1}{N}}}\Bigr)\,F^\lambda(t_D),
$$
and  $\forall \;t\in\R$
$$\bigl(P_D F^\lambda\bigr)(t)=S^{\lambda}(t-t_D)\bigl[\,P_D F^\lambda(t_D)\bigr].$$
\end{definition}

The localization operator $P_D$ behaves almost like
a sharp cut-off function, by admitting small mismatching between the margins of $D$ and $D^{\mathsf{ext}}$.
In particular, $P_D$ localizes a wave to $D_+$ and $1-P_D$ localizes to the exterior of $D_-$.
\begin{lemma}
	\label{lem:Tao10.2}
	Let $D$ be a disk with radius $r_D=r\ge C_0$. Then,  $P_D F^\lambda$ satisfies the following local energy estimates
\begin{align}
\label{2.16}\| P_D F^\lambda\|_{L^2(D_+^\mathsf{ext})}\lesssim& \; r^{-N}\EEE(F^\lambda)^{1/2}\\
\label{2.17}	
	\|(1-P_D)F^\lambda\|_{L^2(D_-)}\lesssim & \; r^{-N}\EEE(F^\lambda)^{1/2}\\
\label{2.18}
	\sup_{t}\|P_D F^\lambda(t)\|_{L^2(\R^{n+1})}^2\le&\; \|F^\lambda\|_{L^2(D_+)}^2+\O( r^{-N} \EEE(F^\lambda))\\
\label{2.19}\sup_{t} \|(1-P_D)F^\lambda(t)\|_{L^2(\R^{n+1})}^2\le&\; \|F^\lambda\|_{L^2( D^{\mathsf{ext}}_-)}^2+\mathcal{O}(r^{-N}\EEE(F^\lambda))\\
\label{2.20}\sup_{t} \|(1-P_D )F^\lambda(t)\|_{L^2(\R^d)}\le&\; \EEE(F^\lambda)^{\frac{1}{2}},\; \sup_{t}\|P_D  F^\lambda(t)\|_{L^2(\R^2)}\le \EEE(F^\lambda)^{\frac{1}{2}},
\end{align}
where $D_\pm^\mathsf{ext}$ is the exterior of $D_\pm$ in the sense of \eqref{eq:ext}.
\end{lemma}
\begin{proof}
	The argument is exactly same to \cite{TaoMZ} and we sketch it below.
	To see \eqref{2.16},  for any $x\in D_+^{\mathsf{ext}}$ and any $\xxx'$ such that $(\xxx',t_D)\in D$, one has
	$$
	|\xxx-\xxx'|\ge |\xxx-\xxx_D|-|\xxx'-\xxx_D|\ge r^{1-\frac{1}{2N}}.
	$$
	The rapid decay of $\mathbf{ \Upsilon}_0\in
	\mathcal{ S}(\R^{n+1})$ entails
	$$
	\sup_{\xxx:\;(\xxx,t_D)\in D^{\mathsf{ext}}_+}\Bigl|\,\indic_D*\mathbf{ \Upsilon}_{r^{1-\frac{1}{N}}}(\xxx)\,\Bigr|\lesssim\,r^{-N}.
	$$
	By using the Plancherel theorem, we get \eqref{2.16}.\smallskip
	
	To show \eqref{2.17}, we use $\int \mathbf{ \Upsilon}_0=1$  and the rapid decay of $\mathbf{ \Upsilon}_0 $  to get
	\begin{multline}
	\Bigl|1-\bigl(\indic_D*\mathbf{ \Upsilon}_{r^{1-\frac{1}{N}}}\bigr)(\xxx)\Bigr|=\Bigl|\int\bigl[1-\indic_D(\xxx-\xxx')\bigr]\mathbf{ \Upsilon}_{r^{1-\frac{1}{N}}}(\xxx')d\xxx'\Bigr|\\
	\lesssim_M\; \bigl(1+ r^{-1+\frac{1}{N}}\;\mathsf{dist}\bigl(\xxx,\;D^\mathsf{ext}\bigr)\bigr)^{-M}\,,\label{eq:*-K**}
	\end{multline}
	where for any $\xxx\in D_{-}$,  we have that $\indic_D(\xxx-\xxx')=0$ implies
	\begin{align*}
	|\xxx'|\ge |\xxx-\xxx'-\xxx_D|-|\xxx_D-\xxx|
	\ge r_D-|\xxx-\xxx_D|\ge \mathsf{dist}(\xxx,D^\mathsf{ext}).
	\end{align*}
	Using \eqref{eq:*-K**}, we obtain \eqref{2.17} by taking $M$ sufficiently large.\smallskip
	
	Next, splitting
	$$
	\|P_D F^\lambda(t_D)\|_{L^2}^2=\|P_D F^\lambda\|_{L^2(D_+)}^2+\|P_D F^\lambda\|^2_{L^2(D_+^\mathsf{ext})},
	$$
	where the second term is bounded by $\O(r^{-N}\EEE(F^\lambda))$ due to \eqref{2.16},
	we get \eqref{2.18} by using $0\le \indic_D*\mathbf{ \Upsilon}_{r^{1-\frac{1}{N}}}\le 1$. Similarly, we get \eqref{2.19}.
	Finally,  \eqref{2.20} is obvious. The proof is complete.
	
\end{proof}

\subsection{Concentration of red and blue waves}
In this subsection, we use the localization property of  $P_D$ to characterize the energy concentration of red and blue waves as in \cite{TaoMZ}.

\begin{lemma}
	\label{lem:conctr}
	Let $r\ge C_0$ and $D= D(z_D;\,r)$ with $z_D=(\xxx_D,\,t_D)$. For any red and blue waves $F^\lambda(t)=S^{\lambda}(t)f$ and $G^\lambda(t)=S^{\lambda}(t)g$ where $\wh{f}$ and $\wh{g}$ are supported in $\widetilde{V}_1\times I$ and $\widetilde{V}_2\times I$ respectively, we have
	\begin{align}
	\label{2.21}
	\|P_D F^\lambda\|_{L^\infty\bigl(\R^{n+2}_{\xxx,t}\setminus \mathbf{ \Lambda}^\lambda_1(z_D,\,r(1+r^{-\frac{1}{2N}}))\bigr)}\lesssim_M\,r^{-M}
	\EEE(F^\lambda)^{1/2},\\
	\label{2.21-b}
	\|P_D G^\lambda\|_{L^\infty\bigl(\R^{n+2}_{\xxx,t}\setminus\mathbf{ \Lambda}^\lambda_{2}(z_D,\,r(1+r^{-\frac{1}{2N}}))\bigr)}\lesssim_M\, r^{-M}
	\EEE(G^\lambda)^{1/2},	
	\end{align}
	for all  $ \, M\ge 1$ where  for $j=1,2$, $\mathbf{ \Lambda}^{\lambda}_j(z_D,r(1+r^{-\frac{1}{2N}}))$ is defined to be the conic $\O(r(1+r^{-\frac{1}{2N}}))-$neighbourhood of  $\mathbf{ \Lambda}_j^\lambda+z_D$ as in Lemma \ref{lem:opposite}.
\end{lemma}
\begin{proof}
	By symmetry, we only show the first estimate \eqref{2.21}. The argument is the same with \eqref{2.16} in  Lemma \ref{lem:Tao10.2}. 
	Let $\mathcal{K}_{j}^\lambda(\xxx,t)$ be as in Proposition \ref{pp:KKK} and put 
$$ \mathfrak{A}^D_{1,t}(\xxx,\xxx')=\mathcal{K}^\lambda_{1}(\xxx-\xxx',t-t_D)\,\bigl(\indic_D*\mathbf{ \Upsilon}_{r^{1-\frac{1}{N}}}\bigr)(\xxx')\, .
$$
Then
\begin{equation}
\label{eq:P_DF}
P_D F^\lambda(t,\xxx)=
\int\mathfrak{A}_{1,t}^D(\xxx,\xxx')\,F^\lambda(\xxx',t_D)\,d\xxx' .
\end{equation}
By \eqref{eq:KKK}, we have
\begin{equation}
	\label{heart}
	|\mathcal{K}_{1}^\lambda(\xxx-\xxx',t-t_D)|\lesssim_M \bigl(1+\mathsf{dist}\bigl((\xxx-\xxx',t-t_D)\,,{\mathbf{\Lambda}^\lambda_{ 1}}\,\bigr)\bigr)^{-100MN}.
	\end{equation}
Assume that $(\xxx,t) \notin \mathbf{ \Lambda}_1^\lambda(z_D,r(1+r^{-\frac{1}{2N}}))$. Then, we have
\begin{equation}
\label{eq:ex-est-lower}
\bigl|\xxx-\xxx_D-(t-t_D)\bigl(v,-\frac{|v|^2}{2}\bigr)\bigr|\gtrsim r(1+r^{-\frac{1}{2N}}),\quad\forall\;v\in\,2\,\Xi_1^\lambda.
\end{equation}
Decompose the domain of integration for $\xxx'$ in \eqref{eq:P_DF} into the local part where for some large constant $C$, we have
$$
\mathsf{dist}\bigl((\xxx-\xxx',t-t_D),\mathbf{ \Lambda}^\lambda_{ 1}\bigr)\le C^{-1}r(1+r^{-\frac{1}{2N}})
$$
and the global part 
$$
\mathsf{dist}\bigl((\xxx-\xxx',t-t_D),\mathbf{ \Lambda}^\lambda_{ 1}\bigr)\ge C^{-1}r(1+r^{-\frac{1}{2N}}).
$$

For the global part, \eqref{2.21} follows from using 
\eqref{heart} and the Cauchy-Schwarz inequality. 

For the local part, we have for some $v_1\in\,2\,\Xi_1^\lambda$
$$
\bigl|\xxx'-\xxx+(t-t_D)\bigl(v_1,-\frac{|v_1|^2}{2}\bigr)\bigr|\lesssim C^{-1}r(1+r^{-\frac{1}{2N}})
$$ 
for $C$ sufficiently large. Writing
$$
\indic_D*\mathbf{ \Upsilon}_{r^{1-\frac{1}{N}}}(\xxx')=\int \indic_D(\xxx'')\mathbf{ \Upsilon}_{r^{1-\frac{1}{N}}}(\xxx'-\xxx'')d\xxx'',
$$
and noting that by \eqref{eq:ex-est-lower}
\begin{align*}
	|\xxx'-\xxx''|\ge&\, |\xxx'-\xxx_D|-|\xxx''-\xxx_D|\\
	\gtrsim&\, \Bigl|\xxx-\xxx_D-(t-t_D)\bigl(v_1,-\frac{|v_1|^2}{2}\bigr)\Bigr|-C^{-1}r(1+r^{-\frac{1}{2N}})-r\gtrsim r^{1-\frac{1}{2N}},
\end{align*}
we obtain \eqref{2.21} by using the rapid decay of $\mathbf{ \Upsilon}_0$. The proof is complete.
\end{proof}
\begin{remark}
	The above lemma is analogous to Lemma 10.3 of \cite{TaoMZ}. 
\end{remark}

We write a $\lambda-$stretched  cube $Q^\lambda=Q^\lambda(\xxx_Q,t_Q;r_Q)$  of side-length $r_Q$ centered at $(\xxx_Q,t_Q)\in\R^{n+2}$.  For any $C>0$, we write $C Q^\lambda=Q^\lambda(\xxx_Q,t_Q;C r_Q)$.
\begin{lemma}
	\label{lem:conctr'}
	Let $r\ge C_0$ and $D= D(z_D;\,r)$ with $z_D=(\xxx_D,\,t_D)$. For any red and blue waves $F^\lambda$ and $G^\lambda$, we have for large $C$
	\begin{align}
	\label{2.22'}
	\|(1-P_D)F^\lambda\|_{L^\infty(Q(z_D,\,r/C))}\lesssim&\, r^{-N} \EEE(F^\lambda)^{1/2},\,\\
	\|(1-P_D)G^\lambda\|_{L^\infty(Q(z_D,\,r/C))}\lesssim&\, r^{-N} \EEE(G^\lambda)^{1/2}.
	\end{align}
\end{lemma}
\begin{proof}
We only show the first estimate \eqref{2.22'} by symmetry.
	Let $\mathscr{N}^\lambda_{\;j,t}(\xxx)$ be  given by Proposition \ref{pp:KKK}. Write
	$$
	(1-P_D) F^\lambda(t,\xxx)=\int \mathscr{N}^\lambda_{\;1,t-t_D}(\xxx-\xxx')\bigl[1-(\indic_D*\mathbf{ \Upsilon}_{r^{1-\frac{1}{N}}})(\xxx')\bigr]F^\lambda(t_D,\xxx')\,d\xxx'.
	$$
	For any $(\xxx,t)\in Q(z_D,r/C)$,  we may write
	\begin{align*}
		\xxx'-\xxx+(t-t_D)(v ,-|v|^2/2)=(\xxx'-\xxx_D)
-(\xxx-\xxx_D)+\bigl[(t-t_D)(v ,-|v|^2/2)\bigr],
	\end{align*}
	with $v\in 2\,\Xi_1^\lambda$. 	In view of \eqref{heart}, 	we may assume that $\xxx'\in \frac{1}{2}D$, otherwise, by using $|\xxx-\xxx_D|\le  C^{-1}r, |t-t_D|\le C^{-1}\lambda r$ and the rapid decay of the kernel $\mathscr{N}^\lambda_{\;1,t-t_\D}$ away from $\mathbf{ \Lambda}^{ \lambda}_1$, we obtain an upper bound  $\O (r^{-N}\EEE(F^{\lambda}))$.
	Next, by \eqref{eq:*-K**}
	$$
	\Bigl|1-\bigl(\indic_D*\mathbf{ \Upsilon}_{r^{1-\frac{1}{N}}}\bigr)(\xxx')\Bigr|
	\lesssim_M\; \bigl(1+ r^{-1+\frac{1}{N}}\mathsf{dist}\bigl(\xxx',\;D^\mathsf{ext}\bigr)\bigr)^{-M}\lesssim_M r^{-M/N},
	$$
 for all $\xxx'\in \frac{1}{2}D$.
	The proof is complete by taking $M$ large.
\end{proof}
\subsection{A non-endpoint bilinear estimate}
For $0<r_1<r_2<+\infty$,  we define the cubical annulus as
$$
Q^{\mathsf{ann}}(\xxx_Q,t_Q;r_1,r_2)=Q^\lambda(\xxx_Q,t_Q;r_2)\setminus Q^\lambda(\xxx_Q,t_Q;r_1).
$$
We show a non-endpoint bilinear estimate for localized blue or red waves, on a dyadic annulus. This corresponds to Lemma 11.1 of Tao \cite{TaoMZ}, which will be used to handle the  case when the energy is highly concentrated.
\begin{lemma}
	\label{lem:non-edpt}
	Let $2^{10 C_0}\le R\le \lambda$,
	$C_0\le r\le R^{\frac{1}{2}+\frac{4}{N}}$, and
	$D=D(z_D,\,r)$ with $z_D=(\xxx_D,t_D)$. Then, there exists $b>0$, depending only on $Z$ and $n$, such that for any  red and blue waves $F^\lambda,G^\lambda$ with  $\EEE(F^\lambda)=\EEE(G^\lambda)=1$, we have
	\begin{equation}
	\|(P_D F^\lambda)G^\lambda\|_{Z(Q^\mathsf{ann}(z_D; R,2R))},	\| F^\lambda\, (P_DG^\lambda)\|_{Z(Q^\mathsf{ann}(z_D; R,2R))}\lesssim 2^{O(C_0)}\lambda^{\frac{1}{q}} R^{-b}.
	\end{equation}
\end{lemma}
\begin{proof}
	We only show the estimate for $(P_D F^\lambda)G^\lambda$. By translation and modulation,  we may take $(\xxx_D,t_D)=(0,0)$ without loss of generality.
	
	By interpolation and taking $N$ large enough, it suffices to show (see Section \ref{sec:pre+top})
\begin{equation}
\label{eq:oooppp}
	\|(P_D F^\lambda)G^\lambda\|_{L^1(Q^\mathsf{ann}(0,0;R,2R))}\lesssim  \lambda R^{\frac{3}{4}+\frac{2}{N}},
	\end{equation}
	\begin{equation}
	\label{eq:ooopppqqq}
	\|(P_D F^\lambda)G^\lambda\|_{L^2(Q^\mathsf{ann}(0,0;R,2R))}\lesssim 2^{O(C_0)} \lambda^{1/2} R^{-\frac{n-1}{4}} R^{C/N}.
\end{equation}
We may focus on
$\Omega_{r,R}^\lambda:=
	Q^\mathsf{ann}(0,0;R,2R)\cap \mathbf{ \Lambda}^\lambda_1(0,0; Cr+R^{\frac{1}{N}})
	$ with $C\gg 1$ fixed,
	since by  Lemma \ref{lem:conctr},  $(P_D F^\lambda)$ is bounded by  $\O(R^{-N})$ outside $ \mathbf{ \Lambda}^\lambda_1(0,0; Cr+R^{\frac{1}{N}})$.\smallskip
	
	To get the $L^1(\Omega_{r,R}^\lambda)$ estimate,
	we have by using Lemma \ref{lem:opposite} and energy estimate
	$$
	\|(P_D F^\lambda) G^\lambda\|_{L^1(\Omega_{r,R}^\lambda)}
	\lesssim \lambda ((r+R^{\frac{1}{N}})R)^{\frac{1}{2}}
	\lesssim \lambda R^{\frac{3}{4}+\frac{2}{N}}.
	$$
	
	To get the $L^2-$estimate \eqref{eq:ooopppqqq} on $\Omega_{r,R}^\lambda$, we use the wave packet decomposition for $P_D F^\lambda$ Lemma \ref{lem:wp-d}  in Section \ref{sec:pre+top} with $\varrho=R^{1/2}$ and $\varpi=2^{-O(C_0)}$
	$$
	P_D F^\lambda=\sum_{T_1\in\mathbf{T}_1} \bigl(P_D F^\lambda\bigr)_{T_1}.
	$$
	Using the rapid decay property of the wave packets away from plates $\mathfrak{ V}_{T_1 }^\lambda\subset T_1$ (see \eqref{eq:decay-wp-0}), we have
$$\|(P_D F^\lambda)G^\lambda\|_{L^2(\Omega_{r,R}^\lambda)}
	\lesssim\biggl\|\sum_{\substack{T_1\in\mathbf{T}_1
			\\
			\mathsf{dist}(T_1,(0,0))\lesssim R^{\frac{1}{2}+\frac{50}{N}}}}(P_D F^\lambda)_{T_1}G^\lambda\,\biggr\|_{L^2(Q^\mathsf{ann}(0,0;R,2R))}+R^{-N}.
	$$
	Note that the directions of the tubes $T_1$ are $\O(\lambda^{-1}R^{-1/2})$-separated. By crude estimates,  the multiplicities of the tubes, that are $\O(R^{1/2+50/N})-$close to the origin, over $Q^\mathsf{ann}(0,0;R,2R)$ is bounded by $\O(R^{C/N})$ with $C$ depending only on $n$.
	By Cauchy-Schwarz, we have
	$$
	\biggl\|\sum_{\substack{T_1\in\mathbf{T}_1
			\\
			\mathsf{dist}(T_1,(0,0))\lesssim R^{\frac{1}{2}+\frac{1}{N}}}}(P_D F^\lambda)_{T_1}G^\lambda\,\biggr\|_{L^2(Q^\mathsf{ann}(0,0;R,2R))}\lesssim R^{\frac{C}{N}}
\Bigl(	\sum_{T_1\in\mathbf{T}_1}
\bigl\|(P_D F^\lambda)_{T_1}G^\lambda\bigr\|_{L^2}^2\Bigr)^{\frac{1}{2}}.
	$$
	By the same partition of tubes into plates using  \eqref{eq:wpd} as we did in Section \ref{sec:pre+top},  where $S^\lambda(t)f$ is replaced by $P_D F^\lambda$, we have
	$$(P_D F^\lambda)_{T_1}=\sum_{k:\,\mathfrak{ V}^\lambda_{\vvv,\mu_{T_1},k}\subset T_1}\tilde{c}_{\vvv,\mu_{T_1},k}\,\tilde{\phi}_{\vvv,\,\mu_{T_1},\,k}+\sum_{k:\,\mathfrak{ V}^\lambda_{\vvv,\mu_{T_1},k}\not\subset T_1}\tilde{c}_{\vvv,\mu_{T_1},k}\,\tilde{\phi}_{\vvv,\,\mu_{T_1},\,k}$$
	and by Cauchy-Schwarz
	\begin{align*}
	\bigl\|(P_D F^\lambda)_{T_1}G^\lambda\bigr\|_{L^2}^2\lesssim
	\sum_{k}
	\tilde{c}_{\vvv,\mu_{T_1},k}^2	\bigl\|\tilde{\phi}_{\vvv,\mu_{T_1},k}G^\lambda\bigr\|_{L^2}^2\,.
	\end{align*}
	Here, we use $\tilde{c}$ and $\tilde{\phi}$ to stress that they are the coefficients and wave packets for $P_DF^\lambda$.
	Noting that $\tilde{\phi}_{\vvv,\mu_{T_1}}$ is essentially  concentrated in  an $\O(R^{1/2})-$neighbourhood of $\mathbf{ \Lambda}_1^\lambda$ satisfying the same formula \eqref{eq:decay-wp-0}, we have
	$$
	\bigl\|\tilde{\phi}_{\vvv,\mu_{T_1},k}G^\lambda\bigr\|_{L^2}^2\lesssim \varpi^{-O(1)} \lambda R^{1/2}
	$$
	 by using Lemma \ref{lem:opposite} and standard dyadic decomposition to incorporate Schwartz tails.
	 The coefficients $(\tilde{c}_{\vvv,\mu,k})_{\vvv,\mu,k}$ also obeying the $\ell^2-$summation formula \eqref{eq:L-2-sum},
	summing over $\vvv,\mu,k$, we obtain
	 \eqref{eq:ooopppqqq}.
	The proof is complete.
\end{proof}

\begin{remark}
	It is in this lemma that we need take $N$ depending on $Z$, whereas in \cite{TaoMZ}, there is no need to do so.
\end{remark}

\section{Explore the energy concentration}\label{sec:wt}

We shall use the method of induction on scales  in \cite{TaoMZ}. To this end, we  consider a subclass of the red and blue waves by imposing a margin condition.

Let $$\varSigma^\lambda=\Bigl\{(\xi,s,\tau): \tau=-\frac{|\xi|^2}{2(\lambda+s)}\Bigr\}$$ and for $j=1,2$
$$
\varSigma_j^\lambda=\Bigl\{(\xi,s,\tau): \tau=-\frac{|\xi|^2}{2(\lambda+s)},\,(\xi,s)\in \widetilde{V}_j\times I\Bigr\}.
$$
We shall say $\varSigma_j^\lambda$ is the \emph{lift} of $\widetilde{V}_j\times I$ to $\varSigma^\lambda$.

For any $R\ge 2^{C_0}$, we say $F^\lambda(t,\xxx)$ is a $\mathfrak{R}^{\lambda}_R$-\emph{wave} if it is a red wave with the spacetime Fourier transform $\wh{F}(\xi,s,\tau)$  being an $L^2$ measure on $\varSigma_1^\lambda$ and satisfying the margin condition
$$\mathsf{marg}(F^\lambda):=\mathsf{dist}\bigl(\supp(\wh{F^\lambda}),\varSigma^\lambda\setminus \varSigma_1^\lambda\bigr)\ge (100n)^{-1}-R^{-\frac{1}{N}}.$$
Similarly, we define  $\mathfrak{B}^\lambda_R$ to be  the subset of  blue waves of  functions  $G^\lambda$ such that $\supp\; \wh{G^\lambda}\subset \varSigma_2^\lambda$ and satisfying the margin condition
$$\mathsf{marg}(G^\lambda):=\mathsf{dist}\bigl(\supp(\wh{G^\lambda}),\varSigma^\lambda\setminus \varSigma_2^\lambda\bigr)\ge (100n)^{-1}-R^{-\frac{1}{N}}.$$

\begin{definition}
	\label{def-A}
	Fix $\lambda\ge 2^{10C_0}$.
	For any $R\in [2^{10 C_0}, \lambda]$, fix $Q_R^\lambda\subset\R^{n+2}$ a $\lambda-$stretched  spacetime cube of size $R$. Let $A^\lambda(R)$ be the optimal constant $C$ such that \begin{equation}
	\label{eq:bilinar-wolff}
	\| F^\lambda G^\lambda\|_{Z(Q^\lambda_R)}\le C	\lambda^{1/q}\,\EEE(F^\lambda)^{1/2}\EEE(G^\lambda)^{1/2},
	\end{equation}
	 holds for all  $F^\lambda\in\mathfrak{R}^\lambda_R, G^\lambda\in \mathfrak{B}^\lambda_R$ and all $Q^\lambda_R$\,.
\end{definition}
 By translation   in the physical spacetime and  modulating the frequency variables resp., $A^\lambda(R)$ is independent of the  particular choice of $Q_R^\lambda$. \smallskip

To show Theorem \ref{thm:main}, we shall show that there is a fixed constant $C_*$ depending only on $n,\eps$, $\widetilde{V}_1, \widetilde{V}_2$ and the $(q,s)$ exponent in $Z-$norm  taken sufficiently close to the critical index $(q_c,r_c)$, such that $A^\lambda(R)\le C_*$ holds for all $R\le \lambda$ and all $\lambda\ge 2^{10C_0}$. We may assume $A^\lambda(R)\ge 1$.
 We set for any $2^{C_0}\le R\le\lambda$ $$\displaystyle\overline{A}^{\lambda}(R)=\sup_{2^{C_0}\le\lambda'\le\lambda}\;\sup_{2^{C_0}\le R'\le \min(R,\lambda')}A^{\lambda'}(R')$$  for a technical issue.\smallskip

We
 need to introduce an auxiliary quantity, a crucial innovation made in \cite{TaoMZ}.
 We call $\mathsf{LS}(Q^\lambda):=[t_{Q^\lambda}-\lambda \frac{r_{Q^\lambda}}{2},\,t_{Q^\lambda}+\lambda \frac{r_{Q^\lambda}}{2}]$ the \emph{lifespan} of $Q^\lambda=Q^\lambda(\xxx_{Q^\lambda},t_{Q^\lambda};r_{Q^\lambda})$. In particular, the length of $\mathsf{LS}(Q^\lambda_R)$ equals to $\lambda R$.
\begin{definition}
	\label{def:AAA}
	For any $R\ge 2^{10C_0}$ and  $r,r'>0$, we define
	$\mathscr{A}^{\lambda}(R,r,r')$ to be the optimal constant $C$ such that
	$$
	\bigl\|F^\lambda G^\lambda\bigr\|_{Z(Q^\lambda_R\cap\, \mathcal{C}^\lambda(z_0,r'))}
	\le C\lambda^{1/q} \bigl(\EEE(F^\lambda)\EEE(G^\lambda)\bigr)^{\frac{1}{2q}}\; \EEE_{r,C_0Q^\lambda_R}(F^\lambda,G^\lambda)^{\frac{1}{q'}},
	$$
	holds for all  $F^\lambda\in\mathfrak{R}^\lambda_R,G^\lambda\in\mathfrak{B}^\lambda_R$  and all cubes $Q_R^\lambda$ being as in Definition \ref{def-A} and all $z_0=(\xxx_0,t_0)\in\R^{n+2}$.
	Recall that $\mathcal{ C}^\lambda(z_0,r')=\mathbf{ \Lambda}^\lambda_1(z_0,\,r')\cup \mathbf{ \Lambda}^\lambda_2(z_0,\,r')$, and $q$ is given by the $Z-$norm, $q'=\frac{q}{q-1}$. Here,  $\EEE_{r,Q^\lambda_R}$ is the \emph{energy concentration} defined in the same way as  \cite{TaoMZ} by letting
	$$
	\EEE_{r,Q^\lambda_R}(F^\lambda,G^\lambda)=\max\Bigl\{\frac{1}{2}\EEE(F^\lambda)^{1/2}\EEE(G^\lambda)^{1/2},\; \sup_D\|F^\lambda\|_{L^2(D)}\|G^\lambda\|_{L^2(D)}\Bigr\},
	$$
	where $D$ ranges over all disks of radius $r$ with the time coordinate $t_D\in\mathsf{LS}(Q^\lambda_R)$.
\end{definition}
We remark here that it is $C_0Q^\lambda_R$ rather than $Q^\lambda_R$ in the definition of $\mathscr{A}^\lambda$ in order to cover the structural constants  by taking $C_0$ large.

\subsection{Persistence of the non-concentration of energy}
The following result will be used to handle the energy-concentrated case.

	\begin{proposition}
	\label{pp:persist}
	Let $ 2^{10C_0}\le R\le \lambda$  and $Q=Q^\lambda_R$ be a $\lambda-$stretched spacetime cube of size $R$. For each $r\ge R^{1/2+1/N}$, we define $r^\#:=r(1-C_0r^{-\frac{1}{3N}})$.
	There exists a constant $C>0$, such that  if $F^\lambda\in\mathfrak{R}^\lambda_R, G^\lambda\in\mathfrak{B}^\lambda_R$
	with $\EEE(F^\lambda)=\EEE(G^\lambda)=1$
	and $\mathcal{F}^\lambda$, $\mathcal{G}^\lambda$
	are $(\lambda,\varpi,R^{1/2})$-wave tables for $F^\lambda, G^\lambda$ over $Q$, then
	\begin{align*}
	\sup_{\varDelta\in\mathcal{Q}_{C_0}(Q)}
	\EEE_{r^\#,5Q}\bigl(\mathcal{F}^{\lambda,\varDelta},\mathcal{G}^{\lambda,\varDelta}\bigr)
	\le (1+C\varpi)\,\EEE_{r,5Q}(F^\lambda,G^\lambda)+\mathcal{O}\bigl(\varpi^{-O(1)}R^{-N/2}\bigr)
	\end{align*}
	holds for all $\varpi \in (0,2^{-C_0})$ and all $Q$.
\end{proposition}
\begin{proof}
	The argument is the same to \cite{TaoMZ} and we only sketch it.
	By \eqref{eq:bessel},
	it suffices to show there is a universal constant $C$ such that
	\begin{multline*}
	\sup_{\triangle\in\mathcal{Q}_{C_0}(Q)}
\|\mathcal{F}^{\lambda,\varDelta}\|_{L^2(D(z,r^\#))}\|\mathcal{G}^{\lambda,\varDelta}\|_{L^2(D(z,r^\#))}\\
	\le (1+C\varpi)\,\|F^\lambda\|_{L^2(D(z,r))}\|G^\lambda\|_{L^2(D(z,r))}+\mathcal{O}(r^{-100N})
	\end{multline*}
	holds for all $\varpi \in (0,2^{-C_0})$ and all  $z,Q$.

	Following the proof of (56) in \cite{TaoMZ}, we let
	\begin{align*}
	D=D(z,r),\;
	D^\#=&D(z,{r}^\#),\;D^\flat=D(z,{r}^\flat),\;\;r^\flat=r\Bigl(1-\frac{C_0}{2}{r}^{-\frac{1}{3N}}\Bigr).
	\end{align*}
	Then, we have
	$D^\#\,\subsetneqq\,D^\flat_{-}\,\subsetneqq D^\flat\,\subsetneqq\,D_{+}^\flat\,\subsetneqq \,D$ by taking $C_0$ sufficiently large.\smallskip

	We only deal with $\mathcal{F}^\lambda$ and the same arguments works for $\mathcal{G}^\lambda$.
	Write
	\begin{align*}
	\mathcal{F}_{\mathsf{int}}^{\lambda,\,\varDelta}(t)=&\,S^\lambda(t-t_{D^\#})\sum_{T_1\in\mathbf{T}_1}\frac{m_{T_1}^{\varDelta}}{m_{T_1}}\,\bigl(P_{D^\flat}\,F^\lambda\bigr)_{T_{1}}(t_{D^\#}),\\
		\mathcal{F}_{\mathsf{ext}}^{\lambda,\,\varDelta}(t)=&\,S^\lambda(t-t_{D^\#})\sum_{T_1\in\mathbf{T}_1}\frac{m_{T_1}^{\varDelta}}{m_{T_1}}\,\bigl((1-P_{D^\flat})\,F^\lambda\bigr)_{T_{1}}(t_{D^\#}).
	\end{align*}
	Then
	\begin{align*}
	\mathcal{F}^{\lambda,\varDelta}(t)=\;\;	\mathcal{F}_{\mathsf{int}}^{\lambda,\,\varDelta}(t)+\	\mathcal{F}_{\mathsf{ext}}^{\lambda,\,\varDelta}(t).
	\end{align*}
	By linearity of $P_{D^\flat}$,
	applying Lemma \ref{lem:bessel-w-t}  and then using \eqref{2.18}, we get
	\begin{multline*}
	\Bigl\|\mathcal{F}_{\mathsf{int}}^{\lambda,\varDelta}(t_{D^\#})\Bigr\|_{L^2(\R^{n+1})}^2
	\le\; (1+C\varpi) \|P_{D^\flat}F^\lambda(t_{D^\#})\|_{L^2(\R^{n+1})}^2\\
	\le\;(1+C\varpi)\|F^\lambda\|_{L^2(D^\flat_{+})}^2+\O(r^{-N})\le (1+C\varpi) \|F^\lambda\|_{L^2(D)}^2+\O(r^{-N})
	\end{multline*}
	where we used $D_{+}^\flat\subset D$ in the last estimate.\smallskip

	On the other hand, using the fast decay of wavepackets away from $CQ$, we have
	\begin{multline*}
	\;\Bigl\|\,\sum_{T_1}\frac{m_{T_1}^\varDelta}{m_{T_1}} \bigl((1-P_{D^\flat})F^\lambda\bigr)_{T_1}(t_{D^\#})\Bigr\|_{L^2(D^\#)}^2\\
	\le\;\Bigl\|\,\sum_{T_1 :\,T_1\cap CQ\neq\emptyset}\frac{m_{T_1}^{\varDelta}}{m_{T_1}} \bigl((1-P_{D^\flat})F^\lambda\bigr)_{T_{1}}(t_{D^\#})\Bigr\|_{L^2(D^\#)}^2+\O( \,r^{-N}).\quad\quad\quad\quad\quad
	\end{multline*}
	 for some large fixed constant $C$. By  Minkowski's inequality, we are reduced to
	\begin{equation}
	\label{eq:ext-red}
	\max_{T_1\in\mathbf{T}_1}
	\Bigl\|\bigl((1-P_{D^\flat}) F^\lambda\bigr)_{T_1}(t_{D^\#})\Bigr\|_{L^2(D^\#)}\lesssim r^{-100N}.
	\end{equation}
	To this end, we condider the two cases
	\begin{itemize}
		\item Case A. $\mathsf{dist}(T_1,D^\#)\ge R^{\frac{1}{2}+\frac{1}{100 N}}$,
		\item Case B. $\mathsf{dist}(T_1,D^\#)\le R^{\frac{1}{2}+\frac{1}{100 N}}$,
	\end{itemize}
	For Case A, \eqref{eq:ext-red} follows from the concentration property \eqref{eq:decay-wp-0} and the summability estimate \eqref{eq:L-2-sum}. To handle Case B, we recall from the wave-packet decomposition that if we let $F_\mathsf{ext}^\lambda=(1-P_{D^\flat})F^\lambda$, then 
	$$
	\EEE\bigl(\bigl(F^\lambda_{\mathsf{ext}}\bigr)_{T_1}\bigr)\lesssim \varpi^{-O(1)}\|\psi_{T_1}(t_{D^{\#}}) F^{\lambda}_{\mathsf{ext}}(t_{D^\#})\|_{L^2}^2 + r^{-100N}.
	$$
	Using the assumption $r\ge R^{\frac{1}{2}+\frac{1}{N}}$ and that in Case B, we have $$\mathsf{dist}\Bigl(\underbrace{\xxx_{T_1}+t_{D^\#}\bigl(\frac{\mu_{T_1}}{\lambda},-\frac{|\mu_{T_1}|^2}{\lambda^2},1\bigr)}_{:=\widetilde{X}},\;\xxx_{D^\#}\Bigr)\lesssim r^\# + R^{\frac{1}{2}+\frac{1}{100 N}},$$
	where $z=(\xxx_{D^\#},t_{D^\#})$ and $T_1$ is parametrized by $(\xxx_{T_1},\mu_{T_1})$. Simple calculation yields that the disk centered at $\widetilde{X}$ of radius $R^{\frac{1}{2}+\frac{1}{100 N}}$ is contained in $D^\flat_-$. Using the rapid decay of $\psi_{T_1}$,	
	\eqref{2.17} and $D^\#\subset D_{-}^\flat$, we obtain \eqref{eq:ext-red}.

	By the same argument, we  have similar estimates for $\mathcal{G}^{\lambda,\varDelta}$.
	Collecting all of these estimates, we obtain the desired result by suitably adjusting the constant $C$.
\end{proof}

\subsection{Control of $\mathscr{A}^{\lambda}$ by $A^\lambda$}
\begin{proposition}
	\label{pp_2}
	There is a constant $C>0$ depending only on $n$ and $Z$, but not explicitly on $C_0$, such that we have for all $R\in[ 2^{10C_0},\lambda]$ and all $r\ge R^{\frac{1}{2}+\frac4N}$
	\begin{equation}
	\label{eq:pp-2}
	\mathscr{A}^{\lambda}({R},r,C_0(r+1))\le (1+C2^{-C_0})\overline{A}^\lambda(R)+2^{CC_0}.
	\end{equation}
\end{proposition}
We divide the proof into three steps.

\subsubsection{Step 1. The non-concentrated case $r \ge C_0R$}

Recall a technical lemma:
\begin{lemma}
	\label{lem:Te}
	Let $F_1,F_2,\ldots,F_k$ be a finite number of functions on $\R^{n+2}$ such that  $\{F_j\}_j\subset Z(\R^{n+2})$ and that the supports of these functions are mutually disjoint. Then, we have
	$$
	\Bigl\|\sum_{j=1}^k F_j\Bigr\|_{Z}^q\le \sum_{j=1}^k\|F_j\|_Z^q\;,
	$$
	for all $(q,s)\in\mathbf{\Gamma}$ close to the critical index $(q_c,r_c)$ .
\end{lemma}
Please see Lemma 5.3 of \cite{Temur} for the proof.

\begin{proposition}
	\label{pp:the noncon}
	There is a constant $C>0$ such that for any $R\in[2^{10C_0},\lambda]$, we have for all $r \ge C_0R$ and $r'>0$
	$$\mathscr{A}^{\lambda}(R,r, r')\le (1+C \varpi)
	\overline{A}^{\lambda}(R)+\varpi^{-C},
	$$
	for all $0<\varpi\le 2^{-C_0}$.
\end{proposition}
\begin{proof}
	Let $F^\lambda\in\mathfrak{R}^\lambda_R,G^\lambda\in\mathfrak{B}^\lambda_R$ be red and blue waves  with normalized energy.
	For any $Q=Q^\lambda_R$, let $(\xxx_Q,t_Q)$ be the center of $Q$. Let $D=D(z_D,r/2)$ with $z_D=(\xxx_Q,t_Q)$ and write
	$$
	F^\lambda=P_D F^\lambda+(1-P_D)F^\lambda,\;G^\lambda=P_D G^\lambda+(1-P_D)G^\lambda.$$
	Using Lemma \ref{lem:conctr'}, we have
	$$
	\|\bigl((1-P_D)F^\lambda\bigr)\; G^\lambda\|_{Z(Q_R^{\lambda})},\;\|(P_D F^\lambda)(1-P_D)G^\lambda\|_{Z(Q_R^\lambda)}\le \lambda^{1/q}\varpi^{-O(1)}.
	$$
	We are reduced to
	\begin{equation}
	\label{UHU}
\lambda^{-1/q}	\|P_D F^\lambda P_D G^\lambda\|_{Z(Q)}\le (1+C\varpi)\overline{A}^{\lambda}(R)\, \EEE_{r,C_0Q}(F^\lambda,G^\lambda)^{1/q'}+\varpi^{-O(1)}.
	\end{equation}
	To see this is the case, let $\mathcal{F}^\lambda_D$ and $\mathcal{G}^\lambda_D$ be the wave tables  for the red and blue waves $P_DF^\lambda$ and $P_D G^\lambda$ on an enlarged cube $Q^*$ containing $Q$ and  apply Proposition \ref{pp:C-0 quilt} so that we have
	$$
	\|P_D F^\lambda P_D G^\lambda\|_{Z(Q_R^\lambda)}\le (1+C\varpi)
	\bigl\|[\mathcal{F}_D^\lambda]_{C_0}[\mathcal{G}_D^\lambda]_{C_0}\bigr\|_{Z(\mathfrak{I}^{\varpi,C_0}(Q^*))}+\lambda^{1/q}\varpi^{-O(1)}.
	$$
	Applying Lemma \ref{lem:Te} and the definition of $A^\lambda(R)$, we get
	\begin{multline}
	\lambda^{-1/q}\bigl\|[\mathcal{F}_D^\lambda]_{C_0}[\mathcal{G}_D^\lambda]_{C_0}\bigr\|_{Z(\mathfrak{I}^{\varpi,C_0}(Q^*))}\le\Bigl(\sum_{\varDelta\in \mathcal{Q}_{C_0}(Q^*)}
	\lambda^{-1}\bigl\|\mathcal{F}_{D}^{\lambda,\varDelta}\,\mathcal{G}_{D}^{\lambda,\varDelta}\bigr\|^q_{Z(\varDelta)}\Bigr)^{1/q}\\
	\le A^\lambda(2^{-C_0}R)\Bigl(\sum_{\triangle\in\mathcal{Q}_{C_0}(Q^*)}
	\EEE(\mathcal{F}_{D}^{\lambda,\varDelta})^{q/2} \EEE(\mathcal{G}_{D}^{\lambda,\varDelta})^{q/2}\Bigr)^{1/q},
	\end{multline}
	where we have used $\mathcal{F}^{\lambda,\varDelta}_D\in\mathfrak{R}^\lambda_{2^{-C_0}R},\, \mathcal{G}^{\lambda,\varDelta}_D\in\mathfrak{B}^\lambda_{2^{-C_0}R}.$
	Using Cauchy-Schwarz,
	$\EEE(\mathcal{F}_{D}^\lambda), \EEE(\mathcal{G}_D^\lambda)\le 1+C\varpi$ by  Lemma \ref{lem:bessel-w-t}, and \eqref{2.18} of Lemma \ref{lem:Tao10.2},
we get \eqref{UHU}.
\end{proof}
\subsubsection{Step 2.  The concentrated case $R^{1/2+3/N}\le r\le C_0 R$}
\begin{proposition}
	\label{pp:the-con}
	There is $C>0$ and $\theta>0$ such that for any $R\in[ 2^{10C_0},\lambda]$, we have for all $r,r'>0$ with $R^{1/2+3/N}\le r\le C_0 R$.
	$$
	\mathscr{A}^{\lambda}(R,r,r')\le (1+C\varpi) \mathscr{A}^{\lambda}(R/C_0,r^\#,r')+\varpi^{-C}\Bigl(1+\frac{R}{r'}\Bigr)^{-\theta}
	$$
	with $r^\#=r(1-C_0r^{-\frac{1}{3N}})$
	holds	for all $0<\varpi\le 2^{-C_0}$.
\end{proposition}
\begin{proof}
	Let $F^\lambda\in\mathfrak{R}^\lambda_R,G^\lambda\in \mathfrak{B}^\lambda_R$ be red and blue waves  with $\EEE(F^\lambda)=\EEE(G^\lambda)=1$. For any $Q=Q^\lambda_R$, by using Proposition \ref{pp:C-0 quilt}, we have for all $z$
	\begin{multline*}
	\bigl\|F^\lambda G^\lambda\bigr\|_{Z(Q\cap \mathcal{C}^\lambda(z,r'))}\\
	\le (1+C\varpi)\Bigl\|[\mathcal{F}^{\lambda}]_{C_0}[\mathcal{G}^{\lambda}]_{C_0}\Bigr\|_{Z(\mathfrak{I}^{\varpi,C_0}(Q^*)\cap
		\mathcal{C}^\lambda(z,r'))}
	+\lambda^{1/q}\varpi^{-C}\Bigl(1+\frac{R}{r'}\Bigr)^{-\kappa}.
	\end{multline*}
	We are reduced to showing
	\begin{multline}
	\label{OIU}
	\lambda^{-1/q}	\Bigl\|[\mathcal{F}^{\lambda}]_{C_0}[\mathcal{G}^{\lambda}]_{C_0}\Bigr\|_{Z(\mathfrak{I}^{\varpi,C_0}(Q^*)\cap
		\mathcal{C}(z,r'))}\\ \le (1+C\varpi) \mathscr{A}^{\lambda}(R/C_0,r^\#,r')\;
	\EEE_{r,C_0Q}(F^\lambda,G^\lambda)^{1/q'}+\varpi^{-C} R^{-N/2}.
	\end{multline}
	Using the definition of $\mathscr{A}^{\lambda}(R,r,r')$, we have for all $\varDelta$
	\begin{multline*}
	\lambda^{-1/q}\bigl\|\mathcal{F}^{\lambda,\varDelta}\mathcal{G}^{\lambda,\varDelta}\bigr\|_{Z(\varDelta\cap
		\mathcal{C}(z,r'))}\\
	\le \mathscr{A}^{\lambda}(R/C_0, r^\#,r')\,\EEE_{r^\#,C_0\varDelta}(\mathcal{F}^{\lambda,\varDelta},\mathcal{G}^{\lambda,\varDelta})^{1/q'}(\EEE(\mathcal{F}^{\lambda,\varDelta})\EEE(\mathcal{G}^{\lambda,\varDelta}))^{1/(2q)}
	\end{multline*}
	By using Lemma \ref{lem:Te}, Proposition \ref{pp:persist} with $2^{-C_0}CR\ll C_0^{-1}R$ so that $C_0\varDelta\subset 2Q$, we obtain by  Cauchy-Schwarz
	\begin{multline}
	\label{OIKU}
		\lambda^{-1}	\Bigl\|[\mathcal{F}^{\lambda}]_{C_0}[\mathcal{G}^{\lambda}]_{C_0}\Bigr\|^q_{Z(\mathfrak{I}^{\varpi,C_0}(Q^*)\cap
		\mathcal{C}(z,r'))}\\
	\le
\lambda^{-1}	\sum_{\varDelta\in\mathcal{Q}_{C_0}(Q^*)}
	\bigl\|\mathcal{F}^{\lambda,\varDelta}\mathcal{G}^{\lambda,\varDelta}\bigr\|^q_{Z(\mathfrak{I}^{\varpi,C_0}(Q^*)\cap
		\mathcal{C}^\lambda(z,r'))}\\ \le (1+C\varpi) \mathscr{A}^{\lambda}(R/C_0,r^\#,r')^q\,\,
	\EEE_{r,C_0Q}(F^\lambda,G^\lambda)^{q/q'}+\varpi^{-C} R^{-qN/2},
	\end{multline}
	and \eqref{OIU} follows by adjusting the constant $C$.
\end{proof}
\subsubsection{Step 3. Proof of Proposition \ref{pp_2}}
With Proposition \ref{pp:the noncon} and \ref{pp:the-con}, we may complete the proof of  Proposition \ref{pp_2}. For the nonconcentrated case $r\ge C_0\,{R}$,  letting $r'=C_0(r+1)$ and $\varpi=2^{-C_0}$  in Proposition \ref{pp:the noncon}, we are done.
In the high concentrated case, we follow \cite{TaoMZ} by letting $J$ be the smallest integer such that $r\ge 2^{-J}C_0 R$ and define
$
r:=r_0>r_1>\cdots>r_J
$
inductively by letting $r_{j+1}=r_j^\#$ which leads to $r_J=r+\O(Jr^{-\frac{1}{4N}})$.
Iterating Proposition \ref{pp:the-con} yields
\begin{multline*}
	\mathscr{A}^\lambda\bigl(2^{-j}R,r_j,C_0(r+1)\bigr)\\
	\le (1+C\varpi_j) \mathscr{A}^\lambda\bigl(2^{-(j+1)}R,r_{j+1},C_0(r+1)\bigr)
	+\varpi_j^{-C}\Bigl(1+\frac{R}{2^jr}\Bigr)^{-\theta}\quad\quad
\end{multline*}
with $\varpi_j=\varpi 2^{-(J-j)\theta/C_\circ}$ for some fixed large $C_\circ\gg C$ so that (see \cite[Section 9]{JLeeTAMS})
$$
\prod_{j=0}^{J}(1+C\varpi_j)\le e^{C\sum_{j=0}^J\varpi_j}\le 1+\widetilde{C}\varpi,\;\,\sum_{j=0}^{J}\varpi_j^{-C} 2^{-(J-j)\theta}\le \varpi^{-O(1)},
$$
where $\widetilde{C}$ is a universal constant.
Therefore, we arrive at
$$
\mathscr{A}^\lambda(R,r,C_0(r+1))\le (1+\widetilde{C}\varpi)\mathscr{A}^\lambda(2^{-J}R,r_J,C_0(r+1))+\varpi^{-O(1)}.
$$
Using Proposition \ref{pp:the noncon}, we are done by suitably adjusting the constant $C$.\qed

\section{End of the proof }\label{sect:pf-thm}

We start the induction by fixing a pair of red and blue waves $\mathscr{F}^\lambda\in\mathfrak{R}^\lambda_R,\, \mathscr{G}^\lambda\in\mathfrak{B}^\lambda_R$ with $\EEE(\mathscr{F}^\lambda)=\EEE(\mathscr{G}^\lambda)=1$. We are to show there is a universal constant $ C_*$ independent of $\mathscr{F}^\lambda$ and $\mathscr{G}^\lambda$ such that
$\|\mathscr{F}^\lambda \mathscr{G}^\lambda\|_{Z(Q^\lambda_R)}\le C_* \lambda^{1/q}$ holds for all $Q^\lambda_R$.
To this end, we will show there is a universal constant $\delta>0$ small depending only on $C_0$ and possibly some other structural constants, such that $$
\lambda^{-1/q}
\|\mathscr{F}^\lambda \mathscr{G}^\lambda\|_{Z(Q^\lambda_R)}\le (1-\delta)\overline{A}^\lambda(R)+2^{O(C_0)}
$$
holds for all $Q^\lambda_R$. Here $O(C_0)$ is a universal constant as well. Taking suprema with respect to $\mathscr{F}^\lambda,\mathscr{G}^\lambda$
satisfying the above conditions, we close the induction by  definition of $A^\lambda(R)$
and the monotonicity of $\overline{A}^\lambda(R)$ with respect to $R$ and $\lambda$.
\smallskip

Let $\mathscr{A}^{\lambda}$ be given by Definition \ref{def:AAA}. We first  prove the essential concentration of waves on the conic sets and then finish the proof of Theorem \ref{thm:main}.

\subsection{Essential concentration along conic regions}
The following property is for the use of the Kakeya compression property.

\begin{proposition}
	\label{pp_1}
	Let  $\mathscr{F}^\lambda\in\mathfrak{R}^\lambda_R$ and $\mathscr{G}^\lambda \in\mathfrak{B}^\lambda_R$  be the red and blue waves fixed at the beginning of this section with $\EEE(\mathscr{F}^\lambda)=\EEE(\mathscr{G}^\lambda)=1$.
	There exists a constant $C>0$ depending only on $n,\eps,q,s$ such that for any $R\in[ 2^{10C_0},\lambda]$ and $\delta\in(0,1/2)$,
	 if $Q^\lambda_R$ satisfies
	\begin{equation}
	\label{eq:crit}
  \lambda^{-1/q}	\|\mathscr{F}^\lambda\,\mathscr{G}^\lambda\|_{Z(Q^\lambda_R)}\ge \overline{A}^\lambda(R)/2,
	\end{equation}
	and we  let $r_\delta$ be the supremum of all radii $r\ge {C_0}$ such that
	\begin{equation}
	\label{eq:en-cr}
	\EEE_{r,C_0Q^\lambda_R}(\mathscr{F}^\lambda,\mathscr{G}^\lambda)\le 1-\delta
	\end{equation}
	holds and  let $r_\delta={C_0}$ if no such radius exists,
	then  there exists a cube $\widetilde{Q}^\lambda_{\overline{R}_\delta}$ 	of size  $\overline{R}_\delta\in[ 2^{C_0},R]$ and $z_\delta\in\R^{n+2}$  such that $\overline{R}_\delta^{1/2+4/N}\le r_\delta$ when $r_\delta\ge 2^{2C_0}$, and  we have
	\begin{equation}
	\label{eq:pp-1}
	\|\mathscr{F}^\lambda \mathscr{G}^\lambda\|_{Z(Q_R^{\lambda})}\le (1-C(\delta+C_0^{-C})^{q})^{-2/q}\|\mathscr{F}^\lambda \mathscr{G}^\lambda\|_{Z(\Omega_{\delta}^\lambda)}+\lambda^{1/q}\;2^{CC_0}\,,
	\end{equation}
	where
	$\Omega_{\delta}^\lambda:=\widetilde{Q}_{\overline{R}_\delta}^{\lambda}\cap \mathbf{ \Lambda}^\lambda_\delta$ with  $\mathbf{\Lambda}^\lambda_\delta=\mathcal{C}^\lambda(z_\delta,\,C_0(r_\delta+1))$.
\end{proposition}

The proof is divided into two cases: $r_\delta\ge R^{1/2+4/N}$ and $r_\delta\le R^{1/2+4/N}$ which are treated in the following two subsubsections.

\subsubsection{The medium or low concentration case: $r_\delta\ge R^{1/2+4/N}$}
In this case, we show there is a constant $C$ such that for some $z_\delta\in\R^{n+2}$, we have
\begin{equation}
\label{eq:low-med}
\bigl\|\mathscr{F}^\lambda \mathscr{G}^\lambda\bigr\|_{Z(Q_R^{\lambda})}
\le \bigl(1-C(\delta+C_0^{-C})^q\bigr)^{-1/q}\bigl\|\mathscr{F}^\lambda \mathscr{G}^\lambda\bigr\|_{Z(Q_R^{\lambda}\cap \mathbf{ \Lambda}^\lambda_\delta)}
\end{equation}
holds with $C$ independent of $Q^\lambda_R$ and $z_\delta$. In particular, we have in this case $\overline{R}_\delta=R$ and $\widetilde{Q}^\lambda_{\overline{R}_\delta}=Q^\lambda_R$.\medskip

By definition, there is $D_\delta=D(z_\delta,r_\delta)$ with  $z_\delta=(\xxx_0,t_0)$ and $t_0\in\mathsf{LS}(C_0Q_R^\lambda)$, such that we have
\begin{equation}
\label{eq:sup}
\min\Bigl(\|\mathscr{F}^\lambda\|_{L^2(D_\delta)}^2, \|\mathscr{G}^\lambda\|_{L^2(D_\delta)}^2\Bigr)\ge 1-2\delta\,.
\end{equation}
Let $D^\natural=C_0^{1/2}D_\delta=D(z_\delta,\,C_0^{1/2}r_\delta)$ and write
$$
\mathscr{F}^\lambda=P_{D^\natural}\mathscr{F}^\lambda+(1-P_{D^\natural})\mathscr{F}^\lambda,\;\;
\mathscr{G}^\lambda=P_{D^\natural}\mathscr{G}^\lambda+(1-P_{D^\natural})\mathscr{G}^\lambda.
$$
By using Lemma \ref{lem:Te} and the  condition $1\le \overline{A}^\lambda(R)\le 2\,\lambda^{-\frac{1}{q}}\|\mathscr{F}^\lambda \mathscr{G}^\lambda\|_{Z(Q^{\lambda}_R)}$, it suffices to show for some universal constant $C>0$, we have
\begin{equation}
\label{eq:extr-1}
\|(P_{D^\natural}\mathscr{F}^\lambda ) \,\mathscr{G}^\lambda\|_{Z(Q_R^{\lambda}\setminus\mathbf{ \Lambda}^\lambda_\delta)}\lesssim C_0^{-C}\lambda^{1/q},
\end{equation}
\begin{equation}
\label{eq:extr-2}
\|(1-P_{D^\natural})\mathscr{F}^\lambda\, P_{D^\natural}\mathscr{G}^\lambda\|_{Z(Q_R^{\lambda}\setminus \mathbf{ \Lambda}^\lambda_\delta)}\lesssim C_0^{-C}\lambda^{1/q},
\end{equation}
and
\begin{equation}
\label{eq:extr-3}
\|(1-P_{D^\natural})\mathscr{F}^\lambda \, (1-P_{D^\natural})\mathscr{G}^\lambda\|_{Z(Q_R^{\lambda})}\lesssim (\delta+C_0^{-C})\;\lambda^{\frac{1}{q}}\;\overline{A}^\lambda(R).
\end{equation}

The proofs of \eqref{eq:extr-1} and \eqref{eq:extr-2} are the same. To see that these estimates are true, by using the energy estimate, we are reduced to
\begin{equation}
\label{eq:PP-huygens}
\|P_{D^\natural}\mathscr{F}^\lambda\|_{L^\infty(Q_R^{\lambda}\setminus \mathbf{ \Lambda}^\lambda_{\delta})}\lesssim_N R^{-N/2},\, \|P_{D^\natural}\mathscr{G}^\lambda\|_{L^\infty(Q_R^{\lambda}\setminus \mathbf{ \Lambda}^\lambda_{\delta})}\lesssim_N R^{-N/2},
\end{equation}
which is obvious in view of Lemma \ref{lem:conctr} and  $r_\delta\ge R^{1/2+4/N}$.

To show \eqref{eq:extr-3}, we use the induction argument. Note that by using \eqref{eq:en-cr}, \eqref{2.19}, \eqref{2.20} and the assumption on $r_\delta$, we have
$$
\EEE\bigl((1-P_{D^\natural})\mathscr{F}^\lambda\bigr)\lesssim \delta+R^{-N/2},\; \;\EEE\bigl((1-P_{D^\natural})\mathscr{G}^\lambda\bigr)\lesssim \delta+R^{-N/2}.
$$
It is easy to verify that we have $(1-P_{D^\natural})\mathscr{F}^\lambda\in\mathfrak{R}^\lambda_{R'}$ and $(1-P_{D^\natural})\mathscr{G}^\lambda\in\mathfrak{B}^\lambda_{R'}$ with $R'=\frac{R}{(1+o(1))}$. In fact, $R'=R (1+\frac{50}{2^{5C_0}})^{-N}$ will do the job. This yields \eqref{eq:extr-3} by finitely partitioning $Q_R^\lambda$ and using the definition of $A^\lambda(R')$ and the monotonicity of $\overline{A}^\lambda(R)$.
The proof is complete for this case.

\subsubsection{The high concentration case: $r_\delta\le  R^{1/2+4/N}$}
We turn to the case where the blue and red waves are highly concentrated.  Define $$\overline{R}_\delta=\max\Bigl(2^{2C_0},\;r_\delta^{1/(1/2+4/N)}\Bigr).$$

Consider the case $\overline{R}_\delta>2^{2C_0}$. In this case, we necessarily have $r_\delta>2^{C_0}$ and there is $z_\delta$ such that we have \eqref{eq:sup}.
Let $\widetilde{Q}=Q^{z_\delta,\lambda}_{\overline{R}_\delta}$ be the $\lambda-$stretched cube of size $\overline{R}_\delta$ centered at $z_\delta$.
By  splitting $Q_R^\lambda=\bigl(Q^\lambda_R\cap \widetilde{Q}\bigr)\cup\bigl(Q^\lambda_R\setminus \widetilde{Q}\bigr)$ and using Lemma \ref{lem:Te}, we have
$$
\bigl\|\mathscr{F}^\lambda \mathscr{G}^\lambda\bigr\|_{Z(Q_R^{\lambda})}^q
\le \bigl\|\mathscr{F}^\lambda \mathscr{G}^\lambda\bigr\|_{Z(\widetilde{Q})}^q
+\bigl\|\mathscr{F}^\lambda \mathscr{G}^\lambda\bigr\|_{Z(Q_R^{\lambda}\setminus \widetilde{Q})}^q.
$$

For the first term on $\widetilde{Q}$, the argument as in the medium or low concentration case leads to an estimate of the form \eqref{eq:low-med} with $Q^\lambda_R$ there replaced by $\widetilde{Q}$.

For the second term,
write as before
$$
\mathscr{F}^\lambda \mathscr{G}^\lambda=\underbrace{\bigl((P_{D^\natural}\mathscr{F}^\lambda)\mathscr{G}^\lambda\bigr)}_{:=I}
+\underbrace{\bigl((1-P_{D^\natural})\mathscr{F}^\lambda \; P_{D^\natural}\mathscr{G}^\lambda\bigr)}_{:=II}+
\underbrace{\bigl((1-P_{D^\natural})\mathscr{F}^\lambda\,(1-P_{D^\natural})\mathscr{G}^\lambda\bigr)}_{:=III}
$$
For $I$ and $II$, dyadic decomposing $Q^\lambda_R\setminus \widetilde{Q}$
into annuli around $z_\delta$ of the form $Q^{\mathsf{ann}}(z_\delta; 2^{j},2^{j+1})$ with $2^j\gtrsim \overline{R}_\delta$. Taking $C_0$ large and applying Lemma \ref{lem:non-edpt} then summing over dyadic $2^{-jb}$, we are done.

It remains to handle the $III-$term.
Denote
$$
\mathring{\mathscr{F}}^\lambda=(1-P_{D^\natural})\mathscr{F}^\lambda,\quad \mathring{\mathscr{G}}^\lambda=(1-P_{D^\natural})\mathscr{G}^\lambda\,.
$$
Note that $\mathring{\mathscr{F}}^\lambda, \mathring{\mathscr{G}}^\lambda$ are red and blue waves without the relaxed margin conditions required in $\mathfrak{R}^\lambda_R$ and $\mathfrak{B}^\lambda_R$. Thus, we can not apply the inductive argument as in the case when $r_\delta\ge R^{1/2+4/N}$. However, we may use the smallness of the exterior energy of $\mathring{\mathscr{F}}^\lambda, \mathring{\mathscr{G}}^\lambda$ from the definition of $r_\delta$ and a non-optimal estimate in terms of $A^\lambda(R)$. To this end, one needs to apply Galilean  transforms sending the $\xi-$variables to a neighborhood of the orgin and a mild scaling so that by modifying the input functions, they meets the required  margin conditions. To deal with the mixed-norm where Galilean transform does not directly apply, one needs to apply the duality argument and also covering $C_0Q^\lambda_R$ by a larger cube of the same shape
$CC_0Q^\lambda_R$  for some fixed large $C$, then partition it into $\O(1)$ many cubes of size $R$, so that we can apply the definiton of $A^\lambda(R)$.
Taking the inverse transform and affording a fixed universal constant, we have
\begin{equation}
\label{eq:ext-en-ind}
\|III\|_{Z(Q^\lambda_R)}\lesssim\;\delta \lambda^{1/q} \overline{A}^\lambda(R)+\lambda^{1/q} 2^{O(C_0)}
\end{equation}
Thus, by using the \eqref{eq:crit} condition
$$\|III\|_{Z(Q^\lambda_R)}\lesssim\;\delta \|\mathscr{F}^\lambda \mathscr{G}^\lambda\|_{Z(Q^\lambda_R)}+\lambda^{1/q} 2^{O(C_0)},$$
Plugging this back we are done.

During this process when using Galilean transform, the phase function is somehow distorted, however, this obstacle can be overcome by using the same argument as done for the proof of Lemma \ref{lem:wp-d}, so that the error term of the phase function can be handled using the standard trick  in \cite{TV-1,TV-2},  by using Taylor expansions  switching to the discretized version, applying the inductive argument based on $A^\lambda(R)$  and then summing over the absolutely convergent series. This is a tedious but very standard procedure, we refer to \cite[Section 5]{TV-2}, or Appendex \ref{sec:ext-en}  for an outline of the argument. It is because of this term $\|III\|_{Z(Q^\lambda_R)}$ that we need the condition $R\le \lambda$.\medskip

It remains to consider the case $\overline{R}_\delta=2^{2C_0}$. In this case, the energy is concentrated in a scale $\le 2^{C_0}$, we use the same argument as above using the trivial energy estimate for $\|\mathscr{F}^\lambda\mathscr{G}^\lambda\|_{Z(\widetilde{Q})}\lesssim \lambda^{1/q}2^{O(C_0)}$.\medskip

Collecting all these estimates, we obtain \eqref{eq:pp-1} and the proof of Proposition \ref{pp_1} is complete. \qed

\subsection{Proof of Theorem \ref{thm:main}}

 We are ready to show Theorem \ref{thm:main}. Let $C_1$ and $C_2$ be the structural constants given by Proposition \ref{pp_1} and Proposition  \ref{pp_2} respectively.  We may take $C_1$ large so that $C_1\ge 10000n$. Next, we take  $\delta=C_0^{-C_1/100}$.

 Let $\mathscr{F}^\lambda\in\mathfrak{R}^\lambda_R$,  $\mathscr{G}^\lambda \in\mathfrak{B}^\lambda_R$.
  For any $\lambda-$stretched spacetime cube $Q^\lambda_R$, if it satisfies the condition \eqref{eq:crit}, we let  $z_\delta, r_\delta,D_\delta$ be given by Proposition \ref{pp_1}. If $\overline{R}_\delta>2^{C_0}$, then by using  the definition of $\mathscr{A}^{\lambda}(\overline{R}_\delta,r_\delta, C_0(r_\delta+1))$ and \eqref{eq:pp-1}, we get
  \begin{multline*}
  	\lambda^{-1/q}\|\mathscr{F}^\lambda \mathscr{G}^\lambda \|_{Z(Q^{\lambda}_R)}\\
  	\le \bigl(1-C_1(\delta+C_0^{-C_1})^{q}\bigr)^{-2/q}
  	\mathscr{A}^{\lambda}(\overline{R}_\delta,r_\delta,C_0(r_\delta+1))\, \EEE_{r_\delta,C_0Q^\lambda_R}(\mathscr{F}^\lambda,\mathscr{G}^\lambda)^{1/q'}+2^{O(C_0)},
  	\end{multline*}
which entails by Proposition \ref{pp_2} and  \eqref{eq:en-cr}
\begin{multline*}
\lambda^{-1/q}\|\mathscr{F}^\lambda  \mathscr{G}^\lambda\|_{Z(Q^{\lambda}_R)}\\
\le \bigl(1-C_1(\delta+C_0^{-C_1})^{q}\bigr)^{-2/q} (1-\delta)^{1/q'}\Bigl((1+C_22^{-C_0})\overline{A}^{\lambda}(R)+2^{C_2C_0}\Bigr)+2^{O(C_0)}.	
	\end{multline*}
Using  $q>1$,  and taking $C_0$ large if necessary (depending only on $q, C_1$), one  has  $\exists\;\delta_\blacktriangle\in (0,1/10)$  and  $0<C_\blacktriangle<\infty$, depending only on $C_0,C_1,C_2$ and $q$, such that
$$
\lambda^{-1/q}
\|\mathscr{F}^\lambda \mathscr{G}^\lambda\|_{Z(Q_R^{\lambda})}\le (1-\delta_\blacktriangle)\,\overline{A}^\lambda(R)+C_\blacktriangle.
$$

If $\overline{R}_\delta=2^{C_0}$, then we have
the trivial estimate $\|\mathscr{F}^\lambda\mathscr{G}^\lambda\|_{Z({Q}^\lambda_R)}\le \lambda^{1/q}\, C_\blacktriangle$ by recalling the proof in the last subsection.

Thus, we have
\begin{multline*}
\max_{Q^\lambda_R}\lambda^{-1/q}
\|\mathscr{F}^\lambda \mathscr{G}^\lambda\|_{Z(Q_R^{\lambda})}\\
\le
\max\Bigl\{\max_{Q^\lambda_R: \eqref{eq:crit} \text{holds}}\lambda^{-1/q}
\|\mathscr{F}^\lambda \mathscr{G}^\lambda\|_{Z(Q_R^{\lambda})}, \max_{Q^\lambda_R: \eqref{eq:crit}\text{ fails}}\lambda^{-1/q}
\|\mathscr{F}^\lambda \mathscr{G}^\lambda\|_{Z(Q_R^{\lambda})}\Bigr\}+C_\blacktriangle\\
\le\max\Bigl\{(1-\delta_\blacktriangle)\overline{A}^\lambda(R)+{C}_\blacktriangle,\;\frac{1}{2}\,\overline{A}^\lambda(R)\Bigr\}+C_\blacktriangle\\
\le (1-\delta_\blacktriangle)\overline{A}^\lambda(R)+2{C}_\blacktriangle\,\quad.\qquad\qquad\,
\end{multline*}

Since the right side is independent of $\mathscr{F}^\lambda\in\mathfrak{R}^\lambda_R$ and $\mathscr{G}^\lambda\in \mathfrak{B}^\lambda_R$,
we  get
$$
A^\lambda(R)\le (1-\delta_\blacktriangle)\overline{A}^\lambda(R)+2{C}_\blacktriangle.
$$
Taking suprema, we obtain
$$
\overline{A}^\lambda(R)\le 2 \delta_\blacktriangle^{-1}{C}_\blacktriangle.
$$
\smallskip

Finally, fix $\eta\in \mathcal{S}(\R)$ with $\supp \;\eta\subset[-1,1]$ so that for any $f_1$ and $f_2$ being test functions supported in $V_1$ and $V_2$,   if we let $\tilde{f}_j(\xi,s)=f_j(\xi)\eta(s)$, then for $R\ge 2^{100 C_0}$
$$
F_j(x,x_{n+1},t):=\iint e^{2\pi i (x\cdot\xi+x_{n+1}s-\frac{t}{2}\frac{|\xi|^2}{R+s})}\tilde{f}_j(\xi,s)d\xi ds,
$$
satisfies the conditions in $\mathfrak{R}^{R}_R$ and $\mathfrak{B}^{R}_R$ for $j=1,2$ respectively.\medskip

Applying the uniform estimate on $A^{R}(R)$ to $F_1$ and $F_2$, we get
$$\|F_1F_2\|_{Z(Q_R^{R})}\lesssim_{\eta}R^{1/q} \|f_1\|_2\|f_2\|_2.$$
Changing variables $t\to  R\,t$  and letting $R\to +\infty$, we get \eqref{eq:B-parab} by using Lebesgue's dominated convergence  and then Fatou's theorem followed with integrating  $x_{n+1}$ out, c.f. \cite{TaoMZ}.
The proof is complete.\qed

\appendix

\section{On the locally constant property for the plate maximal function: proof of \eqref{eq:max-plate}}\label{sec:max}
We follow the standard argument in \cite{TaoGFA,MA}.

If $r\ge C $,  noting that
$
(\nu,k)+\mathcal{R}^{\varpi,\varrho}_r\subset \xxx+C\mathcal{R}^{\varpi,\varrho}_{Cr}
$ whenever $\xxx\in (\nu,k)+C \mathcal{R}^{\varpi,\varrho}_1$, the average of $f_\mu $ on $(\nu,k)+\mathcal{R}_r^{\varpi,\varrho}$ is bounded by  $\lesssim_C\mathcal{M}^{\varpi,\varrho}f_\mu(\xxx)$.

Next, we assume $0<r\le C$. Note that the Fourier transform of $f_\mu$
is supported in the set $\{(\xi,s); \,|\xi-\mu|\le 10n\varrho^{-1},\;|s|\le 10\}$. Let $p(\xi,s)$ be as in the proof of Lemma \ref{lem:wp-d}. Then,
using the reproducing formula $\wh{f}_\mu(\xi,s)=p(\varrho(\xi-\mu),s)\wh{f}_\mu(\xi,s)$ and taking inverse Fourier transform, we have
$f_\mu=\mathbf{\Psi}_{\mu,\varrho}*f_\mu$ where
$$
\mathbf{\Psi}_{\mu,\varrho}(\xxx)=\varrho^{-n}p^\vee(\varrho^{-1}x,x_{n+1})e^{2\pi ix\cdot\mu},\quad \xxx=(x,x_{n+1}).
$$
By Minkowski inequality, the average of $f_\mu$
over $(\nu,k)+\mathcal{R}^{\varpi,\varrho}_r$ is bounded by
\begin{equation}
\label{eq:OPJ}
(\varpi^{-2}\varrho)^{-n}r^{-(n+1)}
\int |f_\mu(\xxx')| \int_{(\nu,k)+\mathcal{R}^{\varpi,\varrho}_r}\Bigl|\mathbf{ \Psi}_{\mu,\varrho}(\xxx''-\xxx')\Bigr|d\xxx'' d\xxx'.
\end{equation}
Splitting
$$
\R^{n+1}_{\xxx'}=\bigl(\xxx+\mathcal{R}^{\varpi,\varrho}_1\bigr)\cup\bigcup_{k\ge 1}\Bigl(\xxx+\mathcal{R}^{\varpi,\varrho}_{2^k }\setminus \mathcal{R}^{\varpi,\varrho}_{2^{k-1}}\Bigr),
$$
we have $\eqref{eq:OPJ}\le \sum_{k\ge 0}I_k$, where
$$
I_0=(\varpi^{-2}\varrho)^{-n}r^{-(n+1)}
\int_{\xxx+\mathcal{R}^{\varpi,\varrho}_1} |f_\mu(\xxx')| \int_{(\nu,k)+\mathcal{R}^{\varpi,\varrho}_r}\Bigl|\mathbf{ \Psi}_{\mu,\varrho}(\xxx''-\xxx')\Bigr|d\xxx'' d\xxx',
$$
and for $k\ge 1$
$$
I_k=(\varpi^{-2}\varrho)^{-n}r^{-(n+1)}
\int_{\xxx+\mathcal{R}^{\varpi,\varrho}_{2^k }\setminus \mathcal{R}^{\varpi,\varrho}_{2^{k-1}}} |f_\mu(\xxx')| \int_{(\nu,k)+\mathcal{R}^{\varpi,\varrho}_r}\Bigl|\mathbf{ \Psi}_{\mu,\varrho}(\xxx''-\xxx')\Bigr|d\xxx'' d\xxx'\,.
$$

For $\sum_{0\le k\le \varpi^{-O(1)}} I_k$, we use the boundedness of $p^\vee$  to get
$$
\sum_{0\le k\le \varpi^{-O(1)}} I_k
\lesssim \varpi^{-O(1)}\mathcal{M}^{\varpi,\varrho}f_\mu(\xxx),
$$

For $\sum_{k\ge \varpi^{-O(1)}}I_k$.
By using the elementary identity of sets
$$(A_1\times B_1)\setminus (A_2\times B_2)=\bigl[(A_1\setminus A_2)\times B_1\bigr]\cup
\bigl[(A_1\cap A_2)\times(B_1\setminus B_2)\bigr],$$
and $r\le C$, we have  that for any $k\ge \varpi^{-O(1)}$, consider $\xxx''\in(\nu,k)+\mathcal{R}^{\varpi,\varrho}_r$ and
$\xxx'\in \xxx+\mathcal{R}^{\varpi,\varrho}_{2^k }\setminus \mathcal{R}^{\varpi,\varrho}_{2^{k-1}}$,
then either
\begin{multline*}
\varrho^{-1}|x''-x'|\ge \varrho^{-1}|x-x'|-\varrho^{-1}|x-\nu|-\varrho^{-1}|x''-\nu|\\
\ge 2^{k-1}\varpi^{-2}-C\varpi^{-2}-Cr\varpi^{-2}\ge \varpi^{-2}(2^{\varpi^{-O(1)}-1}-2(C+1)^2)\ge C \varpi^{-O(1)}
\end{multline*}
or
\begin{multline*}
|x_{n+1}''-x_{n+1}'|\ge |x_{n+1}-x'_{n+1}|-|x_{n+1}-k|-|x''_{n+1}-k|\\
\ge 2^{k-1}-C-C^2\ge \varpi^{-O(1)}-2(C+1)^2\ge C
\end{multline*}
by taking $C_0$ large if necessary.
Thus, by refining the above two estimates replacing the lower bound with $\gtrsim 2^k$, we have
$$\varrho^{-1}|x''-x'|+|x_{n+1}''-x_{n+1}'|\gtrsim 2^k,\quad\forall\;k\ge \varpi^{-O(1)}.$$
Hence
$$
I_k\lesssim_M 2^{-kM}
\varrho^{-n}
\int_{\xxx+\mathcal{R}^{\varpi,\varrho}_{2^k }\setminus \mathcal{R}^{\varpi,\varrho}_{2^{k-1}}} |f_\mu(\xxx')| \;d\xxx'\lesssim \varpi^{-O(1)} 2^{-k}\mathcal{M}^{\varpi,\varrho}f_\mu(\xxx)\,.
$$
Summing up $k\ge \varpi^{-O(1)}$, we are done.
\smallskip
Combining the above two cases and taking suprema, we complete the proof of \eqref{eq:max-plate}.

\section{The exterior energy estimate}\label{sec:ext-en}
We outline the  proof for the exterior energy induction \eqref{eq:ext-en-ind}, which one may compare with that for the formula (48) P. 239 of \cite{TaoMZ}.
By translation invariance, we take $t_{Q^\lambda_R}=0$.

Let $\mathfrak{r}:=(C_0^{1/2}r_\delta)^{-(1-\frac{1}{N})}+20 R^{-1}.$
By direct computation and $\lambda\ge R$, we have
$$
\mathsf{marg}(\mathring{\mathscr{F}}^\lambda)\ge \mathsf{marg}(\mathscr{F}^\lambda)-\mathfrak{r},\quad
\mathsf{marg}(\mathring{\mathscr{G}}^\lambda)\ge \mathsf{marg}(\mathscr{G}^\lambda)-\mathfrak{r}.
$$
Let  $
\mathfrak{r}'=
(200n)^{-1}+2R^{-\frac{1}{N}}+\mathfrak{r}.
$ and define
$$\mathfrak{O}_1=\Bigl\{\xi\in\R^n: \bigl|\xi-e_1\bigr|\le \mathfrak{r}' \Bigr\},\quad 
\mathfrak{O}_2=\Bigl\{\xi\in\R^n: |\xi|\le \mathfrak{r}'\Bigr\}\,.$$
Denote $\mathfrak{d}=2(1+R^{-\frac{1}{N}})-(100n)^{-1}$ and $\mathfrak{d}'=\mathfrak{d}+\mathfrak{r}$.
Then, we have
\begin{align*}
	\supp(\widehat{\mathring{\mathscr{F}}^\lambda})
	&\subset \Bigl\{(\xi,s,\tau)\,:\,\tau=-\frac{|\xi|^2}{2(\lambda+s)},\,\xi\in\mathfrak{O}_1,\,|s|\le \mathfrak{d}'\Bigr\},\\
	\supp(\widehat{\mathring{\mathscr{G}}^\lambda})&
	\subset \Bigl\{(\xi,s,\tau)\,:\,\tau=-\frac{|\xi|^2}{2(\lambda+s)},\,\xi\in\mathfrak{O}_2,\,|s|\le \mathfrak{d}'\Bigr\}\,.
\end{align*}
Let $\mathfrak{O}_1^\flat=\Bigl\{\xi\in\R^n:|\xi-e_1|\le\mathfrak{r}'-2\mathfrak{r}\Bigr\}$ and $\Gamma$ be the conic set such that $\mathfrak{O}_1^\flat\subset\Gamma$ with the boundary of $\mathfrak{O}_1^\flat$ being tangent to that of $\Gamma$.
Decomposing $\mathfrak{O}_1=\mathfrak{O}_1'\cup\mathfrak{O}''_1$
with $\mathfrak{O}_1'=\mathfrak{O}_1\cap\Gamma \cap\{\xi:|\xi|\ge 1\} $
and $\mathfrak{O}_1''=\mathfrak{O}_1\setminus\mathfrak{O}_1'$,
we may write correspondingly 
$$
\mathring{\mathscr{F}}^\lambda \mathring{\mathscr{G}}^\lambda=\mathring{\mathscr{F}}^{\lambda,'}\mathring{\mathscr{G}}^\lambda+\mathring{\mathscr{F}}^{\lambda,''}\mathring{\mathscr{G}}^\lambda,
$$
where $\,\widehat{{\mathring{\mathscr{F}}^{\lambda,'}}}$ is supported in the lift of $\mathfrak{O}_1'\times [-\mathfrak{d}',\mathfrak{d}']$ to $\varSigma^\lambda$ and likewise for $\mathring{\mathscr{F}}^{\lambda,''}$.

To handle the first term, we  let
$$\varkappa_1=\frac{1+\mathfrak{r}'-\mathfrak{r}}{1+\mathfrak{r}'},\quad\;\varkappa_2=\frac{\mathfrak{d}}{\mathfrak{d}'}\,,$$
and make change of variables $(\xi,s)\to \bigl(\varkappa_1^{-1}\xi, \varkappa_2^{-1}\,s\bigr)$ for $\mathring{\mathscr{F}}^{\lambda,'}\mathring{\mathscr{G}}^\lambda$ to meet the margin condition in the new variables after modifying the initial data. The result follows from the inductive hypothesis.

To handle the second term, we decompose $\mathfrak{O}_1''=\overline{\mathfrak{O}''}_1\cup \overline{\overline{\mathfrak{O}''}}_1$ where $\overline{\mathfrak{O}''}_1=\mathfrak{O}''_1\setminus\Gamma$ and $\overline{\overline{\mathfrak{O}''}}_1=\mathfrak{O}_1''\setminus\overline{\mathfrak{O}''_1}$. Write correspondingly
$$
\mathring{\mathscr{F}}^{\lambda,''}\mathring{\mathscr{G}}^\lambda=\mathring{\mathscr{F}}^{\overline{\lambda,''}}\mathring{\mathscr{G}}^\lambda+\mathring{\mathscr{F}}^{\overline{\overline{\lambda,''}}}\mathring{\mathscr{G}}^\lambda.
$$
For the first term, we partition $\overline{\mathfrak{O}''_1}$ into  the union  of $\mathcal{O}(n)$ many sectors $\Delta$, \emph{i.e.} $\overline{\mathfrak{O}''}_1=\cup_\Delta \Delta$ and   write 
$$\mathring{\mathscr{F}}^{\overline{\lambda,''}}\mathring{\mathscr{G}}^\lambda=\sum_\Delta \mathring{\mathscr{F}}^{\overline{\lambda,''}}_\Delta\mathring{\mathscr{G}}^\lambda.$$
For each $\Delta$,  we rotate $\Delta$ to $\Delta'$ such that $\Delta'$ is centered at $e_1$ so that after doing this rotation and changing variable $s\to 
\varkappa_2^{-1}s$, we have  $\mathring{\mathscr{F}}^{\overline{\lambda,''}}_{\Delta'}\mathring{\mathscr{G}}^\lambda$ fulfills the margin condition and obtain  the result by using the inductive hypothesis.

To deal with the $\mathring{\mathscr{F}}^{\overline{\overline{\lambda,''}}}\mathring{\mathscr{G}}^\lambda$ term, we decompose futher $\mathfrak{O}_2=\mathfrak{O}_2'\cup\mathfrak{O}_2''$ where $\mathfrak{O}_2'=\{\xi: |\xi|\le \mathfrak{r}'-2\mathfrak{r}\}$, $\mathfrak{O}_2''=\mathfrak{O}_2\setminus\mathfrak{O}_2'$. Write correspondingly, 
$$
\mathring{\mathscr{F}}^{\overline{\overline{\lambda,''}}}\mathring{\mathscr{G}}^\lambda=
\mathring{\mathscr{F}}^{\overline{\overline{\lambda,''}}}\mathring{\mathscr{G}}^{\lambda,'}+\mathring{\mathscr{F}}^{\overline{\overline{\lambda,''}}}\mathring{\mathscr{G}}^{\lambda,''}.
$$
For $\mathring{\mathscr{F}}^{\overline{\overline{\lambda,''}}}\mathring{\mathscr{G}}^{\lambda,'}$, we make angular partition as above for $\overline{\overline{\mathfrak{O}''}}_1=\cup_{S''}S''$, and write 
$\mathring{\mathscr{F}}^{\overline{\overline{\lambda,''}}}\mathring{\mathscr{G}}^{\lambda,'}=\sum_{S''}\mathring{\mathscr{F}}^{\lambda,S''}\mathring{\mathscr{G}}^{\lambda,'}$. Note that for each $S''$, we first rotate $S''$ to be centered in the $e_1$ direction and then changing variables $(\xi,s)\to (\varkappa_1' \xi,\varkappa_2^{-1} s)$
with $0<\varkappa_1'<1$ depending only on $\varkappa_1$ and $\mathfrak{r}'$
so that we recover the margin condition in order to use the inductive hypothesis. 

In remains to handle $\mathring{\mathscr{F}}^{\overline{\overline{\lambda,''}}}\mathring{\mathscr{G}}^{\lambda,''}$. We make angular partition for $\overline{\overline{\mathfrak{O}''}}_1=\cup_{S''}S''$ as above and decompose $\mathfrak{O}''_2=\cup_{\mathfrak{D}''}\mathfrak{D}''$ into $\O(n)$ many pieces $\mathfrak{D}''$.  We write correspondingly
$$\mathring{\mathscr{F}}^{\overline{\overline{\lambda,''}}}\mathring{\mathscr{G}}^{\lambda,''}=\sum_{S''}\sum_{\mathfrak{D}''}\mathring{\mathscr{F}}^{\lambda,S''}\mathring{\mathscr{G}}^{\lambda,\mathfrak{D}''}.$$
For each $\mathfrak{D}''$ and $S''$, we first translate $\mathfrak{D}''$ to be centered at the origin,  and  $S''$ is accordingly translated to $\tilde{S}''$. We then rotate $\tilde{S}''$ to the $e_1$ direction and apply the mild scaling as above to recover the margin condition so that the result follows from the induction.  

Unlike all the other above cases, we need to translate the center of $\mathfrak{D}''$ to the origin which makes the argument technically involved. To see this, we let $\xi^{\mathfrak{D}''}$ be the center of $\mathfrak{D}''$. After changing variables $\xi\to \xi^{\mathfrak{D}''}+\xi$, we may write 
\begin{multline*}
\mathring{\mathscr{F}}^{\lambda,S''}(\xxx,t)\mathring{\mathscr{G}}^{\lambda,\mathfrak{D}''}(\xxx,t)
=e^{4\pi i\bigl(x\cdot \xi^{\mathfrak{D}''}-\frac{t}{2\lambda}|\xi^{\mathfrak{D}''}|^2\bigr)}\sum_{k_1,k_2\ge 0}\frac{(t/\lambda^2)^{k_1}}{k_1!}\frac{(t/\lambda^2)^{k_2}}{k_2!}\\
\times
S^\lambda(t)\Bigl[\mathring{f}^{\lambda,S''}_{\mathfrak{D}'',k_1}\Bigr]\bigl(x-t\lambda^{-1}\xi^{\mathfrak{D}''},x_{n+1}\bigr)\,
S^\lambda(t)\Bigl[\mathring{g}^{\lambda}_{\mathfrak{D}'',k_2}\Bigr]\bigl(x-t\lambda^{-1}\xi^{\mathfrak{D}''},x_{n+1}\bigr),
\end{multline*}
where
\begin{align*}
	\widehat{\mathring{f}^{\lambda,S''}_{\mathfrak{D}'',k_1}}(\xi,s)=&\bigl(\mathcal{E}^\lambda(\xi,s;\xi^{\mathfrak{D}''})\bigr)^{k_1}\widehat{\mathring{\mathscr{F}}^{\lambda,S''}}(\xi^{\mathfrak{D}''}+\xi,s,0),\\
		\widehat{\mathring{g}^{\lambda}_{\mathfrak{D}'',k_2}}(\xi,s)=&\bigl(\mathcal{E}^\lambda(\xi,s;\xi^{\mathfrak{D}''})\bigr)^{k_2}\widehat{\mathring{\mathscr{G}}^{\lambda,\mathfrak{D}''}}(\xi^{\mathfrak{D}''}+\xi,s,0),
\end{align*}
with $\mathcal{E}^\lambda\bigl(\xi,s\,;\,\xi^{\mathfrak{D}''}\bigr)=\frac{\lambda s}{\lambda+s}\bigl(\langle\xi,\xi^{\mathfrak{D}''}\rangle+\frac{|\xi^{\mathfrak{D}''}|^2}{2}\bigr)$. Noting that $|t|\lesssim \lambda R$, we may use Minkowski's inequality and then cover $Q^\lambda_R$ with an enlarged $CQ^\lambda_R$, changing variables $x\to x+\lambda^{-1}t\xi^{\mathfrak{D}''}$ to get 
$$
\|\mathring{\mathscr{F}}^{\lambda,S''}\mathring{\mathscr{G}}^{\lambda,\mathfrak{D}''}\|_{Z(Q^\lambda_R)}\lesssim \sum_{k_1,k_2\ge 0}\frac{1}{k_1!k_2!}\bigl\|S^\lambda\bigl[\mathring{f}^{\lambda,S''}_{\mathfrak{D}'',k_1}\bigr]\cdot S^\lambda\bigl[\mathring{g}^{\lambda}_{\mathfrak{D}'',k_2}\bigr]\bigr\|_{Z(CQ^\lambda_R)}
$$
where we have used $R\le \lambda$.

For each $k_1,k_2$, we have
$$
\EEE\Bigl(S^\lambda\bigl[\mathring{f}^{\lambda,S''}_{\mathfrak{D}'',k_1}\bigr]\Bigr)\lesssim\delta,\,\quad\EEE\Bigl(S^\lambda\bigl[\mathring{g}^{\lambda}_{\mathfrak{D}'',k_2}\bigr]\Bigr)\lesssim\delta\,.
$$ 
We first rotate the support of $	\widehat{\mathring{f}^{\lambda,S''}_{\mathfrak{D}'',k_1}}(\cdot,s)$ to the $e_1$ direction and then changing variables $(\xi,s)\to (\varkappa_1''\xi,\varkappa_2 s)$ for some appropriate $\varkappa_1''$ to recover the margin condition. The proof is complete by summing up $k_1,k_2$.

\end{document}